\documentclass[12pt]{amsart}
 
\usepackage{mathtools}
\mathtoolsset{showonlyrefs,showmanualtags}
 
\usepackage{amssymb, amsmath, amsthm, esint}
\usepackage{mathrsfs}
\usepackage{bbm, dsfont}
\usepackage{color}

\usepackage[top=1.5in, bottom=1.5in, left=1.25in, right=1.25in]{geometry}
\usepackage[all]{xy}
\usepackage{amscd}
 
\usepackage[bottom]{footmisc}
 
\usepackage{fancyhdr}
 
\AtBeginDocument{\setlength{\footskip}{30pt}}
\fancypagestyle{plain}{%
  \fancyhf{}%
  \fancyfoot[C]{\thepage}%
}
 
\fancypagestyle{myfancy}{%
  \fancyhf{}%
  \fancyhead[CE]{\scshape Double Recurrence: Primes and Weighted Theory}%
  \fancyhead[CO]{J.\ Fornal and B.\ Krause}%
  \fancyhead[LE,RO]{\thepage}%
}
 
\newcommand{\EE}{\mathbb{E}}

\newcommand{\TT}{\mathbb{T}}

\newcommand{\Z}{\mathbb{Z}}

\numberwithin{equation}{section}

\newtheorem{theorem}{Theorem}[section]
\newtheorem*{theorem*}{Theorem}
\newtheorem{definition}[equation]{Definition}
\newtheorem{proposition}[theorem]{Proposition}
\newtheorem*{proposition*}{Proposition}
\newtheorem{cor}[theorem]{Corollary}
\newtheorem{lemma}[theorem]{Lemma}

\theoremstyle{definition}
\newtheorem{problem}{Problem}
 
\usepackage[
  pagebackref,
  colorlinks=true,
  linkcolor=blue,
  citecolor=magenta,
  filecolor=magenta,
  urlcolor=cyan
]{hyperref}

\makeatletter
\renewcommand*{\backref}[1]{}
\renewcommand*{\backrefalt}[4]{%
  \ifcase #1\relax
  \or
    \space (Page~#2)%
  \else
    \space (Pages~#2)%
  \fi
}
\makeatother

\begin{document}
 
\title{Double Recurrence and Almost Sure Convergence: Primes and Weighted Theory}
 
\author{Jan Fornal}
\address{%
  Department of Mathematics, University of Bristol \\
  Beacon House, Queens Rd, Bristol BS8 1QU}
\email{nc24166@bristol.ac.uk}
 
\author{Ben Krause}
\address{%
  Department of Mathematics, University of Bristol \\
  Beacon House, Queens Rd, Bristol BS8 1QU}
\email{ben.krause@bristol.ac.uk}
 
\date{\today}
 
\begin{abstract}
  Let $(X,\mu)$ be a probability space equipped with an invertible,
  measure-preserving transformation $T\colon X \to X$. We exhibit a wide
  class of weights $w$ so that whenever $f,g \in L^{\infty}(X)$, the
  bilinear ergodic averages
  \begin{align}\label{e:introavg}
    \frac{1}{N} \sum_{n \leq N} w(n)\, T^{an}f \cdot T^{bn}g,
    \qquad a,b \in \mathbb{Z}
  \end{align}
  converge $\mu$-almost surely. This class encompasses the von
  Mangoldt function, resolving Problem~12 from Frantzikinakis' survey on
  open problems in ergodic theory, the divisor function, the sum-of-two-squares representation function, etc., as well as
  their restrictions to lower-density Piatetski-Shapiro sequences of the
  form $\{\lfloor k^{c}\rfloor : k \in \mathbb{N}\}$, $1 \leq c < 7/6$.
 
  Our methods combine combinatorial number theory and higher-order Fourier
  analysis with classical Fourier-analytic/martingale-based methods; the
  role of $U^{3}$ analysis is particularly significant.
\end{abstract}
 
\maketitle
 
\setcounter{footnote}{0} 
 
\pagestyle{myfancy}
 
\setcounter{tocdepth}{1}
\tableofcontents
 
\section{Introduction}
The topic of this paper is bilinear pointwise ergodic theory, a field of study initiated by Bourgain in \cite{B3}, in which he used a brilliant combination of classical Fourier analysis and martingale methods to establish the following theorem; here and below, by a measure-preserving system, we mean a probability space $(X,\mu)$, equipped with an invertible measure-preserving transformation, $T:X \to X$, so that 
\[ \mu(T^{-1} E) = \mu(E) \; \; \; \text{ for all measurable } E \subset X,\]
and for measurable functions, $f: X \to \mathbb{C}$, define $Tf(x) := f(Tx)$.
 
\begin{theorem}[Bourgain's Double Recurrence Theorem]\label{t:BDR}
Let $(X,\mu,T)$ be a measure-preserving system, and suppose that $T_1,T_2$ are powers of $T$.\footnote{Throughout, we allow our powers to be negative, i.e.\ to involve the transformation $T^{-1}$.} Then for $f,g \in L^\infty(X)$, the bilinear ergodic averages
\begin{align}\label{e:introbil}
\frac{1}{N} \sum_{n \leq N} T_1^n f \cdot T_2^n g
\end{align}
converge $\mu$-almost surely.
\end{theorem}
 
In the representative case where $T_i = T^i$, the averages \eqref{e:introbil} are a special case of the multilinear ergodic averages considered by Furstenberg in his ergodic-theoretic proof \cite{Furst} of Szemer\'{e}di's theorem \cite{Sz},
\begin{align}\label{e:intromul}
\frac{1}{N} \sum_{n \leq N} \prod_{k=1}^K T^{kn} f_k, \; \; \; f_k \in L^\infty(X),
\end{align}
whose norm convergence was established in celebrated work of Host-Kra \cite{HK} and Ziegler \cite{Z}. An important take-away from \cite{HK} is the role of higher-order Fourier analysis in analyzing the limiting behavior of \eqref{e:intromul}, which distinguishes the bilinear case from higher degrees of multi-linearity; their work connects to the Green-Tao Theorem \cite{GT0}, which can be derived from an analysis of the weighted averages
\begin{align}\label{e:intromulprime}
\frac{1}{N} \sum_{n \leq N}  \Lambda(n) \prod_{k=1}^K T^{kn} f_k, \; \; \; f_k \in L^\infty(X),
\end{align}
where $\Lambda$ is the von Mangoldt function,
\begin{align}\label{e:vonMangoldt}
    \Lambda(n) := \begin{cases} \log p & \text{ if } n = p^k \text{ is a power of a prime} \\
0 & \text{ otherwise}. \end{cases}
\end{align}
Indeed, for bounded sequences $\{ a_n \}$,
\begin{align}\label{e:sumbyparts}
    \frac{1}{N} \sum_{n \leq N} a_{p_n} \to L \iff \frac{1}{N} \sum_{n \leq N} \Lambda(n) a_n \to L
\end{align}
by summation by parts; here and throughout, we enumerate the primes
\begin{align}\label{e:primes} \{2 = p_1 < p_2 < \dots \}.
\end{align}
 
While norm convergence of \eqref{e:intromulprime} was established in \cite{WZ}, as far as pointwise convergence is concerned, only the case $K=1$ of \eqref{e:intromulprime} has been addressed, first by Bourgain in \cite{Bp}, with a subsequent (optimal) strengthening due to Wierdl \cite{Wierdl}, see \cite{LV}.
 
Accordingly, in his survey of open problems in ergodic theory \cite{Frant}, Frantzikinakis explicitly calls attention to the issue of pointwise convergence of \eqref{e:intromulprime} in the bilinear $K=2$ case.
 
The goal of this paper is to establish this convergence result as a special case of a broader phenomenon.
 
\begin{theorem}[Special Case]\label{t:prime}
    Let $(X,\mu,T)$ be a measure-preserving system and suppose $T_1,T_2$ are powers of $T$. Then for each $f,g \in L^\infty(X)$, the bilinear ergodic averages
    \begin{align}
        \frac{1}{N} \sum_{n \leq N} T_1^{p_n}   f \cdot T_2^{p_n} g
    \end{align}
    converge $\mu$-almost surely, see \eqref{e:primes}. In particular, the averages \eqref{e:intromulprime} converge almost surely in the case $K=2$.
\end{theorem}
 
More generally, our work establishes pointwise convergence for weighted bilinear ergodic averages, provided that -- informally -- the weights have combinatorial decompositions into type one/type two sums which are Gowers $U^{1+}$ uniform with (sufficiently large) logarithmic savings, with the key point being that \cite{LSS} allows one to upgrade $U^{1+}$ uniformity to $U^s$ uniformity, after subtracting a main term coming from major arc approximations; see the discussion in \S \ref{s:adm} below.
 
More rigorously, this paper will be concerned with essentially $\ell^1$-normalized weights $w: \mathbb{Z} \to \mathbb{C}$ that are $U^s$-approximable by exponential sums with sufficiently rapidly-decaying coefficients; we refer the reader to \S \ref{sss:O} to recall asymptotic notation and the definition of $U^s([N])$.
 
\begin{definition}
    A weight $w : \mathbb{Z} \to \mathbb{C}$ will be said to be \emph{admissible} if it satisfies the following properties:
    \begin{itemize}
    \item \textbf{Upper normalization:} There exists an absolute constant $C<\infty$ so that
\begin{align}
    \limsup_{\substack{|I|\to\infty\\ 0\in 10I}} \frac{1}{|I|} \sum_{n\in I}|w(n)| \leq C,
\end{align}
where $I$ ranges over intervals in $\mathbb Z$, and $10I$ denotes the interval concentric with $I$ whose length is $10|I|$;
    \item \textbf{Heath-Brown-type approximation}: There exists an absolute constant $C < \infty$, a function 
    \[ S := S_w : \mathbb{Q} \to \mathbb{C} \text{ so that } |S(a/q)| \leq C \cdot q^{o(1)-1},\]
    and a constant $\nu > 4$, so that for some $s \geq 3$ and all sufficiently large $N$, there exists a truncation parameter $M = M(N) \leq N^{o(1)}$ so that
    \begin{align}
        \| w(n) - \sum_{(a,q) = 1,\ q \leq M} S(a/q) e^{2 \pi i n a/q} \|_{U^s([N])} \leq C \cdot \log^{-\nu/2^s} N.
    \end{align}
\end{itemize}
\end{definition}

The admissibility criterion should be viewed as a bilinear pointwise analogue of the major/minor arc decompositions used in Bourgain's work pointwise ergodic theorems involving one function, \cite{B0,Bp,B1,B2}, see the discussion below. The ``major arc" model
\[ \sum_{q \leq M} \sum_{(a,q) =1} S(a/q) e^{2 \pi i na /q}\]
captures the structured Fourier contribution of $w$, while the $U^s$-small remainder is negligible for our bilinear averages after standard reductions. Indeed, in analogy with \cite{B0,Bp,B1,B2}, in practice, one expects $S = S_w$ to be recovered via
\begin{align}
    S(a/q) = \lim_N \frac{1}{N} \sum_{n \leq N} w(n) e^{-2 \pi i n a/q}.
\end{align}
And, while we emphasize that we do not assume any lower bound on the averages of $|w|$, many of the arguments below become trivial if $w$ is sufficiently sparse or degenerate, so the most interesting case is when
\begin{align}
    \limsup_{N \to \infty} \frac{1}{N} \sum_{n \leq N} |w(n)| > 0.
\end{align}

With this definition in mind, our main result is as follows.

\begin{theorem}\label{t:main}
    Suppose that $w$ is admissible. Then for any $(X,\mu,T)$ measure-preserving system, whenever $T_1, T_2$ are powers of $T$, for any $f,g \in L^\infty(X)$, the bilinear ergodic averages
    \begin{align}
        \frac{1}{N} \sum_{n \leq N} w(n) T_1^n f  \cdot T_2^n g
    \end{align}
    converge $\mu$-almost surely.
\end{theorem}

As we will show below, the following examples are admissible: 
\begin{itemize}
\item The von Mangoldt function $\Lambda(n)$;
    \item The divisor function 
    \begin{align}\label{e:divisor}
    \tau(n) := \sum_{d |n }1
    \end{align}
    (essentially, after re-normalization); and
    \item The sum of two-squares function,
    \begin{align}\label{e:sumofsquares} r_2(n) := |\{ (a,b) : a^2 + b^2 = n \}|.
    \end{align}
\end{itemize}

In particular, we establish the following.
\begin{cor}\label{c:main}
    Let $(X,\mu,T)$ be a measure-preserving system and suppose $T_1,T_2$ are powers of $T$. Then for each $f,g \in L^\infty(X)$, the bilinear ergodic averages
    \begin{align}
        \frac{1}{N \log N} \sum_{n \leq N} \tau(n) T_1^{n} f  \cdot T_2^{n} g
    \end{align}
    and
    \begin{align}
        \frac{1}{\pi N} \sum_{a^2 + b^2 \leq N} T_1^{a^2+b^2} f  \cdot T_2^{a^2 + b^2} g
    \end{align}
    converge $\mu$-almost surely.
\end{cor}

We remark that one can likely verify the admissibility of 
\[ |\{ (a,b) : a^2 + d b^2 = n\}| \]
for $d \geq 1$, as one may rely on the decomposition of the theta series,
\begin{equation}
    \Theta_d(q) = \sum_{x, y \in \mathbb{Z}} q^{x^2 + dy^2},
\end{equation}
into its Eisenstein part and cusp part. Using the Fourier expansion of the Eisenstein part, one expects to deal with it in the same manner as in the case of the sum of squares, see \S \ref{ss:sumofsquares} below; on the other hand, by the Rankin-Selberg estimates and the sparse support of the cusp coefficients, one can favorably estimate the $\ell^1$ norm of the first $N$ coefficients of the cusp form, and thereby discard this contribution by the triangle inequality.

In fact, the same convergence phenomenon persists when we restrict $w$ to so-called Piatetski-Shapiro sequences,
\begin{align}
    \mathbb{N}_c := \{ \lfloor k^c \rfloor : k \geq 1\},
\end{align}
provided $1 \leq c < 7/6$ is sufficiently close to $1$, and the $U^3([N])$ norms of $w$ do not grow too rapidly.

\begin{proposition}\label{t:sparse}
    Suppose that for any measure-preserving system $(X,\mu,T)$, whenever $T_1,T_2$ are powers of $T$, for any $f,g \in L^\infty(X)$, the bilinear ergodic averages
    \begin{align}
        \frac{1}{N} \sum_{n \leq N} w(n) T_1^n f \cdot T_2^n g
    \end{align}
   converge $\mu$-almost surely. Then the same is true for
   \begin{align}
        \frac{1}{|\mathbb{N}_c \cap [1,N]|} \sum_{n \leq N, \ n \in \mathbb{N}_c} w(n)  T_1^n f \cdot T_2^n g,
    \end{align}
    whenever $1 \leq c < 7/6$, provided that 
    \[ \| w \|_{U^3([N])} \leq \text{Const} \cdot N^{o(1)}. \]
\end{proposition}
We remark that all three of our examples considered satisfy the moment estimate,
\begin{align}
    \frac{1}{N} \sum_{n \leq N} |w(n)|^2 \leq \text{Const} \cdot N^{o(1)},
\end{align}
and so in particular satisfy the requisite $U^3$ bound, by the embedding $U^3 \hookrightarrow L^2$.

Following Bourgain's lead, in what follows we will only focus on the cases in Theorem \ref{t:main}, Corollary \ref{c:main}, and Proposition \ref{t:sparse} where $T_1 = T, \ T_2 = T^{-1}$, as all other cases can be addressed by suitably adjusting notation. Accordingly, we set
\begin{align}\label{e:B1Scale}
    B_N(f,g) := B_{w,N}(f,g) := \frac{1}{N} \sum_{n \leq N} w(n) T^n f  \cdot T^{-n} g.
\end{align}

\medskip

We now give some remarks about these results:
\begin{enumerate}
\item Recent work on bilinear (weighted) pointwise ergodic theorems has focused on the case of polynomial orbits with distinct degrees, e.g.\ averages of the form
\begin{align}
    \frac{1}{N} \sum_{n \leq N} T^n f  \cdot T^{n^2} g \; \; \; \text{ or }   \; \; \;   \frac{1}{N} \sum_{n \leq N} \Lambda(n) T^n f  \cdot T^{n^2} g,
\end{align}
see \cite{KMT} and \cite{KMTT}, respectively. In fact, a very well-developed multi-linear theory exists for these classes of ``distinct-degree" examples, see \cite{KMPWW} and \cite{Wan}. The key ingredient in these works is an appropriate implementation of \emph{Peluse Theory} \cite{Pel}, see also \cite{PP1}, a ``degree-lowering" mechanism that essentially exploits the distinct-degree nature of the pertaining orbits, intimately connected to the fact that the procyclic factor is characteristic for multi-linear polynomial ergodic averages with distinct-degree polynomials. On the other hand, the averages we consider encode a hidden \emph{modulation invariance}, see \eqref{e:modinv00} below, corresponding to the fact that the larger Kronecker factor is characteristic for the averages in Theorem \ref{t:main} and Proposition \ref{t:sparse};
\item By following the arguments of \cite{F+}, one can essentially show that admissible weights satisfy the conclusion of the return times:
\begin{theorem*}[Return Times Theorem for Admissible Weights]
    Suppose that $w$ is admissible. Then for any measure-preserving system $(\Omega,\mu_0,S)$ and any $g \in L^{\infty}(\Omega)$, there exists $\Omega_g \subset \Omega$ with $\mu_0(\Omega_g) = 1$, so that for all $\omega \in \Omega_g$, the following hold: for any measure-preserving system, $(X,\mu,T)$, and any $f \in L^\infty(X)$,
    \begin{align}
        \frac{1}{N} \sum_{n \leq N} w(n) T^n f \cdot g(S^n \omega)
    \end{align}
converges $\mu$-almost surely.
\end{theorem*}
Indeed, the only non-formal modification needed to pass from \cite{F+} to the current context concerns the so-called fixed scale $U^3$ estimate for admissible weights, see \cite[\S 2]{F+}, which introduces some further number-theoretic technicality due to the fact that $S(a/q)$ need not vanish whenever $q$ is divisible by the square of a prime. This point will be addressed in the forthcoming work of the first author; in our context, where the quantitative dictates are somewhat less stringent, Lemma \ref{p:U3est} below suffices;
\item Given some mild decay of \[ \frac{1}{N} \sum_{n \leq N} w(n) e^{-2 \pi i n \beta} \]
when $\beta$ is ``$N$-far" from rational with ``$N$-small" denominators -- an estimate of the form
\begin{align}\label{e:FTapprox}
    \| \frac{1}{N} \sum_{n \leq N} (w - \sum_{Q \leq M} w_Q)(n) e^{-2 \pi i n \beta}\|_{L^{\infty}(\mathbb{T})} \leq \text{Const} \cdot N^{-c}
\end{align}
for some $c > 0$ more than suffices -- one can argue as in \cite[\S 5]{BKMod} to establish an $L^p(X)$ maximal inequality for each $p > 1$, 
\begin{align}
\| \sup_N |\frac{1}{N} \sum_{n \leq N} w(n) T^n f| \|_{L^p(X)} \leq \text{Const} \cdot \| f \|_{L^p(X)},
\end{align}
which can be interpolated to show that for each $r > 1$,
\begin{align}
    \| \sup_N |B_{w,N}(f,g) | \|_{L^r(X)} \leq \text{Const} \cdot \|f \|_{L^p(X)} \| g\|_{L^q(X)}
\end{align}
whenever $\frac{1}{p} + \frac{1}{q} = \frac{1}{r} < 1$, see \eqref{e:B1Scale}; in this case, Theorem \ref{t:main} accordingly extends to the case where $f \in L^p(X), \ g \in L^q(X)$ provided that $\frac{1}{p} + \frac{1}{q} < 1$;
\item Proposition \ref{t:sparse} is already new in the case where 
\[ w(n) \equiv \mathbf{1}_{\mathbb{Z}}(n) = e(0/1 \cdot n) \]
in the range $23/22 \leq c < 7/6$, although the range $1 < c < 23/22$ was recently addressed in \cite{LD}; 
\item In the case where $w = \mathbf{1}$, there is a close connection between Theorem \ref{t:BDR} and Lacey's work \cite{L} on the (Euclidean) bilinear maximal function
\begin{align} B_{\mathbb{R}}(f,g)(x) := \sup_{r > 0} |\int_0^1 f(x-rt) g(x+rt) \ dt|;
\end{align}
this connection is most neatly explored in \cite{DOP, BD}. On the other hand, no non-trivial estimates are known for singular variants of $B_{\mathbb{R}}$, i.e.\ 
\begin{align}\label{e:Bmu} B_{\mu} (f,g)(x) := \sup_{r > 0} |\int_0^1 f(x-rt) g(x+rt) \ d\mu(t)|, \end{align}
where $\mu$ is a singular measure on $[0,1]$. Contrast this with Theorem \ref{t:main} and Proposition \ref{t:sparse}: from the perspective of density, the weights we consider are analogous to singular measures, often supported on lower-dimensional sets. For instance, there is an analogy between the weight
\begin{align}
   n \mapsto \Lambda(n) \cdot c n^{1-1/c} \cdot \mathbf{1}_{\mathbb{N}_c}(n)
\end{align}
and measures $\mu$ supported on sets $E \subset [0,1]$ with Hausdorff dimension $1/c$, $\text{dim}_H(E) = 1/c$, but $\mathcal{H}^{1/c}(E) = 0$.
\end{enumerate}

\medskip

With the above in mind, we situate our work in its proper historical context.

\subsection{History}
The study of pointwise convergence of ergodic averages dates back to Birkhoff \cite{BI}; our discussion below will focus only on the case of bounded functions.

\begin{theorem*}[Pointwise Ergodic Theorem] 
Suppose $(X,\mu,T)$ is a measure-preserving system, and that $f \in L^\infty(X)$. Then the averages 
\[ \frac{1}{N} \sum_{n \leq N} T^nf\]
converge $\mu$-almost surely.
\end{theorem*}
This result was dramatically strengthened by Wiener-Wintner \cite{WW}:

\begin{theorem*}[Wiener-Wintner Ergodic Theorem] 
    Suppose $(X,\mu,T)$ is a measure-preserving system, and that $f \in L^\infty(X)$. Then $\mu$-almost surely, the averages 
\[ \frac{1}{N} \sum_{n \leq N} T^nf \cdot  e^{2\pi i n \theta}  \]
converge for \emph{all} $\theta \in [0,1]$.
\end{theorem*}
Note that a statement about Lebesgue-almost every $\theta$ follows directly from the Pointwise Ergodic Theorem in a product system $(X \otimes \mathbb{T}, d\mu\otimes dx)$, so the Wiener-Wintner Theorem is a genuine strengthening.

By the spectral theorem, the Wiener-Wintner Theorem implies that for $\mu$-almost every $x \in X$, for any secondary measure-preserving system $(\Omega,\mu_0,S)$, and any $g \in L^\infty(\Omega)$, the averages 
\begin{align}
    \frac{1}{N} \sum_{n \leq N} f(T^n x) \cdot S^n g
\end{align}
converge in $L^2(\mu_0)$. This result was strengthened to a statement about pointwise convergence by Bourgain-Furstenberg-Katznelson-Ornstein in \cite[Appendix]{B2}.
\begin{theorem*}[Return Times Theorem] 
        For any measure-preserving system $(\Omega,\mu_0,S)$ and any $g \in L^{\infty}(\Omega)$, there exists $\Omega_g \subset \Omega$ with $\mu_0(\Omega_g) = 1$, so that for all $\omega \in \Omega_g$, the following hold: for any measure-preserving system, $(X,\mu,T)$, and any $f \in L^\infty(X)$,
    \begin{align}
        \frac{1}{N} \sum_{n \leq N} T^n f \cdot g(S^n \omega)
    \end{align}
converges $\mu$-almost surely.
\end{theorem*}
The phrase ``return times" is explained by taking $g = \mathbf{1}_G$ to be the indicator function of a non-trivial measurable set, $G \subset \Omega$. Informally, the presence of the secondary function, $g$, adds a ``half-degree" of multi-linearity to Birkhoff's Theorem; although Bourgain et.\ al.\ offered a very succinct proof using ``soft" methods, an earlier unpublished argument of Bourgain \cite{B00} revealed a connection to Theorem \ref{t:BDR}, see \cite{BKMod} for a unified approach to both Theorems.

In another direction, motivated by a question of Furstenberg and Bellow, in \cite{B0,B1,B2} Bourgain addressed the issue of pointwise convergence along \emph{sparse} (polynomial) sequences.
\begin{theorem*}[Bourgain's Polynomial Ergodic Theorem]
    Let $P \in \mathbb{Z}[\cdot]$ be a polynomial with integer coefficients. Then for any measure-preserving system, $(X,\mu,T)$, and any $f \in L^{\infty}(X)$, the averages
    \begin{align}
        \frac{1}{N} \sum_{n \leq N} T^{P(n)} f
    \end{align}
    converge $\mu$-almost surely.
\end{theorem*}

In the course of this work, he also addressed the issue of pointwise convergence along prime times, \cite{Bp}, see also \cite{Wierdl} and \cite{MTZK} for quantitative refinements.
\begin{theorem*}[Bourgain's Ergodic Theorem along the Primes]
    For any measure-preserving system, $(X,\mu,T)$, and any $f \in L^{\infty}(X)$ the averages
    \begin{align}\label{e:primeavg1}
        \frac{1}{N} \sum_{n \leq N} T^{p_n} f
    \end{align}
    converge $\mu$-almost surely, see \eqref{e:primes}.
\end{theorem*}

To prove his prime ergodic theorem, Bourgain re-parametrized the averages \eqref{e:primeavg1} via the von Mangoldt function \eqref{e:vonMangoldt}; namely, by summation by parts, see \eqref{e:sumbyparts} above, he was able to instead address the issue of pointwise convergence of the \emph{weighted} averages
\begin{align}\label{e:primeavg2}
        \frac{1}{N} \sum_{n \leq N} \Lambda(n) T^{n} f.
    \end{align}
Indeed, the polynomial ergodic theorems can be similarly studied from this perspective, as e.g.\ convergence of the ergodic means along the squares can be recast in terms of weighted averages,
\begin{align}
    \frac{1}{N} \sum_{n \leq N} T^{n^2}f \longrightarrow     \frac{1}{N} \sum_{n \leq N} w(n)  T^{n}f, \; \; \; w(n) = 2 \sqrt{n} \cdot \mathbf{1}_{\{ k^2 : k \in \mathbb{N} \}}(n). 
\end{align}
This line of inquiry was later pursued by Cuny-Weber \cite{CW}, who abstracted Bourgain's prime ergodic theorem to address a class of arithmetic weights that satisfied similar number theoretic statistics to the set of primes. For instance, they addressed the convergence of ergodic averages weighted by the divisor function \begin{align}\label{e:diveavg2}
        \frac{1}{N \log N} \sum_{n \leq N} \tau(n) T^{n} f,
    \end{align}
see \eqref{e:divisor}.\footnote{The additional factor of $\log N$ is for normalization purposes: $\frac{\sum_{n \leq N} \tau(n)}{N \log N} \to 1$, see \eqref{voronoi-summation-consequence} below.}

Recently, in collaboration with Fragkos-Lacey-Mousavi-Sun, the authors established a joint synthesis of Bourgain's Return Times Theorem and Ergodic Theorem along the Primes, connecting weighted ergodic averages with return times phenomena. 

\begin{theorem*}[Return Times along the Primes]
    For any measure-preserving system $(\Omega,\mu_0,S)$ and any $g \in L^{\infty}(\Omega)$, there exists $\Omega_g \subset \Omega$ with $\mu_0(\Omega_g) = 1$, so that for all $\omega \in \Omega_g$, the following holds: for any measure-preserving system, $(X,\mu,T)$, and any $f \in L^\infty(X)$,
    \begin{align}
        \frac{1}{N} \sum_{n \leq N} \Lambda(n) T^n f\cdot  g(S^n \omega),
    \end{align}
    and thus
   \begin{align}
        \frac{1}{N} \sum_{n \leq N} T^{p_n} f \cdot g(S^{p_n} \omega),
    \end{align}
converge $\mu$-almost surely.
\end{theorem*}

Given the close connection between return times theorems and double recurrence, the authors were motivated to begin working towards Theorem \ref{t:prime}, and in light of \cite{CW}, to exhibit a testing condition to address more general arithmetic weights. Once again, the presence of the weights introduce an additional ``half-degree" of multi-linearity -- albeit an arithmetically constrained one -- which situates the current work as an important stepping stone on the way to addressing the issue of triple recurrence, namely the $K=3$ case of \eqref{e:intromul}. Indeed, both averages
\begin{align}
\frac{1}{N} \sum_{n \leq N} w(n) T^n f \cdot T^{2n} g, \; \; \;  \; \; \; \frac{1}{N} \sum_{n \leq N} T^n f \cdot T^{2n} g \cdot T^{3n} h
\end{align}
are controlled at the single scale level by $U^3$ statistics, an important point below.\\

\medskip
With this in mind, we describe our approach.

\subsection{Proof Overview}
Proposition \ref{t:sparse} follows from an elementary argument deriving from \cite{KS}, so we confine our discussion below to Theorem \ref{t:main}.

\medskip

Our proof technique is highly motivated by ergodic-theoretic considerations: namely
\begin{enumerate}
    \item The following estimate holds:
    \[ \| \frac{1}{N} \sum_{n \leq N} w(n) T^n f \cdot T^{-n} g \|_{L^{2^s}(X)} \lesssim \| w \|_{U^{s+2}([N])}, \; \; \; s \geq 1,\]
    whenever $f$ and $g$ are $1$-bounded; and
    \item The Kronecker factor is \emph{characteristic} for weighted bilinear averages:
    \begin{align}
        \| \frac{1}{N} \sum_{n \leq N} w(n) T^n f \cdot T^{-n} g \|_{L^2(X)} \to 0
    \end{align}
    whenever $f$ and $g$ are $1$-bounded and $g \in \mathcal{K}(T)^{\perp}$ is orthogonal to the Kronecker factor.
\end{enumerate}

The above two points tell us that we are free to replace weights
\begin{align}
    w \longrightarrow w_N
\end{align}
provided 
\begin{align}
    \| w - w_N \|_{U^{s+2}([N])} \lesssim \log^{-\nu/2^s} N, \; \; \; \nu > 4,
\end{align}
by standard lacunary reductions, but that our arguments should be highly Fourier analytic; morally speaking, we are motivated to decompose $w$ according to its correlation with rational frequencies according to the size of their denominators. Taken together, our admissibility criterion naturally presents, with the constraint 
\begin{align}
    |S(a/q)| \leq \text{Const} \cdot q^{o(1)-1}
\end{align}
naturally appearing in the course of our arguments.\footnote{In point of fact, one can slightly relax the decay constraint on $S(a/q)$ to $|S(a/q)| \leq \text{Const} \cdot q^{c-1}$ for e.g.\ $c = 2^{-100}$ by optimizing our below arguments; we do not pursue the issue of optimal constants.}

After a brief argument involving multi-frequency analysis, see  \S \ref{s:MultiFreq} below, our proof proper begins by assuming that $g \in \mathcal{K}(T)^{\perp}$ lives in the orthocomplement of the Kronecker factor; we transfer this statement to the dynamical systems setting by arguing similarly to \cite{F+}. As in \cite{F+}, two parameters naturally present
\begin{itemize}
    \item $Q$, which controls the ``height" of our rational frequencies; and
    \item $\delta$, which controls the spectral statistics of $g$.
\end{itemize}
The argument when $Q$ is small compared to $\delta^{-1}$ is of a simpler nature, though already encompasses the work of \cite{B3}, so we focus the remainder of this discussion to the case where
\begin{align}
    Q \geq \delta^{-1/1000}.
\end{align}

In particular, we are interested in the interplay between
\begin{align}
    \Gamma_Q := \{ a/q \text{ reduced} : Q/2 < q \leq Q \}
\end{align}
and the pertaining component of the weight,
\begin{align}
    w_Q(n) := \sum_{a/q \in \Gamma_Q} S(a/q) e^{2\pi i n a/q},
\end{align}
see \eqref{e:GammaQ} and \eqref{e:WQ} below, and the $\delta$-spectrum of $g$ at each scale and location,
\begin{align}
\text{Spec}_\delta(I) := \big\{ \xi \in \mathbb{Z}/|I| : \delta/2<\big| \frac{1}{|I|} \sum_{n \in I} g(n) e^{- 2 \pi i n \xi} \big| \leq \delta \big\},
\end{align}
see \eqref{e:Specdelta}. In the simplest case, where $g$ is a linear combination of characters, after an arithmetic combinatorial argument, see \S \ref{s:AE} below, matters reduce to the setting of \cite{B2}, in which a delicate multi-frequency analysis, similar to that developed in \S \ref{s:MultiFreq}, is the key point. To the extent that these arguments were established via a metric chaining argument -- anchored by entropic considerations and L\'{e}pingle's martingale inequality, see \S \ref{ss:Lep} below -- we expect similar tools to arise in our context. While there is no reason to expect $g$ to have such a simple structure, we appeal to subtle orthogonality methods -- energy pigeon-holing, and Cotlar-Stein based wave-packet analysis, in particular -- to try reduce to this case: roughly speaking, we collect scales and locations (indexed by elements of \emph{dyadic grids}, see \S \ref{ss:dyadicgrids}),
\begin{align}\label{e:branches}
    \mathcal{B}([\Lambda]) \subset \mathcal{D},
\end{align}
so that our weighted averages, 
see \eqref{e:AI} below,
\begin{align}\label{e:model}
\big\{ \sum_n \phi_I(n) f(2x-n) w_Q(n-x) g(n) : I \in \mathcal{B}([\Lambda]) \big\}, \; \; \; \phi_I ``=" \frac{1}{|I|} \mathbf{1}_I,
\end{align}
satisfy
\begin{align}\label{e:introkey}
&\sup_{I \in \mathcal{B}([\Lambda])  } |\sum_n \phi_I(n) f(2x-n) w_Q(n-x) g(n)|  \\ 
& ``=" \sup_{I \in \mathcal{B}([\Lambda])  } |\sum_n \phi_I(n) f(2x-n) w_Q(n-x) (\Pi_I[\Lambda]g)(n)|,
\end{align}
where by quotations we mean moral equivalance, and 
\begin{itemize}
    \item the number of distinct collections $\{ \mathcal{B}([\Lambda]) : \Lambda\}$ is ``small";
    \item  $\{ \Pi_I[\Lambda] \}$ are multi-frequency Fourier multipliers, rooted around a common collection of distinguished frequencies, $\Lambda \subset \mathbb{T}$, and are similar to those considered by Bourgain in his work on polynomial ergodic theorems; and
    \item each $\Lambda$ has size controlled by the uncertainty principle
\[ |\Lambda| \; \; \;  ``\leq" \; \; \; \text{Const} \cdot \delta^{-2},
\]
see Corollary \ref{c:sampling} below.
\end{itemize}

As alluded to above, the operator \eqref{e:introkey} is agnostic to replacing the functions
\begin{align}\label{e:modinv00}
    (f(x),g(x)) \longrightarrow (e^{2 \pi i \theta x} f(x),e^{2 \pi i \theta x} g(x) )
\end{align}
for any $\theta$, which means that the major/minor arc analysis used in the study of the polynomial ergodic theorems, or multi-linear ergodic theorems involving polynomials of distinct degrees, is inappropriate. Rather, our arguments are motivated by those of Bourgain in \cite{B3}, but our current context forces us, at many of the steps of the argument, to take into account further issues of constructive interference in Fourier space, deriving from sumsets of the form
\[ m \cdot \text{Spec}_\delta(I) + n \cdot \Gamma_Q, \; \; \; |m|,|n| \leq 10,\]
and accordingly to use finitary methods. With this in mind, our arguments are almost entirely $\ell^2$-based, whereas Bourgain was able to make use of $\ell^1$-techniques; these $\ell^1$-based arguments reappear to some extent in our concluding section, \S \ref{s:remainder}.

With the reduction to \eqref{e:introkey}, valid away from an acceptably small exceptional set, one readily restricts the Fourier transform of ${f}$ to neighborhoods of sumsets of the form
\begin{align}\label{e:sumset0}
    \big\{ \xi \in \mathbb{T} : \text{dist}(\xi,\Gamma_Q + \Lambda) \leq \text{Const}/N \big\}, \; \; \; N = |I|
\end{align}
via additional multi-frequency projections, thus, with
\begin{align}
    \Pi_N[\Lambda] = \Pi_{N;Q}[\Lambda]
\end{align}
denoting a smooth Fourier projection to \eqref{e:sumset0}, we may approximate
\begin{align}
    \eqref{e:introkey} \; ``=" \sup_{I \in \mathcal{B}([\Lambda])} |\sum_n \phi_I(n) (\Pi_N[\Lambda] f)(2x-n) w_Q(n-x) (\Pi_I[\Lambda]g)(n)|;
\end{align}
this in turn necessitates a \emph{bilinear} perspective on entropy. This issue already presented in a simpler context in \cite{B3}, but arithmetic issues introduce serious complications in our setting. These entropy arguments are anchored, as might be expected, by appropriate single scale estimates -- with the caveat that we must take into account the behavior of our single scale averages along arithmetic progressions as well as full intervals -- in conjunction with the combinatorial number theory developed in \S \ref{s:AE} and the estimates of \S \ref{s:MultiFreq}.

\medskip

A major theme of our work is the role of dyadic harmonic analysis, which presents in two distinct, but dual, contexts:
\begin{itemize}
    \item In physical space, where the use of dyadic grids, see \S \ref{ss:dyadicgrids} below, readily connects our problem to the martingale setting: our intervals in \eqref{e:model} can be assumed to derive from a single dyadic grid, and thus satisfy the convenient nesting property
    \begin{align}
    I,J \in \mathcal{D} \; \; \; \text{ with } \;\; \; I \subset J \Rightarrow I \cap J = I;
    \end{align}
    this greatly facilitates our stopping time algorithms of \S \ref{s:ENERGY} and wave-packet analysis of \S \ref{s:tight};
    and
    \item In frequency space, where our analysis essentially takes place in the context of
    \begin{align}
        \{ \mathbb{Z}/2^N : N \in \mathbb{N} \} \subset \mathbb{T};
    \end{align}  
namely, we can morally assume that each $\Lambda$ is comprised of dyadic rational frequencies. This presents us with rich arithmetic structure, see \S \ref{s:AE} below, used crucially in our bilinear entropy arguments. To exploit this structure, we introduce a secondary decomposition of $\Gamma_Q$ according to the size of the maximal power of $2$ that divides our denominators, namely
\begin{align}
    \Gamma_Q = \bigcup_{0 \leq i \leq \log_2 Q} \Gamma_Q^{(i)}, \; \; \;\; \; \; \Gamma_Q^{(i)} := \{ a/q \in \Gamma_Q : 2^i | q, \ i \text{ maximal} \} 
\end{align}
and analogously decompose our weights 
\[ w_Q = \sum_{0 \leq i \leq \log_2 Q} w_Q^{(i)}; \]
see \eqref{e:gammaQi} and \eqref{e:wQi}.
\end{itemize}

With this in mind, we detail the structure of our paper.

\subsection{Structure}
Our paper begins with definitions and notational conventions, with the key point being a hierarchy of parameters/constants, see \S \ref{ss:param} below; the arguments are quite involved, and so the number of terms we introduce is relatively high;

In \S \ref{s:adm}, we present some properties of admissible weights, along with our three examples, and in \S \ref{s:sparse} we establish Proposition \ref{t:sparse};

Next, in \S \ref{s:analess}, we present some analytic tools to which we will appeal throughout, and in \S \ref{s:AE}, we develop certain tools from combinatorial number theory, crucially relying on the dyadic structure of our relevant sets of frequencies;



In our next section, \S \ref{s:MultiFreq}, we develop our multi-frequency theory, similar to \cite{BK1}, adapted to our current context;

In \S  \ref{s:polysplit}, we present a suitable variant of the uncertainty-principle for trigonometric polynomials which we use crucially in \S \ref{s:ENERGY} to establish our key orthogonality arguments;

\medskip

With these preliminaries in place, the remainder of the paper is concerned with the proof of Theorem \ref{t:main}:

\medskip

We begin the main thrust of our argument in \S \ref{s:trans}, where we reduce to the case where one of our functions is orthogonal to the Kronecker factor, and appropriately transfer our problem to the integer lattice;

\S \ref{s:tight} is of a technical nature, but will allow us to morally replace one of our functions with a linear combination of characters at the scales and locations aggregated in the sets $\mathcal{B}([\Lambda])$;

In \S \ref{s:entprelim}, we develop the single scale estimates used to anchor our crucial entropy arguments;

In \S \ref{s:nonboundary}, we close the main body of the argument, namely we resolve Theorem \ref{t:main} when our governing parameters $Q,\delta$ are related by
\[ Q \geq \delta^{-1/1000};\]

Finally, the argument is concluded in \S \ref{s:remainder}, in which we address the reverse inequality,
\[ Q \leq \delta^{-1/1000};\]
the arguments here are much more similar to \cite{B3}.

\subsection{Open Problems}
We close our introduction with a number of open problems.

\begin{problem}[Describing Admissibility]
    Can one explicitly identify our class of admissible weights?
\end{problem}

\begin{problem}[Extending Convergence]
Given the approximation \eqref{e:FTapprox}, can one establish convergence for $f \in L^p(X), \ g \in L^q(X)$ with $\frac{1}{p} + \frac{1}{q} \geq 1$ and $p,q > 1$? By \cite{GG}, this is equivalent to proving a weak-type bound
    \begin{align}
        \| \sup_N | B_{w,N}(f,g) | \|_{L^{r,\infty}(X)} \lesssim \|f \|_{L^p(X)} \| g \|_{L^q(X)}
    \end{align}
    for some $r \leq 1$, see \eqref{e:B1Scale}. In the case where $w \equiv 1$, such an estimate was provided by Lacey \cite{L}, in the range when $r > 2/3$. 
\end{problem}

\begin{problem}[Quantifying Convergence]
    It seems likely that one could combine the methods developed here with those of \cite{BKMod} to prove that
    \begin{align}\label{e:rateofdecay}
       \sup \| \mathcal{O}_{\{ N_i \};J}(f,g) \|_{L^2(X)} = o_{J \to \infty;\lambda}(1) \cdot \|f \|_{L^\infty(X)} \|g \|_{L^\infty(X)},
    \end{align}
where the outer supremum is over all $\lambda$-lacunary sequences, $\{ N_i \}$, 
\begin{align}
    \mathcal{O}_{\{N_i\};J}(f,g)(x)^2  := \frac{1}{J} \sum_{j \leq J} \sup_{N_i \leq N < N_{i+1}} |B_{w,N}(f,g) - B_{w,N_{i+1}}(f,g)|^2,
\end{align}
and all times involved are $\lambda$-lacunarily increasing. On the other hand, quantifying the rate of decay in \eqref{e:rateofdecay} via a power savings in $J$ is currently out of reach, as opposed to the case where $w = \mathbf{1}$, see \cite{DOP,BD}. Note that, given the condition \eqref{e:FTapprox}, a power savings in \eqref{e:rateofdecay} immediately implies an analogous (weaker) power savings from e.g.\ $L^4(X) \times L^4(X) \to L^2(X)$, by interpolation.
\end{problem}

\begin{problem}[Double Recurrence Along Polynomial Orbits]
    Our work makes crucial use of the decay condition
    \begin{align}
        |S(a/q)| \leq \text{Const} \cdot q^{o(1)-1};
    \end{align}
note that even in the case where
\begin{align}
    w(n) := 2 \sqrt{n} \cdot \mathbf{1}_{\{ k^2: k \in \mathbb{N} \}}(n)
\end{align}
is the counting function of the squares, so that one is tempted to take
\begin{align}
    w_Q(n) = \sum_{a/q \in \Gamma_Q} S(a/q) e(na/q ), \; \; \; S(a/q) = \frac{1}{q} \sum_{r \leq q} e(-r^2 a/q),
\end{align}
where one has the sharp estimate
\begin{align}
    |S(a/q)| \leq \text{Const} \cdot q^{-1/2},
\end{align}
and thus
\begin{align}
    \| w_Q \|_{U^2([N])} \approx 1
\end{align}
for all $N \geq Q^{10}$ (say), which leads to difficulty in producing the desired decay rate in the Heath-Brown criterion in admissibility and precludes any sort of single scale savings as in Lemma \ref{p:U3est} -- before encountering numerological breakdowns arising in the analytic part of our proof. Nevertheless, the following question naturally presents: can one prove pointwise almost sure convergence for
\begin{align}
    \frac{1}{N} \sum_{n \leq N} T^{n^2} f \cdot T^{-n^2} g?
\end{align}

The additive combinatorial approach of \cite{PSS} seems promising, but is currently too quantitatively weak to be of use; see also \cite{Prend}, which goes by way of $U^7$ analysis, and so remains beyond the purview of our (primarily) Fourier-analytic approach. 

A toy model concerns averages of the form
\begin{align}
    \frac{1}{N^{1+c}} \sum_{n \leq N, \ m \leq N^c} T^{n^2 + m^2} f \cdot T^{-n^2-m^2} g, \; \; \; 0 < c < 1,
\end{align}
where the $c=1$ case follows from approximating a square from the inside by unions of sectors, and then comparing convergence inside sectors,
\begin{align}
    r_{2;\omega}(n) := |\big\{ a^2 + b^2 = n, \; \; \; \text{arg}(a+ib) \in \omega \subset [0,2\pi] \big\}| 
\end{align}
to our weighted context, using the estimates
\begin{align}
\frac{1}{N} \sum_{n \leq N} |r_{2;\omega}(n) - \frac{|\omega|}{2\pi} \cdot r_2(n)| = o_{N \to \infty}(1).
\end{align}
\end{problem}

\begin{problem}[Euclidean Analogues]
    This problem concerns statistics of the singular variant of the Euclidean bilinear operator, \eqref{e:Bmu}. Can one provide a testing condition on $\mu$ so that
    \begin{align}
        \| B_\mu(f,g) \|_{L^{1,\infty}(\mathbb{R})} \lesssim \| f \|_{L^2(\mathbb{R})} \|g \|_{L^2(\mathbb{R})}?
    \end{align}
A class of good candidate (Cantor) measures were constructed in \cite{LP} and \cite{SS}, for which $B_\mu$ maps into $L^{1+}(\mathbb{R})$. In two dimensions, the natural conjecture concerns 
\[ B_{\theta}(f,g)(x) := \sup_{r > 0} |\int_{\mathbb{S}^1} f(x-r\omega) g(x-r R_\theta \omega) \ d\sigma(\omega)|, \]
where $R_\theta$ is given by rotation by $\theta$. This example has the same dimensionality/density as the squares inside of $\mathbb{N}$ -- and indeed $\| \sigma \|_{U^2}$ diverges only logarithmically. Estimating $B_{\theta}$ is closely linked to the famous problem of detecting vertices of large (non-degenerate) triangles inside subsets of positive upper density inside the plane, and accordingly a distinction occurs according to whether $\theta = \pi$ or otherwise, see \cite{GIKL} for discussion.
\end{problem}

\begin{problem}[Double Recurrence Meets Return Times]
    A natural next step, before attempting the famous problem of triple recurrence (discussed below), is that of combining double recurrence with return times behavior, i.e.\ formally substituting
    \begin{align}
    w(n) \longrightarrow h(S^n \omega)
    \end{align}
    where $h \in L^\infty(\Omega,\mu_0,S)$ is a bounded function on a secondary measure-preserving system, and $\omega \in \Omega$ is suitably ``generic."
    
    Specifically, one would aim to prove the following:
    \medskip 
   
    Suppose that $(\Omega,\mu_0,S)$ is a measure-preserving system, and let $h \in L^\infty(\Omega)$. There exists a set $\Omega_h \subset \Omega$ with $\mu_0(\Omega_h) =1$ so that for all $\omega \in \Omega_h$, the following holds: for any measure-preserving system, $(X,\mu,T)$, and any $f,g \in L^\infty(X)$, the averages 
    \begin{align}
        \frac{1}{N} \sum_{n \leq N} T^n f \cdot T^{-n} g \cdot  h(S^n \omega)
    \end{align}
converge $\mu$-almost surely.
\end{problem}

The following problem remains most compelling, and, as mentioned above, served as a large source of motivation for us to begin our current line of inquiry. 
\begin{problem}[Triple Recurrence]
    Let $(X,\mu,T)$ be a measure-preserving system. Prove that for each $f,g,h \in L^\infty(X)$, the trilinear averages
    \begin{align}
        \frac{1}{N} \sum_{n \leq N} T^n f \cdot T^{2n} g \cdot T^{3n} h
    \end{align}
    converge $\mu$-almost surely.
\end{problem}

\subsection{Acknowledgments}
The second author would like to thank Joni Ter\"{a}v\"{a}inen for generously explaining some of the key ideas of \cite{Leng,LSS}.

\section{Preliminaries}\label{s:Prelim}

\subsection{General Notation}
Throughout, we let
\begin{align}
    e(t) := e^{2 \pi i t}
\end{align}
denote the complex exponential, 
\begin{align}
    \| x \| := \| x \|_{\mathbb{T}} := \min_{k \in \mathbb{Z}} |x-k|,
\end{align}
and for $\theta \in \mathbb{R}$, define
\begin{align}
    \text{Mod}_\theta g(x) := e(\theta x) g(x).
\end{align}
For $r > 0$, we let 
\[ B(r) := \{ |\xi| \leq r\}\]
denote the ball of radius $r$; the ambient space will be clear from context.

We use the standard notation for $L^1$-normalized dilations:
\begin{align}
    \varphi_N(n) := \frac{1}{N} \varphi(\frac{n}{N}),
\end{align}
and let
\begin{align}
    M_{\text{HL}} f(x):=  \sup_{x \in I} \, \frac{1}{|I|} \sum_{n \in I} |f(n)|  
\end{align}
denote the discrete Hardy-Littlewood maximal function; here and throughout, we use $I$ to denote discrete intervals.

We collect
\begin{align}\label{e:GammaQ}
    \Gamma_Q := \{ a/q \text{ reduced} : Q/2 < q \leq Q \},
\end{align}
define
\begin{align}\label{e:lcmQ}
\mathcal{Q} := \text{lcm}\{ q : Q/2 < q \leq Q \},
\end{align}
and set
\begin{align}
    \mathbf{S}_Q := \mathbf{S}_{w,Q} := \sup_{a/q \in \Gamma_Q} |S(a/q)|,
\end{align}
where $S= S_w$ is in the definition of admissibility; regarding this $S$ as fixed, we define
\begin{align}\label{e:WQ}
    w_Q(n) := \sum_{a/q \in \Gamma_Q} S(a/q) e(na/q ).
\end{align}
In particular, $w_Q$ should be thought of as the $Q$-denominator slice of the major arc approximation of $w$; the majority of our later analysis is carried out at the level of each individual slice.

We also introduce dyadic divisions of the foregoing. 
Namely, we set
\begin{align}\label{e:gammaQi}
    \Gamma_Q^{(i)} := \{a/q \in \Gamma_Q :  2^i \parallel q \},
\end{align}
where we say $2^i \parallel q$ if $i$ is the maximal integer so that $2^{i} |q$, and similarly define
\begin{align}\label{e:lcmQi}
    \mathcal{Q}_i := \text{lcm}( q : Q/2 < q \leq Q, \  2^i \parallel q) 
\end{align}
and
\begin{align}\label{e:wQi}
    w_Q^{(i)}(n) := \sum_{a/q \in \Gamma_Q^{(i)}} S(a/q) e(an/q).
\end{align}

We let $\psi$ be a smooth approximation to $\mathbf{1}_{(1/2,2]}$ so that
\begin{align}
    \mathbf{1}_{(0,\delta_0]} \leq \sum_{0 < \delta \leq \delta_0} \psi_\delta(t) \leq \mathbf{1}_{(0,2\delta_0]},
\end{align}
and with $\psi_\delta(t) := \psi(\delta^{-1} t)$, set
\begin{align}\label{e:Psidelta}
    \Psi_\delta(z) := z \psi_\delta(|z|),
\end{align}
so that
\begin{align}\label{e:Lip}
    \sup_\delta \| \Psi_\delta \|_{\text{Lip}} \lesssim 1.
\end{align}

\subsection{Asymptotic Notation}\label{sss:O}
We will make use of the modified Vinogradov notation. We use $X \lesssim Y$ or $Y \gtrsim X$ to denote
the estimate $X \leq CY$ for an absolute constant $C$ and $X, Y \geq 0.$  If we need $C$ to depend on a
parameter, we shall indicate this by subscripts, thus for instance $X \lesssim_p Y$ denotes the estimate $X \leq C_p Y$ for some $C_p$ depending on $p$. We use $X \approx Y$ as shorthand for $Y \lesssim X \lesssim Y$. We use the notation $X \ll Y$ or $Y \gg X$ to denote that the implicit constant in the $\lesssim$ notation is extremely large, and analogously $X \ll_p Y$ and $Y \gg_p X$.

We also make use of big-Oh and little-oh notation: we let $O(Y)$  denote a quantity that is $\lesssim Y$, and similarly
$O_p(Y )$ will denote a quantity that is $\lesssim_p Y$; we let $o_{t \to a}(Y)$
denote a quantity whose quotient with $Y$ tends to zero as $t \to a$ (possibly $\infty$), and
$o_{t \to a;p}(Y)$
denote a quantity whose quotient with $Y$ tends to zero as $t \to a$ at a rate depending on $p$. When clear from context, we will suppress the $t \to a$ subscript.

\subsection{Fourier Transforms}
We will make heavy use of the Fourier transform below; we will use three different formulations (and their corresponding inverses):

For $g : I \to \mathbb{C}$, define
\begin{align}
    \mathcal{F}_I g(\xi) := \sum_{n \in I} g(n) e(-n \xi)
\end{align}
and
\begin{align}
    \mathcal{F}_I^{-1} G(n) := \frac{1}{|I|} \sum_{\xi \in \mathbb{Z}/|I|} G(\xi) e(\xi n);
\end{align}

For $g : \mathbb{Z} \to \mathbb{C}$, define
\begin{align}
    \mathcal{F}_{\mathbb{Z}} g(\xi) := \sum_{n \in \mathbb{Z}} g(n) e(-n \xi) 
\end{align}
and
\begin{align}
    \mathcal{F}_{\mathbb{Z}}^{-1} G(n) := \int_{\mathbb{T}} G(\xi) e(\xi n) \ d\xi;
\end{align}

For $g: \mathbb{R} \to \mathbb{C}$, define
\begin{align}
    \mathcal{F}_{\mathbb{R}} g(\xi) := \int_{\mathbb{R}} g(x) e(-x \xi) \ dx 
\end{align}
and
\begin{align}
    \mathcal{F}_{\mathbb{R}}^{-1} G(x) := \int_{\mathbb{R}} G(\xi) e(\xi x) \ d\xi.
\end{align}

\subsection{Gowers Norms and Technology}
\label{sub:gowers_norms}
Let $f \colon \mathbb{Z} \to \mathbb{C}  $ be a finitely supported function on the integers. Set the conjugation-difference operator to be 
\begin{align}\label{e:triangle}
  \Delta_h f(x) \coloneqq  f(x) \overline{f(x+h)}, 
  \qquad x,h\in \mathbb{Z} . 
\end{align}
The basic fact here is 
\begin{align} \label{e;double}
\Bigl\lvert \sum_x f(x)\Bigr\rvert ^2 
= 
\sum_{x,h} \Delta_h f(x). 
\end{align}
The higher order conjugation-difference operator is inductively defined to be 
\begin{align}
   \Delta_{h_1,\dots,h_s}f(x) \coloneqq  \Delta_{h_s}(\Delta_{h_1,\dots,h_{s-1}} f)(x), 
    \qquad x,h_1, \ldots, h_s\in \mathbb{Z} . 
\end{align}
Then, the $s$th order Gowers norm is 
\begin{align} \label{e;GowersDef}
    \| f \|_{U^s(\mathbb{Z})}^{2^s} \coloneqq  
    \sum_{x,h_1,\dots,h_s} \Delta_{h_1,\dots,h_s}f(x). 
\end{align}
For $s=1$, this is a semi-norm, while higher orders are norms. In particular, for $s=2$, we have 
\begin{align}  \label{e;U2-ell4}
\lVert f \rVert_{U^2(\mathbb{Z} )} ^{4} 
= \int_{\mathbb{T} } \lvert  \mathcal{F}_{\mathbb{Z}} f (\beta )\rvert ^{4}\;d \beta;
\end{align}  
we recall the inductive relationship between norms given by 
\begin{align}\label{e:foliate}
\| f \|_{U^{s+1}(\mathbb{Z})}^{2^{s+1}} = \sum_{h_1,\dots,h_{s-1}} \| \Delta_{h_1,\dots,h_{s-1}} f \|_{U^2(\mathbb{Z})}^4.
\end{align}

For integers $N$, let $[N]=\{1,2, \dots, N\}$. 
Denote the usual expectation by 
\begin{align}
 \mathbb{E}_{n \in [N]} f(n) := \frac{1}{N} \sum_{n \in [N]} f(n)\end{align}
and more generally, for finite $X \subset \mathbb{Z}$, define
\begin{align}
 \mathbb{E}_{n \in X} f(n) := \frac{1}{|X|} \sum_{n \in X} f(n).\end{align}
 
We then define normalized $U^s$ norms as follows. For $f \colon [N] \to \mathbb C$, 
set 
\begin{equation}
    \lVert f \rVert_{U^s([N]) } \coloneqq  
    \frac{\| f \|_{U^s(\mathbb{Z})}}{\lVert \mathbf{1}_{[N]} \rVert_{U^s(\mathbb Z) }} \approx \frac{\| f \|_{U^s(\mathbb{Z})}}{N^{\frac{s+1}{2^s}}}.
\end{equation}

\subsection{Quantifying Convergence}
For a sequence of scalars, $\{ a_N\}_{N \in \mathcal{I}}$, define the \emph{$r$-variation}, for $0 < r < \infty,$
\begin{align}
    \mathcal{V}^r(a_N : N \in \mathcal{I}) := \sup \big( \sum_i |a_{N_i} - a_{N_{i+1}}|^r \big)^{1/r},
\end{align}
where the supremum runs over all finite increasing subsequences $\{ N_i \} \subset \mathcal{I}$; throughout this paper, we will only be interested in the regime where $r > 2$.

A closely related measurement of oscillation is the \emph{jump-counting function} at altitude $\lambda >0$,
\begin{align}
    N_\lambda(a_N : N \in \mathcal{I}) &:= \sup\{ K : \text{ there exists } N_0 < N_1 < \dots < N_K \in \mathcal{I} \\
    & \qquad \text{ such that } |a_{N_i} - a_{N_{i-1}}| \geq \lambda, \ 1 \leq i \leq K \},
\end{align}
so that we may relate the jump function at altitude $\lambda$ to the $\lambda/2$-covering number of $\{ a_N : N \in \mathcal{I} \}$:
\begin{align}
    &\min \big \{ K : \text{there exists } N_1 < N_2 < \dots N_K : \{ a_N : N \in \mathcal{I} \} \subset \bigcup_{i=1}^K \, a_{N_i} + B(\lambda/2) \big\} \\
    &\leq N_{\lambda}(a_N : N \in \mathcal{I});
\end{align}
we further note the inequalities, valid for each $0 < r < \infty$
\begin{align}
    \sup_{\lambda > 0} \, \lambda^r N_\lambda(a_N : N \in \mathcal{I}) \leq \mathcal{V}^r(a_N : N \in \mathcal{I})^r \lesssim_r \sum_{v \in \mathbb{Z}} 2^{-v r} N_{2^{-v}}(a_N : N \in \mathcal{I}).
\end{align}
For functions defined on a measure space $\{ f_N : N \in \mathcal{I} \}$, we define
\begin{align}
    \mathcal{V}^r( f_N : N \in \mathcal{I})(x) :=     \mathcal{V}^r( f_N(x) : N \in \mathcal{I}),
\end{align}
and
\begin{align}
     N_\lambda(f_N : N \in \mathcal{I})(x) := N_\lambda(f_N(x) : N \in \mathcal{I});
\end{align}
often we will be interested in functions taking values in (finite-dimensional) Hilbert spaces, so we define
\begin{align}
    \mathcal{V}^r(\vec{f}_N(x) : N \in \mathcal{I}) := \sup \big( \sum_i \|\vec{f}_{N_i}(x) - \vec{f}_{N_{i+1}}(x)\|_{\mathcal{H}}^r \big)^{1/r},
\end{align}
where again the supremum is over all finite increasing subsequences, and
\begin{align}
    N_\lambda(\vec{f}_N(x) : N \in \mathcal{I}) &:= \sup\{ K : \text{ there exists } N_0 < N_1 < \dots < N_K \in \mathcal{I} \\
    & \qquad \text{ such that } \|\vec{f}_{N_i}(x) - \vec{f}_{N_{i-1}}(x)\|_{\mathcal{H}} \geq \lambda, \ 1 \leq i \leq K \}.
\end{align}

Note that by Minkowski's inequality, if
\begin{align}
    \vec{f}_N := (f_{1,N},\dots,f_{K,N}),
\end{align}
then whenever $r \geq 2$, we may pointwise bound
\begin{align}
    \mathcal{V}^r(\vec{f}_N(x) : N \in \mathcal{I}) \leq \| \mathcal{V}^r(f_{k,N}(x) : N \in \mathcal{I}) \|_{\ell^2(k \in [K])}.
\end{align}

\subsection{Organizing Parameters}\label{ss:param}
There are two principal parameters which will drive our analysis:
\begin{itemize}
    \item $Q$, which measures a degree of rationality; and
    \item $\delta$, which provides spectral information about one of our functions.
\end{itemize}

For the bulk of the paper, we will be interested in the regime where
\begin{align}
    Q \geq \delta^{-1/1000},
\end{align}
but will address the opposite situation in the final section, \S \ref{s:remainder}, where we will re-define some of the below parameters accordingly. But, unless otherwise stated, we will use the following conventions throughout:

We let $0 < \kappa \ll 2^{-1000}$ be a small constant, and $\delta_0$ a sufficiently small parameter, which we will specify in \S \ref{s:trans} below. We then define
\begin{align}
    t_0 := t_0(Q) := \delta_0^{\kappa} Q^{-\kappa},
\end{align}
and, for general $\delta \leq \delta_0$, define
\begin{align}
    t := t(\delta,Q) := \delta^{\kappa} Q^{-\kappa}.
\end{align}

We set
\begin{align}\label{e:R00}
    R := R(\delta,Q) := \delta^{-100\kappa } Q^{100 \kappa}.
\end{align}
Next, we choose prime numbers
\begin{align}\label{e:DELTA}
    \Delta_0 \approx t_0^{-10},
    \; \; \; 
    \Delta \approx t^{-10},
\end{align}
noting that consecutive primes in this range differ by multiplicative factors on the order of $1+o(t^3)$ by \cite{BHP}. We then choose
\begin{align}\label{e:K0size}
    K_0 := K_0(\delta,Q) \approx \delta^{-12\kappa}Q^{12\kappa}
\end{align}
so that
\begin{align}
    \Delta \mid K_0,
\end{align}
and will ultimately restrict all times considered to be of the form
\[
    K_0 2^{\mathbb{N}}.
\]
We set
\begin{align}\label{e:Vdef}
    V := Q^{\rho/10}
\end{align}
where we think of $2^{100} \kappa \ll \rho \ll 2^{-800}$ as being much larger than $\kappa$ but still extremely small.

And, we let
\begin{align}\label{e:lambda}
    \overline{\lambda} &:= \overline{\lambda}(\delta,Q) := \min\{ K_0 Q, V \delta^{-2} \}
\end{align}
so that
\begin{align}\label{e:lambdabound}
    \overline{\lambda} \lesssim \min\{ \delta^{-12 \kappa} Q^{1+12 \kappa}, \delta^{-2} Q^{\rho/10}  \}.
\end{align}

Finally, every factor of $\epsilon$ appearing in this paper can be assumed to be on the order of $\kappa^{10}$ (which is certainly not sharp), and we define 
\[ \tilde{\epsilon} \ll \epsilon^{10} \]
to be the smallest non-infinitesimal parameter appearing throughout.

The following parameter inequality will be used repeatedly:
\begin{align}\label{e:paramhier}
    ( t^{-1} \Delta R V )^{2^{500}} \le Q^{2^{-100}}.
\end{align}
In particular, any power of $t^{-1}$, $\Delta$, $R$, and $V$ on the order of $2^{500}$ may be (generously) absorbed into a small power of $Q$.

\section{Properties and Examples of Admissible Weights}\label{s:adm}
The primary goal of this section is to describe three classes of weights to which Theorem \ref{t:main} applies, namely \eqref{e:vonMangoldt}, \eqref{e:divisor}, and \eqref{e:sumofsquares}.

We begin, however, by outlining some elementary properties of admissible weights: moment and $U^3$ estimates.

\subsection{Properties of Admissible Weights}
We begin with a straightforward moment computation.
\begin{lemma} \label{heath-brown-moments}
    The following moment estimate holds whenever $N \gg Q^{10k}$ (say):
     \[ \mathbb{E}_{n \in [N]} |w_Q(n)|^{2k} \lesssim \mathbf{S}_Q^{2k} Q^{2k } (\log Q)^{4^k} \] for each $k \geq 1$. In particular
     \begin{align}
\mathbb{E}_{n \in [N]} |w_Q(n)|^{2k} \lesssim Q^{o_k(1)} 
     \end{align}
     for each admissible weight, and each $k \geq 1$.
\end{lemma}
\begin{proof}
    We just expand the left hand side and apply \cite[Theorem 1.1]{GM} for the case where $k \geq 2$, and \cite[Theorem 1.2]{CK} when $k=1$:
    \begin{align}
        & \sum_{\theta_1,\dots,\theta_k,\tau_1,\dots,\tau_k} \prod_{i=1}^k S(\theta_i) \overline{S(\tau_i)} \EE_{n \in [N]} e(\sum_{i= 1}^k (\theta_i - \tau_i) n ) \\
        & \lesssim N^{-1/10} + \mathbf{S}_Q^{2k} \cdot \mathbb{E}_{n \in [N]} |\sum_{a/q \in \Gamma_Q} e(na/q)|^{2k} \\
        & \lesssim \mathbf{S}_Q^{2k} Q^{2k} (\log Q)^{4^k}.
    \end{align}    
\end{proof}

We will also record the following $U^3$ statistic of each $w_Q$; this result is similar to \cite[Proposition 2.1]{F+}, in which case the role of $S(a/q)$ was given by 
\[ \frac{\mu(q)}{\phi(q)},\]
which in particular allowed us to restrict
\begin{align}
\Gamma_Q \longrightarrow \Gamma_Q \cap \{ a/q : \sup_{p \text{ prime}} v_p(q) \leq 1 \},
\end{align}
where $v_p$ denote the $p$-adic valuations, namely
\begin{align}\label{e:vp} v_p(n) := \max\{ k : p^k | n \}, \; \; \; v_p(\frac{a}{q}) := v_p(a) - v_p(q).
\end{align}

Specifically, \cite[Proposition 2.1]{F+} can be recast as follows.
\begin{proposition}\label{p:U3F+}
    Suppose that $N \geq Q^{100}$, and that 
    \[ |S(a/q)| \lesssim q^{o(1)-1} \cdot \mathbf{1}_{q \text{ square-free}}.\] Then
    \begin{align}
        \| \sum_{a/q \in \Gamma_Q} S(a/q) e(a/qn) \|_{U^3([N])} \lesssim Q^{o(1)-3/8}.
    \end{align}
\end{proposition}

We use this proposition to derive the following extension, albeit without the same quantitative savings. Specifically, we have the following.
\begin{lemma}\label{p:U3est}
    Suppose that $N \geq Q^{100}$, and that $w$ is admissible. Then
    \begin{align}
        \| w_Q \|_{U^3([N])} \lesssim Q^{o(1)-3/52}.
    \end{align}
\end{lemma}

Below, we will decompose each $q = uv$, where throughout, $v$ will be square-free and $(u,v) =1$; we will call $u$ \emph{powerful}, decompose $u = a^2 b^3$, and observe
\begin{align}\label{e:powercount}
    |\{u \leq X, \ u \text{ powerful}\}| &\leq \sum_{b \leq X^{1/3}} |\{ a^2 \leq X/b^3 \}| \\
    & \lesssim X^{1/2} \sum_{b \geq 1} b^{-3/2} \lesssim X^{1/2}.
\end{align}

The point of making such a decomposition is to reduce matters to Proposition \ref{p:U3est} above: the contribution with large powerful parts will be discarded by crude $L^2$ considerations, while small powerful parts allow a reduction to the squarefree case after restricting to residue classes $\mod u$.

With this in mind, define
\begin{align}
    F_u(n) := \sum_{v: \, Q/2 < uv \leq Q} \big( \sum_{(a,uv) = 1} S(a/uv) e(an/uv) \big) 
\end{align}
and for $U \leq Q$, separate
\begin{align}
F_{>U} := \sum_{u > U \text{ powerful}} F_u, \; \; \; F_{\leq U} := \sum_{u \leq U \text{ powerful}} F_u
\end{align}
according to the relative size of our powerful $u$.

Then, Lemma \ref{p:U3est} will follow from the following estimates and an optimization in $U$:
\begin{align}\label{e:highlowest}
    \| F_{> U} \|_{U^3([N])} \lesssim Q^{o(1)} U^{-1/4}, \; \; \; \| F_{\leq U} \|_{U^3([N])} \lesssim Q^{o(1) - 3/8} U^{11/8}.
\end{align}

\medskip

We turn to the proof.

\begin{proof}[Proof of Lemma \ref{p:U3est}]
The estimate for $F_{>U}$ is of a simpler nature, and goes through the embedding of $U^3 \hookrightarrow L^2$. To see this, it suffices to count 
\begin{align}\label{e:countfreq}
    |\{ a/q \in \Gamma_Q : a/q \text{ contributes to } F_{>U} \}| \lesssim Q^2 U^{-1/2},
\end{align}
as the (normalized) $L^2$ estimate for $F_{>U}$ follows from the fact that distinct frequencies in \eqref{e:countfreq} are separated by $\gtrsim Q^{-2}$, and the large size of $N$. To show \eqref{e:countfreq}, we just observe that for each $u$, maintaining our restrictions on $v$, we may bound
    \begin{align}
        |\{ a/uv : Q/2u < v \leq Q/u \}| \leq \frac{Q^2}{u},
    \end{align}
    as there are at most $Q$ choices for $a$, so the estimate follows from \eqref{e:powercount} and a dyadic decomposition.

So, we will have established \eqref{e:highlowest} once we obtain the estimate
\begin{align}\label{e:lowgoal}
 \| F_{\leq U} \|_{U^3([N])} \lesssim Q^{o(1) - 3/8} U^{11/8};
\end{align}
we do so by reducing to Proposition \ref{p:U3F+}:

Fix a powerful $u \leq U$, and define
\begin{align}
    G_{u,r}(m) := F_u(um+r),
\end{align}
so that we may express
\begin{align}
    G_{u,r}(m) = \sum_{Q/2u < v \leq Q/u} \big( \sum_{(b,v) = 1} T_{u,r}(b/v) e(bm/v) \big)
\end{align}
where
\begin{align}
    T_{u,r}(b/v) := \sum_{\substack{ a \equiv b \mod v, \\ a \leq uv, \ (a,uv) = 1}} S(a/uv) e(ar/uv) = O(u Q^{o(1)-1}),
\end{align}
with the final estimate follows from the fact that for each fixed $b \mod v$, 
\begin{align}
    |\{ (a,uv) = 1 : a \equiv b \mod v\}| \leq u,
\end{align}
as the congruence $a \equiv b \mod v$ gives exactly $u$ residue classes $\mod uv$.

In particular, each $F_u$ is decomposed into arithmetic progressions $\mod u$, and on each progression the frequency set collapses to denominators $v$ which are square-free: by Proposition \ref{p:U3F+}, we may bound
\begin{align}
    \| G_{u,r} \|_{U^3([N/u])} \lesssim (Q/u)^{o(1)-3/8}.
\end{align}
Consequently, if we decompose $F_u$ into
\begin{align}
    F_u(n) = \sum_{r \leq u} f_{u,r}(n), \; \; \; f_{u,r}(n) := F_u(n) \cdot \mathbf{1}_{n \equiv r \mod u},
\end{align}
so that
\begin{align}
    \| F_u \|_{U^3([N])} \leq \sum_{r \leq u} \| f_{u,r} \|_{U^3([N])}
\end{align}
by the triangle inequality for the $U^3$ norm, we may estimate
\begin{align}\label{e:furU3}
    \| f_{u,r} \|_{U^3([N])} \lesssim u^{-1/2} \|G_{u,r} \|_{U^3([N/u])} \lesssim Q^{o(1)-3/8} u^{-1/8},
\end{align}
by expanding out the definition of the $U^3$ norm, and noting that 
\begin{align}
    \prod_{\omega \in \{0,1\}^3} f_{u,r}(x + \omega \cdot h) = 0 
\end{align}
unless $x \equiv r \mod u, \ h_i \equiv 0 \mod u$. This yields a bound,
\begin{align}
    \| F_u \|_{U^3([N])} \lesssim Q^{o(1) -3/8} u^{7/8},
\end{align}
and so \eqref{e:lowgoal} follows from a sum over $u \leq U$ and a final application of \eqref{e:powercount}.
\end{proof}

With this in mind, we address each weight in turn.
\subsection{The von Mangoldt Function}
In the case when $w(n) = \Lambda(n)$, the normalization condition is satisfied by the Prime Number Theorem, and the decay of exponential sums follows from the Siegel-Walfisz Theorem, see e.g.\ \cite[Corollary 5.29]{IK}, the computation involving Ramanujan sums,
\begin{align}
    \sum_{n : (n, q) = 1} e({n}/{q}) = \mu(q),
\end{align}
and well-known estimates on the totient function. Finally, the Heath-Brown-type approximation appears as \cite[Proposition 3.1]{F+}.

\subsection{The Divisor Function}
The divisor function is not, strictly speaking, handled by Theorem \ref{t:main}, as 
\begin{align}
\sum_{n \leq N} \tau(n) \approx N \log N,
\end{align}
so in particular $\tau$ is \emph{not} normalized. More precisely, by the Voronoi summation formula, one has
\begin{align} \label{voronoi-summation-consequence}
    \sum_{\substack{n \le N \\ n \equiv a \mod q}} \tau(n) = \frac{N}{q} \sum_{d\mid q} \frac{c_d(a)}{d} \Bigl( \log N + 2\gamma - 1 - 2\log d \Bigr) + O(N^{2/3}),
\end{align}
where $\gamma$ is Euler's constant, and $c_d$ is a Ramanujan sum, namely
\begin{align}\label{e:ram}
c_d(a) := \sum_{(n,d) = 1} e(an/d).
\end{align}

On the other hand, after some elementary manipulations, we can reduce our analysis of $\tau$ at scale $N$ to weights of the form
\begin{align}
    \sum_{Q \leq N^{o(1)}} \big( \sum_{a/q \in \Gamma_Q} S(a/q) e(an/q) \big), \; \; \; S = S_\tau
\end{align}
to which our analysis applies.

To begin, define
\begin{align}
    S_\tau(a/q;N) \equiv S_\tau(1/q;N) := \frac{1}{q}( \log N + 2\gamma - 1 - 2 \log q)
\end{align}
so that, by Fourier inversion,
\begin{align}
    S_\tau(a/q;N) = \frac{1}{N} \sum_{n \leq N} \tau(n) e(-an/q) + O(N^{-1/3}).
\end{align}

Now, set 
\begin{align}
Q := Q(N) := \exp((\log N)^c), \; \; \; 0 < c \leq 2^{-100},
\end{align}
and define the functions
\begin{align}
    \tau_{L;N}(n) := \sum_{a/q \in \Gamma_L} S_\tau(a/q;N) e(an/q), \; \; \; \tau_L(n) := \sum_{a/q \in \Gamma_L} 1/q \cdot e(an/q),
\end{align}
and
\begin{align}
    \tau_{\leq Q;N} := \sum_{L \leq Q} \tau_{L;N}, \; \; \; \tau_{\leq Q} := \sum_{L \leq Q} \tau_L,
\end{align}
where the sums run over dyadic parameters. Regarding $\mathbf{S}_L$ as deriving from $\tau_L$, we have the bound 
\[ \mathbf{S}_L \lesssim \frac{1}{L}; \]
in particular, our theory applies to ergodic averages of the form
\begin{align}
    \big\{ \frac{1}{N} \sum_{n \leq N} \tau_{\leq Q}(n) T^n f \cdot T^{-n} g : N \big\},
\end{align}
so to prove convergence of ergodic averages weighted by the divisor function, it suffices to prove that
\begin{align}\label{e:U3tauclose}
    \| \tau - \tau_{\leq Q; N} \|_{U^3([N])} \lesssim (\log N)^{-10}
\end{align}
and
\begin{align}\label{e:smallsumtau}
\frac{1}{N \log N} \sum_{n \leq N} |\tau_{\leq Q;N}(n) - \log N \cdot \tau_{\leq Q}(n)| \lesssim (\log N)^{-1/10},
\end{align}
say.

We begin with \eqref{e:U3tauclose}, by showing that
\begin{align}\label{e:divisorapprox}
    \eqref{voronoi-summation-consequence} = \sum_{n \leq N \ n \equiv a \mod q} \tau_{\leq Q;N}(n) + O(N^{2/3}).
\end{align}
To establish the above approximation, we expand the right-hand side of \eqref{e:divisorapprox} and interchange orders of summation
\begin{align}
    &\sum_{\substack{n \le N \\ n \equiv a \mod q}} \sum_{(r,s) = 1, \ s \leq Q} \frac{1}{s} (\log N + 2 \gamma -1 - 2 \log s) e(rn/s) \\
    &=  \sum_{(r,s) = 1, \ s \leq Q} \frac{1}{s} (\log N + 2 \gamma -1 - 2 \log s) \cdot ( \frac{N}{q} \cdot \mathbf{1}_{s | q } \cdot e(ra/s) + O(s) ) \\
    &= \frac{N}{q} \sum_{(r,s) = 1, \ s | q} \frac{1}{s} (\log N + 2 \gamma -1 - 2 \log s) e(ra/s) + O(Q \log N) \\
    & = \frac{N}{q} \sum_{s |q } \frac{1}{s} (\log N + 2 \gamma -1 - 2 \log s) c_s(a) + O( Q \log N).
\end{align}
Since
\begin{align}\label{e:divmoment}
    \frac{1}{N} \sum_{n \leq N} |\tau(n)|^{1024} \lesssim (\log N)^{2^{1024}}
\end{align}
by the Selberg-Delange method, and 
\begin{align}
    \frac{1}{N} \sum_{n \leq N} |\tau_{ \leq Q;N} (n)|^{1024} \lesssim (\log N)^{100} 
\end{align}
(say) by Lemma \ref{heath-brown-moments}, by the inverse theorem for Gowers norms and Leng's algorithm, see \cite[\S 3]{F+}, once we have verified that $\tau, \tau_{\leq Q;N}$ can be written as type I and II sums up to (much) lower order terms, as in \cite[Proposition 5.1]{Leng}, we will have established \eqref{e:U3tauclose}. But, we may express
\begin{align}
    \tau(n) \cdot \mathbf{1}_{(N^{2/3},N]}(n) = \sum_{d \leq N^{1/3}} 2 \cdot \mathbf{1}_{d|n} + \sum_{d,w > N^{1/3}} \mathbf{1}_{dw =n};
\end{align}
and, $\tau_{\leq Q;N}$ has an expression as a type I sum on $[N]$:
\begin{align} \label{type-i-divisor-heath-brown}
	\tau_{\leq Q; N}(n) &= \sum_{q\leq Q} S_{\tau}(1/q; N) c_q(n) \\
    &= \sum_{q\leq Q} S_{\tau}(1/q; N) \sum_{d|n} \mathbf{1}_{d|q} \cdot \mu(q/d)d \\
	&= \sum_{d|n,d\leq Q} \, \sum_{q\leq Q} S_{\tau}(1/q; N) \mathbf{1}_{d|q} \cdot \mu(q/d)d\\
	& = \sum_{d|n,d\leq Q} d \sum_{\substack{q\leq Q/d \\ (d, q) = 1}} S_\tau({1}/{qd}; N) \mu(q) \\
    &=: \sum_{d\leq N^{2/3}}\alpha_d \mathbf{1}_{d|n} \cdot \mathbf{1}_{[N]}(n), 
	\end{align}
where the bound $Q\leq N^{2/3}$ was used and we bound
\begin{align}
	\alpha_d &=  \mathbf{1}_{[Q]}(d) d  \sum_{\substack{q\leq Q/d \\ (d, q) = 1}} S_\tau({1}/{qd}; N) \mu(q) \\
    & \lesssim \log N \cdot \log Q \lesssim (\log N)^2.
\end{align}
The key point in the above analysis was that the exponential sums
\[ S_\tau(1/q;N) \equiv S_\tau(a/q;N) \]
are independent of $a$.

We now turn to \eqref{e:smallsumtau}, and compare
\begin{align}
    &\tau_{\leq Q; N}(n) - \log N \cdot \tau_{\leq Q}(n) \\
    &= \sum_{\substack{(a, q) = 1 \\ q \leq Q}} \frac{1}{q}(2 \gamma - 1 - 2 \log q) e({an}/{q}) \\
    &= \sum_{q \leq Q} \frac{1}{q}(2 \gamma - 1 - 2 \log q) \sum_{d | (q,n )} \mu({q}/{d}) d \\
    &= \sum_{\substack{d \leq Q \\ d|n}} \sum_{\substack{u \leq Q/d}} \frac{\mu(u)}{u}(2 \gamma - 1 - 2 \log d - 2\log u)   \\
    & \lesssim \tau(n;Q) \log^2 Q,
\end{align}
where 
\[ \tau(n;Q) := \sum_{d | n, \ d \leq Q} 1\]
is the sum of divisors of $n$ up to $Q$; \eqref{e:smallsumtau} follows:
\begin{align}
    \frac{1}{N \log N} \sum_{n \leq N} |\tau(n;Q)| \log^2 Q \lesssim \frac{(\log Q)^3}{\log N} \ll (\log N)^{-1/10}.
\end{align}

\subsection{The Sum of Two Squares}\label{ss:sumofsquares}
We first remark that $r_2$ is normalized, as 
\begin{align}
    \frac{|\{ a^2 + b^2 \leq N\}|}{\pi N} \to 1;
\end{align}
once again, set
\begin{align}
    Q = Q(N) := \exp( (\log N)^c), \; \; \; 0 < c \leq 2^{-100}.
\end{align}

With this in mind, we begin by expressing
\begin{align}
    r_2(n) := 4(\tau_1(n) - \tau_3(n)), \; \; \; \tau_i(n) := \sum_{
    \substack{d | n \\ d \equiv i \mod 4}} 1,
\end{align}
and set
\begin{align}
    S(a/q) = \lim_N \frac{1}{\pi N} \sum_{n \leq N} r_2(n) e(-na/q) =  G(a/q)^2,
\end{align}
where $G(a/q)$ are quadratic Gauss sums, namely
\begin{align}\label{e:gausssum0}
    G(a/q) :=  \frac{1}{q} \sum_{r \leq q} e(-r^2a/q) ,
    \end{align}
so
\begin{align}\label{e:gausssum}
    G(a/q)^2      = \begin{cases}
        \frac{(-1)^{(q-1)/2}}{q} & \text { if } q \equiv 1 \mod 2 \\
        0 & \text{ if } q \equiv 2 \mod 4 \\
        - 2i \cdot \frac{(-1)^{(a-1)/2}}{q} & \text{ if } q \equiv 0 \mod 4;
    \end{cases}
\end{align}
in particular, once again, $\mathbf{S}_L \lesssim L^{-1}$.

For notational ease, below we set
\begin{align}
    w(n) := \frac{r_2(n)}{\pi},
\end{align}
so that
\begin{align}
    \frac{1}{N}\sum_{n\leq N} w(n) e(na/q) &= G(a/q)^2 + O(\frac{q}{N^{1/2}} + \frac{q^2}{N}) \\
    & = G(a/q)^2 + O(\frac{1}{N^{1/4}})
\end{align}
whenever $q \leq Q \ll N^{1/4}$. In particular, if we define
\begin{align}
    w_L(n) := \sum_{a/q \in \Gamma_L} S(a/q) e(na/q)
\end{align}
and 
\begin{align}
    w_{\leq Q} := \sum_{L \leq Q} w_L
\end{align}
as above, then
\begin{itemize}
    \item $w, \ w_{\leq Q}$ have very similar behavior on arithmetic progressions with gap-size $q \leq N^{1/4}$;
    \item $\frac{1}{N} \sum_{n \leq N} |w(n)|^{1024} \lesssim (\log N)^{2^{1024}}$, see \eqref{e:divmoment}; and
    \item $\frac{1}{N} \sum_{n \leq N} |w_{\leq Q}(n)|^{1024} \lesssim (\log N)^{100}$ by Lemma \ref{heath-brown-moments},
\end{itemize}
so we may again apply Leng's algorithm to deduce
\begin{align}
    \| w - w_{\leq Q} \|_{U^3([N])} \lesssim (\log N)^{-10},
\end{align}
once we have verified that we can express $w, w_{\leq Q}$ as $\log N$-convex combinations of sums that are suitable for \cite[Proposition 5.2]{Leng} and \cite[Proposition 5.4]{Leng}.

To express $w$ appropriately, we proceed as follows:
since
\begin{align}
    \frac{r_2(n)}{4} = \sum_{d|n, \ d \text{ odd}} (-1)^{(d-1)/2},
\end{align}
on $(N^{2/3},N]$ we can express
\begin{align}
    \frac{\pi}{4} w(n) \cdot \mathbf{1}_{(N^{2/3},N]}(n) &= \sum_{k = 0}^{\infty} \Big( \sum_{\substack{d \leq N^{1/3} \\ d \text{ odd}}} (-1)^{(d-1)/2}  \cdot \mathbf{1}_{d | n} \cdot \mathbf{1}_{n \equiv 2^k \bmod 2^{k+2}} \\
    & \qquad \qquad + \sum_{\substack{d, w > N^{1/3} \\ d \text{ odd}}} (-1)^{(d-1)/2}  \cdot \mathbf{1}_{dw = n} \cdot \mathbf{1}_{n \equiv 2^k \bmod 2^{k+2}} \Big) \\
    & \equiv \sum_{k \leq \log N} \Big( \sum_{\substack{d \leq N^{1/3} \\ d \text{ odd}}} (-1)^{(d-1)/2}  \cdot \mathbf{1}_{d | n} \cdot \mathbf{1}_{n \equiv 2^k \bmod 2^{k+2}} \\
    & \qquad \qquad + \sum_{\substack{d, w > N^{1/3} \\ d \text{ odd}}} (-1)^{(d-1)/2}  \cdot \mathbf{1}_{dw = n} \cdot \mathbf{1}_{n \equiv 2^k \bmod 2^{k+2}} \Big);
\end{align}
viewing the indicators 
\[ \mathbf{1}_{n \equiv 2^k \bmod 2^{k+2}}\]
as indicators of arithmetic progressions, these can be absorbed directly into the pertaining (one-bounded) ``test" nilsequences. In particular, the methods of \cite[Lemma 5.1]{Leng}, \cite[Proposition 5.2]{Leng}, and \cite[Proposition 5.4]{Leng} apply to these sums.

It remains to address $w_{\leq Q}$. Splitting denominators into odd $q$ and $q$ that are divisible by $4$ and using \eqref{e:gausssum}, we decompose
\begin{align}
    w_{\leq Q}(n) = w_{\leq Q;1}(n) + w_{\leq Q;2}(n),
\end{align}
where
\begin{align}\label{e:wQ1}
    w_{\leq Q;1}(n) &:= \sum_{\substack{(a,q) = 1, \ q \leq Q \\ q \equiv 1 \bmod 2}} \frac{(-1)^{(q-1)/2}}{q} e(na/q) \\
    & = \sum_{q \leq Q, \ q \equiv 1 \bmod 2} \frac{(-1)^{(q-1)/2}}{q} c_q(n)
\end{align}
and
\begin{align}\label{e:wQ2}
   w_{\leq Q;2}(n) := - \sum_{\substack{(a,q)= 1,\ q \leq Q \\ q \equiv 0 \bmod 4}} \frac{2i}{q} (-1)^{(a-1)/2} e(na/q).
\end{align}
This decomposition reflects the explicit shape of the quadratic Gauss sums in \eqref{e:gausssum}: the odd-denominator piece, $w_{\leq Q;1}$, is independent of $a$, so we can represent it as a type I sum, similar to the way we addressed $\tau_{\leq Q;N}$ above; the remaining piece, $w_{ \leq Q;2}(n)$, requires a slightly more delicate congruence analysis, which we carry out below, dropping the prefactor of $-1$ as inessential.

To begin, we simplify the sum in $\{ a : (a,q) = 1 \}$ as
\begin{align}
    &\sum_{(a, q) = 1} (-1)^{(a-1)/2} e(na/{q}) \\
    &= \sum_{(a, q) = 1} e(\frac{a-1}{4}) e(na/q) \\
    &= -i \sum_{(a, q) = 1} e({a}/{q} \cdot (n + {q}/{4})) \\
    & = - i c_q(n+q/4).
\end{align}

If we foliate 
\[ \{ q \equiv 0 \mod 4, \ q \leq Q \} = \bigcup_{s=2}^{\log_2 Q} \{ q \equiv 2^s \mod 2^{s+1} \},\]
it suffices to address the contribution from $q$ deriving from a single value of $s \leq \log_2 Q \ll \log N$, provided we do so in a uniform way.

With
\begin{align}\label{e:mnq}
m := n + q/{4},
    \end{align}
we begin by expressing
\begin{align}\label{e:mobiusinversion00}
    2 \sum_{\substack{q \equiv 2^s \mod{2^{s+1}} \\ q \leq Q}} \frac{1}{q} c_q(n + {q}/{4}) = 2 \sum_{\substack{q \equiv 2^s \mod{2^{s+1}} \\ q \leq Q}} \frac{1}{q} \sum_{d | (m, q)} \mu(q/d) d,
\end{align}
and factor
\begin{align}\label{e:qq'}
    q = 2^s q', \; \; \; d = 2^j d',
\end{align} 
with $d'|q'$; we can assume that $j \in \{s-1,s\}$, as otherwise $\mu(q/d) = 0$. Thus
\begin{align}\label{e:mobiusapprox0}
    \sum_{d | (m, q)} \mu(q/d) d = \big(\sum_{d' | (m,q') } \mu(q'/d') d' \big) \cdot( 2^s \cdot \mathbf{1}_{2^s | m} - 2^{s-1}  \cdot \mathbf{1}_{2^{s-1} | m}),
\end{align}
and since $m = n + 2^{s-2} q'$, we have
\[ (m,q') = (n,q'),\]
so 
\begin{align}
\eqref{e:mobiusapprox0} = (\sum_{d' | (n,q') } \mu(q'/d') d' ) \cdot( 2^s \cdot  \mathbf{1}_{2^s | m} - 2^{s-1} \cdot  \mathbf{1}_{2^{s-1} | m});
\end{align}
consolidating, and substituting \eqref{e:mnq} and \eqref{e:qq'}, we obtain
\begin{align}
    \eqref{e:mobiusinversion00} = \sum_{q' \leq Q/2^s, \ q' \text{ odd}} \frac{1}{q'} \big(\sum_{d' | (n,q') } \mu(q'/d') d' \big) \cdot( 2\cdot \mathbf{1}_{2^s | n + 2^{s-2} q'} -  \mathbf{1}_{2^{s-1} | n+2^{s-2} q'}).
\end{align}
We now simplify the $2$-adic conditions:

Since $q'$ is odd, 
\[ 2^{s-1} | n + 2^{s-2} q' \iff 2^{s-1} | n + 2^{s-2},\]
so
\[ \mathbf{1}_{2^{s-1} | n+2^{s-2} q'} \equiv \mathbf{1}_{2^{s-1} | n + 2^{s-2}}.\]
If we now analyze $q' \mod 4$, we can express
\begin{align}
    2 \cdot \mathbf{1}_{2^s | n + 2^{s-2} q'} -  \mathbf{1}_{2^{s-1} | n+2^{s-2} q'} = 2A_+(n) + 2A_-(n) - B(n),
\end{align}
where
\begin{align}
    A_{\pm }(n) := \mathbf{1}_{n \pm 2^{s-2} \equiv 0 \mod 2^{s}}, \; \; \; B(n) := \mathbf{1}_{2^{s-1} | n + 2^{s-2}},
\end{align}
so that
\begin{align}
    &\eqref{e:mobiusinversion00} = 2 A_+(n) \sum_{q' \leq Q/2^s, \ q' \equiv 1 \mod 4} \frac{1}{q'} \sum_{d' | (n,q') } \mu(q'/d') d' \\
    & \qquad + 2 A_-(n) \sum_{q' \leq Q/2^s, \ q' \equiv 3 \mod 4} \frac{1}{q'} \sum_{d' | (n,q') } \mu(q'/d') d' \\
    & \qquad - B(n)  \sum_{q' \leq Q/2^s, \ q' \text{ odd}} \frac{1}{q'} \sum_{d' | (n,q') } \mu(q'/d') d'.
\end{align}
If we express $q' = d'u$ with $d'|n$ and $d'$ odd, and note that
\begin{align}
    q' \equiv 1 \mod 4 &\iff u \equiv d' \mod 4, \\
    q' \equiv 3 \mod 4 &\iff u \equiv - d' \mod 4,
\end{align}
then
\begin{align}
    &\eqref{e:mobiusinversion00}   = 2 A_+(n) \sum_{d' |n, \ d' \text{ odd}} \  \sum_{\substack{ u \leq Q/(2^sd'), \\ u \equiv d' \mod 4}} \frac{\mu(u)}{u} \\
    & \qquad + 2 A_-(n) \sum_{d' |n, \ d' \text{ odd}} \  \sum_{\substack{u \leq Q/(2^sd'), \\ u \equiv - d' \mod 4}} \frac{\mu(u)}{u} \\
    & \qquad - B(n) \sum_{d' |n, \ d' \text{ odd}} \  \sum_{\substack{u \leq Q/(2^sd'), \\ u \text{ odd}}} \frac{\mu(u)}{u}
\end{align}
is the desired representation of \eqref{e:mobiusinversion00} as a sum of type I sums.

\bigskip

These conclude our three examples; below, we will show that the restriction of these three weights to Piatetski-Shapiro times remain good weights for pointwise convergence of bilinear ergodic averages, as part of a more general phenomenon.

\section{Addressing Sparsity}\label{s:sparse}
Recall our notation
\[ \mathbb{N}_c := \{ \lfloor k^c \rfloor : k \in \mathbb{N} \};\]
the goal of this section is to prove Proposition \ref{t:sparse}, recalled below for convenience.
\begin{proposition*}
    Suppose that for any measure-preserving system $(X,\mu,T)$, whenever $T_1,T_2$ are powers of $T$, for any $f,g \in L^\infty(X)$, the bilinear ergodic averages
    \begin{align}
        \frac{1}{N} \sum_{n \leq N} w(n) T_1^n f \cdot T_2^n g
    \end{align}
   converge $\mu$-almost surely. Then the same is true for
   \begin{align}
        \frac{1}{|\mathbb{N}_c \cap [1,N]|} \sum_{n \leq N, \ n \in \mathbb{N}_c} w(n) T_1^n f  \cdot T_2^n g,
    \end{align}
    whenever $1 \leq c < 7/6$, provided
    \[ \| w \|_{U^3([N])} \lesssim N^{o(1)}. \]
\end{proposition*}

We establish Proposition \ref{t:sparse} by analyzing the Fourier statistics of 
\begin{align}\label{e:1orderdiff}
\{ \Delta_h W_c : h \in [N] \},
\end{align}
see \eqref{e:triangle},
where
\begin{align}
    W_c(n) :=W_{c,N}(n) := ( 1 - c n^{1-1/c} \mathbf{1}_{\mathbb{N}_c}(n) ) \cdot \mathbf{1}_{(N/2,N]}(n).
\end{align}
Indeed, after an elementary reparameterization, see \eqref{e:reparam} below, matters reduce to a $U^3$ analysis of $w W_c$, which reduces to controlling the Fourier transform of \eqref{e:1orderdiff}, via \eqref{e:foliate} and \eqref{e;U2-ell4}. The following lemma is precisely tailored to this reduction.

\begin{lemma}\label{l:tech}
    Suppose $1 < c < 7/6$. Then there exists a set $X \subset [N]$ and two real numbers $\epsilon_2(c) > \epsilon_1(c) > 0$ and some $\epsilon'>0$ so that
    \begin{itemize} 
\item     $\max_{h \in X} \| \mathcal{F}_{\mathbb{Z}}({\Delta_h W_c}) \|_{L^1(\mathbb{T})} \lesssim N^{1/2-\epsilon'};$
    \item $X$ exhausts most of $[N]$, in that 
    \[ \frac{|[N] \smallsetminus X|}{N} \lesssim N^{-\epsilon_2(c) + o(1)};\] and
    \item $\max_{h \in [N]} \| \mathcal{F}_{\mathbb{Z}}({\Delta_h W_c}) \|_{L^1(\mathbb{T})} \lesssim N^{1/2+\epsilon_1(c) + o(1)}$.
    \end{itemize}
    \end{lemma}
In fact, we will prove that one can take 
\[ \epsilon_1(c) = 1 - 1/c.\]

\begin{proof}[Proof of Proposition \ref{t:sparse} Assuming Lemma \ref{l:tech}]

By summation by parts, we may re-parametrize 
    \begin{align}\label{e:reparam}
        w(\lfloor k^c \rfloor) \longrightarrow c n^{1-1/c} \cdot  w(n) \cdot \mathbf{1}_{\mathbb{N}_c}(n),
    \end{align}
as whenever $\{ a_n \}$ is a bounded sequence,
\begin{align}
    \frac{1}{N} \sum_{n \leq N} a_{\lfloor n^c \rfloor} \to L \iff \frac{1}{N} \sum_{n \leq N} c n^{1-1/c} \cdot a_n \cdot \mathbf{1}_{\mathbb{N}_c}(n) \to L. 
\end{align}

Indeed, since 
\[ |\mathbb N_c\cap [1,N]| = N^{1/c} + O(1),\]
the sparse averages
\[
\frac{1}{|\mathbb N_c\cap[1,N]|}\sum_{n\le N,\ n\in \mathbb N_c} a_n
\]
are asymptotically equivalent to
\[
\frac{1}{N}\sum_{m\le N^{1/c}} a_{\lfloor m^c\rfloor},
\]
and hence, by summation by parts, to
\[
\frac{1}{N}\sum_{n\le N} c n^{1-1/c} \cdot a_n \cdot \mathbf{1}_{\mathbb N_c}(n).
\]
    
With this in mind, implicitly restricting all functions as we may to $(N/2,N]$, we show that
\begin{align}
    \| w W_c \|_{U^3([N])} &\approx N^{-1/2} \| w W_c \|_{U^3} \\
    &\lesssim N^{-\epsilon(c)'} \| w \|_{U^3([N])} \lesssim N^{-\epsilon(c)'/2}.
\end{align}
To do so, we use Young's convolution inequality to estimate
\begin{align}
    \|  \Delta_h (wW_c) \|_{U^2}^4 \lesssim \| \mathcal{F}_{\mathbb{Z}}( \Delta_h W_c ) \|_{L^1(\mathbb{T})}^4 \| \Delta_h w \|_{U^2}^4,
\end{align}
and bound, for $p = 1 + o_c(1)$,
\begin{align}
   & N^{-1/2} \cdot (\mathbb{E}_{h \in [N]}     \|  \Delta_h (wW_c) \|_{U^2}^4)^{1/4} \\
   &\leq     N^{-1/2} \cdot (\mathbb{E}_{h \in [N]} \mathbf{1}_{X^c}(h)     \|  \Delta_h (wW_c) \|_{U^2}^4)^{1/4} + N^{-1/2} \cdot (\mathbb{E}_{h \in [N]} \mathbf{1}_{X}(h)     \|  \Delta_h (wW_c) \|_{U^2}^4)^{1/4} \\
    & \leq N^{-\epsilon_2(c)/p} N^{\epsilon_1(c)} ( \mathbb{E}_{h \in [N]} \| \Delta_h w \|_{U^2}^{4p'} )^{1/4p'} + N^{-\epsilon'}\| w \|_{U^3([N])} \\
    & \lesssim N^{1+\epsilon_1(c) -(1+\epsilon_2(c))/p + o(1)} + N^{-\epsilon' + o(1)},
\end{align}
with the final line following from an application of H\"{o}lder's inequality.

This implies by a dyadic decomposition, Calder\'{o}n's Transference Principle \cite{C1}, and Lemma \ref{e:gowersUs} below, that whenever $N$ runs over a lacunary sequence,
\begin{align}
    \sum_N |\mathbb{E}_{n \in [N]} w(n) \big(1 - cn^{1-1/c} \mathbf{1}_{\mathbb{N}_c}(n) \big) T^n f \cdot T^{-n} g|^2 \in L^1(X)
\end{align}
whenever $f,g \in L^{\infty}(X)$, which yields the result.
    \end{proof}

We accordingly devote the remainder of this section to the proof of Lemma \ref{l:tech}.
\subsection{The Trivial Bound: $\epsilon_1(c)$.}
For any weight, $W$, we can always bound
\begin{align}
    \| \mathcal{F}_{\mathbb{Z}}( \Delta_h W ) \|_{L^1(\mathbb{T})} \leq (\sum_n |W(n) W(n+h)|^2)^{1/2}
\end{align}
by Cauchy-Schwarz and Plancherel.

Specializing to $W= W_c$, we bound
\begin{align}
    |W_c| \leq 1 + N^{1-1/c} \mathbf{1}_{\mathbb{N}_c \cap [N]},
\end{align}
so we can bound
\begin{align}
    \| \Delta_h W_c \|_2^2 \lesssim N + N^{2-1/c} + N^{4-4/c} \sum_n \mathbf{1}_{\mathbb{N}_c}(n) \mathbf{1}_{\mathbb{N}_c}(n+h).
\end{align}
The final term is bounded by 
\[ N^{4-4/c} (N^{2/c - 1 + o(1)}) \]
by \cite[Corollary 4.10]{KS}, so simplifying, we see that we can take
\begin{align}
    \epsilon_1(c) = 1-1/c.
\end{align}

\subsection{Importing the Fourier Decomposition of $c n^{1-1/c} \cdot \mathbf{1}_{\mathbb{N}_c}$}

Set $H := N^{2-2/c + \delta'}$, where $0 < \delta' = \delta'(c) \ll 2/c-1$ is sufficiently small.

Then, by \cite[p. 18]{KS}, on $(N/2,N]$ we may decompose,
\begin{align}
c n^{1-1/c} \cdot \mathbf{1}_{\mathbb{N}_c} = 1 + f_{1} + f_{2} + \mathcal{E},
\end{align}
where
\begin{align}
    (\mathcal{F}_{\mathbb{Z}}f_1)(\beta) := - \frac{1}{2\pi i} \sum_{N/2 < m \leq N} e(m \beta) \sum_{1 \leq |h| \leq H} \frac{1}{h}(e(-h(m+1)^{\frac{1}{c}}) - e(-hm^{\frac{1}{c}}))
\end{align}
is the main term, with $|f_1| \lesssim \log N$ pointwise,
\begin{align}
    (\mathcal{F}_{\mathbb{Z}}f_2)(\beta) := - \frac{1}{2\pi i} \sum_{N/2 < m \leq N} e(m \beta) \sum_{|h|> H} \frac{1}{h}(e(-h(m+1)^{\frac{1}{c}}) - e(-hm^{\frac{1}{c}}))
\end{align}
is coming from the tail of the saw-tooth function, and can be bounded by
\begin{align}
    |f_{2}(x)| \lesssim \sum_{u \in \{0,1\}} \min\{1,\frac{1}{N^{2-2/c}\|(x+u)^{1/c}\|} \},
\end{align}
so that
\begin{align}
    \| f_2 \|_{\ell^2}^2 \lesssim \| f_2 \|_{\ell^1} \lesssim N^{2/c-1-\delta'/2} 
\end{align}
by \cite[Lemma 4.8]{KS}, and 
\[ \mathcal{E} = O(N^{1/c-2} \mathbf{1}_{(N/2,N]})
\]
is coming from the Taylor expansion of $t \mapsto (m+t)^{1/c}$, and is an error term. 

Consequently,
\begin{align}
    \Delta_h W_c(n) = -\Delta_h( f_1 + f_2 + \mathcal{E}).
\end{align}
We will make use of the moment bound
\begin{align}
&    |\{ h \in [N] : \|\mathcal{F}_{\mathbb{Z}}( \Delta_h W )\|_{L^2(\mathbb{T})} \geq \lambda \}|  \lesssim \lambda^{-2} \sum_{h \in [N]} \| \Delta_h W \|_{\ell^2}^2 \\
& \qquad \lesssim \lambda^{-2} \sum_{x,h \in [N]} |W(x)|^2 |W(x+h)|^2 = \lambda^{-2} \| W \|_{\ell^2}^4;
\end{align}
specializing to $\lambda = N^{1/2 -\epsilon_0}$, where we think of $0 < \epsilon_0 \ll \delta'$ as being extremely small, we see that we are interested in bounds of the form
\begin{align}
    |\{ h \in [N] : \| \mathcal{F}_{\mathbb{Z}}( \Delta_h W )\|_{L^2(\mathbb{T})} \geq N^{1/2 - \epsilon_0} \}| \lesssim N^{2\epsilon_0 - 1} \| W \|_{\ell^2}^4,
\end{align}
so we are free to ignore any weight with
\begin{align}
    \| W \|_{\ell^2}^2 \lesssim N^{1 -\epsilon_2(c)/2 - \epsilon_0 + o(1)}.
\end{align}
In particular, we can disregard the contribution of $\mathcal{E}$, and whenever
\begin{align}
N^{2/c - 1 - \delta'/2} \ll N^{1 - \epsilon_2(c)/2 - \epsilon_0 + o(1)} \Rightarrow \epsilon_2(c) < 4 - 4/c +\delta'/2
\end{align}
(which is always acceptable since we are taking $\epsilon_1(c) = 1 - 1/c)$
we can disregard the contribution of $f_2$.

So, it remains to prove only
\begin{align}
    |\{ h \in [N] : \| \mathcal{F}_{\mathbb{Z}}( \Delta_h f_1 )\|_{L^2(\mathbb{T})} \geq N^{1/2 - \epsilon_0} \}| \ll N^{1/c - \epsilon'},
\end{align}
which will follow from the moment estimate:
\begin{align}
    \| f_1 \|_{\ell^2}^2 \lesssim N^{1/2+1/2c - \epsilon'};
\end{align}
below, we allow $\epsilon' = \epsilon'(c)$ to vary from line to line, but each instance will be bounded away from $0$.

If we set
\begin{align}
    A_m := \big| \sum_{1 \leq |h| \leq H} \frac{1}{h} \cdot ( e(-h(m+1)^{1/c}) - e(-h m^{1/c}) ) \big|,
\end{align}
then by orthogonality of phases we just need to prove that
\begin{align}
    \sum_{m \approx N} |A_m|^2 \lesssim N^{1/2+1/2c - \epsilon'}.
\end{align}
In fact, if we split
\begin{align}
    A_{m,K} := \big| \sum_{|h| \approx K} \frac{1}{h} \cdot ( e(-h(m+1)^{1/c}) - e(-h m^{1/c}) ) \big|
\end{align}
then we may bound
\begin{align}
    \sum_{m \approx N} | \sum_{K \leq H} A_{m,K} |^2 \lesssim \log N \sum_{m \approx N} \sum_{K \leq H} |A_{m,K}|^2 
\end{align}
so it suffices to prove that for each fixed $K \leq H$,
\begin{align}
    \sum_{m \approx N} |A_{m,K}|^2 \lesssim N^{1/2+1/2c - \epsilon'}
\end{align}
or
\begin{align}
    \frac{1}{K^2} \sum_{m \approx N} |B_{m,K}|^2 \lesssim N^{1/2+1/2c - \epsilon'}
\end{align}
    where
\[ B_{m,K} := \big| \sum_{|h| \approx K} D_h(m) \big|, \; \; \; D_h(m) := e(-h(m+1)^{1/c}) - e(-h m^{1/c}).\]
The previous reduction can be made to work by summation by parts, since
\[ \frac{1}{h} = \frac{1}{K} \cdot \frac{K}{h} \]
and 
\[ h \mapsto \frac{K}{h} \mathbf{1}_{|h| \approx K}\]
has a total variation norm of $O(1)$. And, by conceding a factor of $2$, we will only assume that we are summing over positive $h \approx K$.

Note that by the Lipschitz nature of the exponential and Taylor expansion, we may bound
\[ |D_h(m)| \lesssim \min\{1, |h| N^{1/c-1} \},\]
so 
\begin{align}\label{e:diag}
    \sum_{m \approx N} |D_h(m)|^2 \ll N \min \{ 1, |h|^2 N^{2/c-2} \}.
\end{align}

With this in mind, we express
\begin{align}
    \sum_{m \approx N} |B_{m,K}|^2 = \sum_{h_1,h_2 \approx K } \Big( \sum_{m \approx N} D_{h_1}(m) \overline{D_{h_2}(m)} \Big).
\end{align}

We can address the diagonal term using \eqref{e:diag},
\begin{align}
    \frac{1}{K^2} \sum_{h \approx K} \sum_{m \approx N} |D_h(m)|^2 \ll \min \{ N/K, K N^{2/c-1}\} \lesssim N^{1/2+1/2c -\epsilon'},
\end{align}
by taking the geometric mean of the minimum. So, we don't have to worry about the diagonal contribution.

Our task, therefore, is to bound
\begin{align}
    \frac{1}{K^2} \sum_{h_1,h_2 \approx K, \ |h_1-h_2| \geq 1 } \Big( \sum_{m \approx N} D_{h_1}(m) \overline{D_{h_2}(m)} \Big) \lesssim N^{1/2 + 1/2c - \epsilon'}.
\end{align}
Set $k := h_1-h_2$; if we expand out 
\[ D_{h_1}(m) \overline{D_{h_2}(m)},\]
we get
\begin{align}
&   \big( e(-h_1(m+1)^{1/c}) - e(-h_1 m^{1/c}) \big) \big( e(h_2(m+1)^{1/c}) - e(h_2 m^{1/c}) \big) \\
& = e(-k (m+1)^{1/c}) - e(-h_1( (m+1)^{1/c} - m^{1/c}) - k m^{1/c}) \\
& \qquad - e(-h_1 (m^{1/c} - (m+1)^{1/c}) -k (m+1)^{1/c}) + e(-k m^{1/c}),
\end{align}
so we need to understand
\begin{align}
    & \frac{1}{K^2} \sum_{h_1 \approx K}  \sum_{1 \leq |k| \lesssim K } \Big| \sum_{m \approx N} e(-k m^{1/c}) \Big| ;
\end{align}
we apply van der Corput. Indeed, the second derivative of the phase is 
\[ k m^{1/c-2} \approx |k| N^{1/c-2}, \]
so by the second derivative test, \cite{vdC},
\begin{align}
    \Big| \sum_{m \approx N} e(-k m^{1/c}) \Big|  &\lesssim N (|k| N^{1/c-2})^{1/2} + (|k| N^{1/c-2})^{-1/2} \\
    & \lesssim |k|^{1/2} N^{1/2c} + N^{1-1/2c} |k|^{-1/2}.
\end{align}
Since $K \leq H \leq N^{2-2/c+\delta'}$, we arrive at the desired bound:
\begin{align}
    \frac{1}{K} \sum_{1  \leq |k| \lesssim K} |k|^{1/2} N^{1/2c} + N^{1-1/2c} |k|^{-1/2} &\lesssim K^{1/2} N^{1/2c} + K^{-1/2} N^{1-1/2c} \\
    & \lesssim N^{1/2 + 1/2c - \epsilon'}.
\end{align}

The proof is complete.

\bigskip

The above sections have been concerned with applications and consequences of Theorem \ref{t:main}. We now switch perspectives, and begin developing and consolidating the toolkit needed to prove our main result.

\section{Analytic Essentials}\label{s:analess}
In this section, we collect a number of analytic tools that will recur throughout the main argument.

\subsection{Elementary Inequalities}
\begin{lemma}[Summation by Parts]
    Suppose that $\{ a_n \}$ are uniformly bounded and 
    \[ \sup_{N \gg \epsilon_1 N_0} |\mathbb{E}_{n \in [N]} a_n|  \ll \epsilon_1;\]
then, whenever $\varphi \in \mathcal{C}^\infty_c([0,1])$ is $\mathcal{C}^1(\mathbb{R})$-normalized, 
\begin{align}
    \sup_{N \gg N_0} |\sum_n \varphi_N(n) a_n| \ll \epsilon_1.
\end{align}
\end{lemma}
\begin{proof}
    Express
    \begin{align}
        \varphi_N(n) := -\int_0^1 \frac{1}{t N} \mathbf{1}_{[0,Nt]}(n) (t \varphi'(t))  \ dt
    \end{align}
so that for any $N \geq N_0$
\begin{align}
    \big| \sum_{n} \varphi_N(n) a_n\big| &= \big| \int_0^1 \big( \frac{1}{tN} \sum_{n \leq Nt} a_n \big) t \varphi'(t) \ dt\big|  \\
    & \leq O(\epsilon_1) + \int_{\epsilon_1}^1 \big|\frac{1}{tN} \sum_{n \leq Nt} a_n \big| \cdot |t| |\varphi'(t)| \ dt \\
    & = O(\epsilon_1).
\end{align}
\end{proof}

The following elementary Lemma will be used crucially below.
\begin{lemma}\label{l:lacred}
    Suppose that $\{ f_N \}$ are a sequence of $1$-bounded functions satisfying
    \begin{align}\label{e:liptime} |f_N-f_M| \lesssim \frac{|N-M|}{N}\end{align}
    whenever $N/2 \leq M \leq N$ are sufficiently large.

Suppose that
\begin{align}
    \sup_{K \approx 1/t, \ K \in \mathbb{N}} \mu(\{ \limsup_{N \in K 2^{\mathbb{N}}} |f_N| \gg t \}) \leq \epsilon_1;
\end{align}
then
\begin{align}
    \mu(\{ \limsup_{N} |f_N| \gg t \}) \lesssim t^{-1} \epsilon_1.
\end{align}
Similarly, if
\begin{align}
    \sup_{t^{-5/2} \leq K \leq t^{-3}, \ K \in \mathbb{N} \text{ prime}} \mu(\{ \limsup_{N \in K 2^{\mathbb{N}}} |f_N| \gg t \}) \leq \epsilon_1,
\end{align}
then
\begin{align}
    \mu(\{ \limsup_{N} |f_N| \gg t \}) \lesssim t^{-3} \epsilon_1.
\end{align}    
\end{lemma}
\begin{proof}
We begin with the first point: enumerate 
    \[ [10/t,200/t] \cap \mathbb{Z} =: \{ K_1 <K_2< \dots < K_L \}, \] and collect the subsequences
\[ \mathcal{I}_i := K_i 2^{\mathbb{N}}, \; \; \;  \mathcal{I} := \bigcup_{i=1}^L \mathcal{I}_i.\]
Then
\begin{align}
    \{ \limsup_N |f_N| \gg t \} \subset \{ \limsup_{N \in \mathcal{I}} |f_N| \gg t \} \subset \bigcup_{i=1}^L \{ \limsup_{N \in \mathcal{I}_i} |f_N| \gg t\},
\end{align}
with the corresponding inequality for measures. To see this, note that for every sufficiently large $N$, there exists some (possibly non-integral) $\alpha \in [20/t,100/t]$ so that we may express
\[ N = 2^k \alpha; \]
if we choose
\[ M = 2^k K\] with $0 \leq K-\alpha \leq 1$ and $K \in [10/t,200/t]$, then
\begin{align}
    |f_M - f_N| \lesssim t,
\end{align}
by \eqref{e:liptime}, from which the result follows.

For the second, by \cite{BHP}, if $\{ p_1<p_2< \dots \}$ is the enumeration of the set of primes, see \eqref{e:primes}, then whenever $p_L \gg t^{-5/2} \gg t^{-40/19}$ 
\begin{align}
    \sup_{l \geq L} \big |\frac{p_{l+1}}{p_l} -1 \big| \ll t,
\end{align}    
from which the result follows by arguing as above.
\end{proof}

\subsection{Fourier Analytic Inequalities}

\begin{lemma}[Sampling]\label{l:sampling}
Suppose that $\phi$ is Schwartz, and consider the function
\begin{align}
    P_N(\beta) := \sum_n \phi_N(n) g(n) e(n \beta).
\end{align}
Then, whenever $\Lambda \subset \mathbb{T}$ is $1/N$-separated
\begin{align}
    \sum_{\theta \in \Lambda} |P_N(\theta)|^2 \lesssim_A 1/N \sum_n (1 + |n|/N)^{-A} |g(n)|^2.
\end{align}
\end{lemma}
\begin{proof}
    For any $A \geq 100$, we may decompose
    \begin{align}
        \phi = \sum_{k \geq 0} 2^{-kA} \phi^{(k)}
    \end{align}
where $\phi^{(k)}$ are smooth and satisfy $|\phi^{(k)}| \lesssim_A \mathbf{1}_{|x| \leq 2^k}$. If we set
\begin{align}
P_N^{(k)}(\beta) := \sum_{n} \phi_N^{(k)}(n) g(n) e(n \beta),
\end{align}
and choose smooth bump functions, $\{ \chi_{k,N} \}$ so that
\begin{align}
   \mathbf{1}_{|x| \leq 2^{k+2} N} \leq  \mathcal{F}_{\mathbb{R}}^{-1} \chi_{k,N} \leq \mathbf{1}_{|x| \leq 2^{k+4} N}
\end{align}
then by Fourier localization/reproducing
\begin{align}
    P_N^{(k)} \equiv \chi_{k,N}*P_N^{(k)},
\end{align}
so, if $\{ \Lambda_k^{(i)} : 1 \leq i \lesssim 2^k\} \subset \Lambda$ are $2^k/N$ separated subsets of $\Lambda$, we may bound
\begin{align}
&    \sum_{\theta \in \Lambda} |P_N^{(k)}(\theta)|^2 \\
& = \sum_{i \lesssim 2^k} \sum_{\theta \in \Lambda_k^{(i)}} |P_N^{(k)}(\theta)|^2 \\
& = \sum_{i \lesssim 2^k} \sum_{\theta \in \Lambda_k^{(i)}} |\int \chi_{k,N}(\theta - t) P_N^{(k)}(t) \ dt|^2 \\
& \lesssim \int \Big( \sum_{i \lesssim 2^k} \sum_{\theta \in \Lambda_k^{(i)}} |\chi_{k,N}(\theta-t)| \Big) |P_{N}^{(k)}(t)|^2 \ dt \\
& \lesssim 2^{2k} N \| P_N^{(k)} \|_{L^2(\mathbb{T})}^2.
\end{align}
In particular
\begin{align}
    &\sum_{\theta \in \Lambda} |P_N(\theta)|^2 \\
    & = \sum_{\theta \in \Lambda} |\sum_{k \geq 0} 2^{-kA} P_N^{(k)}(\theta)|^2 \\
    & \lesssim \sum_{k \geq 0} 2^{-k A/2} \sum_{\theta \in \Lambda} |P_N^{(k)}(\theta)|^2 \\
    & \lesssim N \sum_{k \geq 0} 2^{-k A/3} \| P_N^{(k)} \|_{L^2(\mathbb{T})}^2 \\
    & = \sum_n \Big( \sum_{k \geq 0}  1/N \cdot 2^{-kA/3} |\phi^{(k)}(n/N)|^2 \Big) |g(n)|^2,
\end{align}
from which the result follows, upon relabeling $A/3 \longrightarrow A$. 
\end{proof}

The following corollary immediately presents.
\begin{cor}\label{c:sampling}
In the setting of Lemma \ref{l:sampling}, suppose that 
\[ \Lambda \subset \{ \beta \in \mathbb{T} : | P_N (\beta) | \geq \delta \} \]
is $1/N$-separated and $|g(n)| \leq 1$. Then,
\begin{align}
    |\Lambda| \lesssim \delta^{-2}.
\end{align}
\end{cor}

The next Lemma allows us to use $L^2(\mathbb{T})$-statistics to control the inverse Fourier transform of Fourier multipliers; it will be used in our final section, \S \ref{s:remainder}.

\begin{lemma}[``$H^{1/2}$"-Sobolev Embedding]\label{l:sob}
Suppose that $m \in \mathcal{C}^1(\mathbb{T})$. Then, for any $k \in \mathbb{Z}$,
\begin{align}
    \| m^{\vee} \|_{\ell^1} \lesssim |m^{\vee}(k)| + \big( \| m \|_{L^2(\mathbb{T})} \cdot \| \partial \big( e(k\cdot)m \big) \|_{L^2(\mathbb{T})} \big)^{1/2}.
\end{align}
\end{lemma}
\begin{proof}
By modulation invariance, we can normalize $k = 0$. Then, with $A$ a parameter to be optimized, we just split
\begin{align}
    \sum_{|n| \geq 1} |m^{\vee}(n)| &\leq \sum_{1 \leq |n| \leq A} |m^{\vee}(n)| + \sum_{|n| > A} |m^{\vee}(n) n| \cdot |n|^{-1} \\
    & \lesssim A^{1/2} \|m^{\vee}(n) \|_{\ell^2} + A^{-1/2} \| m^{\vee}(n) n \|_{\ell^2} \\
    & \lesssim A^{1/2} \| m\|_{L^2(\mathbb{T})} + A^{-1/2} \| \partial m \|_{L^2(\mathbb{T})},
\end{align}
from which the result follows by specializing 
\[ A := \frac{\| \partial m \|_{L^2(\mathbb{T})}}{\| m \|_{L^2(\mathbb{T})}}.\]
\end{proof}

\subsection{L\'{e}pingle's Inequality and Consequences}\label{ss:Lep}
The following inequalities all essentially go back to the work of L\'{e}pingle \cite{LE}, who was interested in quantifying convergence in the martingale context; they were imported to the pointwise ergodic theoretic setting 
by Bourgain in \cite{B2}. We summarize the main points, and leave a fuller discussion to e.g.\ \cite[\S 3]{BOOK}, \cite{BK}, or \cite{KZK}.

We begin by recalling the dyadic (reverse) martingale for (finite dimensional) Hilbert-space-valued functions. With $K = O(1)$ a fixed integer, call an interval $K$-dyadic if
\begin{align}
    I \in \{ K \cdot 2^k [n,n+1) : k \geq 0, \ n \in \mathbb{Z} \};
\end{align}
when $K$ is clear from context, we will just refer to $K$-dyadic intervals as \emph{dyadic}, and set
    \[ \big\{ \mathbb{E}_k \vec{f} := \sum_{|I| = K 2^k \text{ dyadic}} \mathbb{E}_I \vec{f} \cdot \mathbf{1}_I : k \geq 0 \big\} \]
    where
\begin{align}\label{e:EI}
    \mathbb{E}_I \vec{f} := \frac{1}{|I|} \sum_{n \in I} \vec{f}(n)
    \end{align}
denotes the average value of $\vec{f} \in \ell^2(\mathcal{H})$ over $I$.

\begin{lemma}\label{l:LEP}
    The following inequalities hold, independent of $\mathcal{H}, \ K$:
    \begin{align}
        \sup_{\lambda > 0} \| \lambda N_\lambda( \mathbb{E}_k\vec{f} : k)^{1/2} \|_{\ell^2} \lesssim \| \vec{f} \|_{\ell^2(\mathcal{H})},
    \end{align}
and for $r > 2$
        \begin{align}
\| \mathcal{V}^r( \mathbb{E}_k\vec{f} : k) \|_{\ell^2} \lesssim \frac{r}{r-2} \| \vec{f} \|_{\ell^2(\mathcal{H})}. 
    \end{align}
\end{lemma}

A convolution form presents as well. 
\begin{cor}\label{c:LEP}
Suppose that $\chi \in \{ \mathbf{1}_{[0,1]}, \phi\}$ where $\int \phi = 1$ is Schwartz and in the unit ball of a suitable semi-norm, and that all our times derive from a $2$-lacunary sequence, $N \in \mathcal{I}$.

Then the following inequalities hold, independent of $\mathcal{H}$ and $\mathcal{I}$:
    \begin{align}
        \sup_{\lambda > 0} \| \lambda N_\lambda( \chi_N*\vec{f}  : N \in \mathcal{I} )^{1/2} \|_{\ell^2} \lesssim \| \vec{f} \|_{\ell^2(\mathcal{H})},
    \end{align}
and for $r > 2$
        \begin{align}
\| \mathcal{V}^r( \chi_N*\vec{f} : N \in \mathcal{I}) \|_{\ell^2} \lesssim         \frac{r}{r-2}  \| \vec{f} \|_{\ell^2(\mathcal{H})}. 
    \end{align}
\end{cor}

 To derive the above from Lemma \ref{l:LEP}, it suffices to prove the $\ell^2$-boundedness of the square functions
\begin{align}
    (\sum_k |\mathbb{E}_k \vec{f} - \frac{1}{N} \sum_{n \leq N} \vec{f}(x-n)|^2)^{1/2}, \; \; \; K=1, \ 2^k \leq N < 2^{k+1},
\end{align}
which can be seen by expanding $f$ into haar coefficients and arguing spatially, see e.g. \cite[\S 3]{BOOK}, and
\begin{align}
    (\sum_k |\sum_n \chi_N(n) \vec{f}(x-n) - \frac{1}{N} \sum_{n \leq N} \vec{f}(x-n)|^2)^{1/2},
\end{align}
whenever $\chi$ is a smooth bump function with $\int \chi = 1$, which follows from a brief argument with the Fourier transform.

We will make use of the following corollary.
\begin{cor}\label{c:sep}
    Suppose that $\Lambda \subset \mathbb{T}$ is a $\kappa-$separated set, 
    \[ \min_{\theta \neq \theta' \in \Lambda} \|\theta - \theta'\|_{\mathbb{T}} \geq \kappa,\]
    and that we are in the regime where $K 2^k \geq A \kappa^{-1}$ and $A \geq |\Lambda|^{100}$.

For $(K-)$dyadic intervals, define
\[ A_I^\theta f(x) := \mathbb{E}_{n \in I} \text{Mod}_{-\theta}f(n) \cdot  \mathbf{1}_I(x), \]
see \eqref{e:EI}
and let
\begin{align*}
\vec{N}_\lambda( A_I \vec{f}(x) : I \ni x)
\end{align*}
be the jump-counting function of $( A_I^\theta f)_{\theta \in \Lambda}$ at altitude $\lambda > 0$ with respect to the norm $\ell^2(\Lambda)$. Then, if $|f| \leq \mathbf{1}_{3P}$
\begin{align*}
\sup_{\lambda > 0} \| \lambda \vec{N}_\lambda( A_I \vec{f} : I \ni x)^{1/2} \|_{\ell^2(P)} \lesssim |P|^{1/2}.
\end{align*}
\end{cor}
\begin{proof}
Set 
\begin{align}
    (\mathcal{F}_{\mathbb{Z}} f_\theta)(\beta) := (\mathcal{F}_{\mathbb{Z}}f)(\beta) \varphi(\kappa^{-1}(\beta - \theta)),
\end{align}
where
\begin{align}
    \mathbf{1}_{[-2,2]} \leq \varphi \leq \mathbf{1}_{[-5,5]}
\end{align}
is smooth.
Then
\begin{align}
   \sup_{\lambda > 0} \| \lambda \vec{N}_\lambda( A_I^\theta f_\theta : I)^{1/2} \|_{\ell^2} \lesssim \|f \|_{\ell^2}
\end{align}
by the corollary to L\'{e}pingle's inequality, Corollary \ref{c:LEP}, so it suffices to address the complementary term. 

Set
\begin{align}
    R_k^\theta :=  \mathbb{E}_k \big( \text{Mod}_{-\theta} (f- f_\theta)\big)
\end{align}
and
\begin{align}
    S_k^\theta :=  \mathbb{E}_k \, \chi*\big( \text{Mod}_{-\theta} (f- f_\theta)\big)
\end{align}
where $\chi$ is a Schwartz function with 
\begin{align}
    \mathbf{1}_{[-\kappa/A^{1/2},\kappa/A^{1/2}]} \leq \mathcal{F}_{\mathbb{Z}} \chi \leq     \mathbf{1}_{[-3\kappa/A^{1/2},3\kappa/A^{1/2}]}.
\end{align}
By convexity and the boundedness of $f$, we note the pointwise bound
\begin{align}\label{e:pointwisecomp}
    \sup_\theta \| R_k^\theta - S_k^\theta \|_{\ell^\infty} \lesssim \frac{A^{1/2}}{K 2^k \kappa};
\end{align}
on the other hand
\begin{align}
    \chi* \text{Mod}_{-\theta} (f- f_\theta) \equiv 0
\end{align}
by an argument with the Fourier transform.

Now, by trivially dominating the jump-counting function,
\begin{align}
\sup_{\lambda > 0} \lambda N_\lambda(a_n)^{1/2} \leq 2 \cdot (\sum_n |a_n|^2)^{1/2},
\end{align}
it suffices to show that 
\begin{align}
    \sum_{\theta \in \Lambda} \sum_{K 2^k \geq A/\kappa} \sum_{x \in P} | R_{k}^\theta(x)|^2  \lesssim \frac{|\Lambda|}{A^{1/10}} |P|,
\end{align}
or -- uniformly in $\theta \in \Lambda$ -- 
\begin{align}
    \sum_{x \in P} \sum_{K 2^k \geq A/\kappa} |R_k^\theta(x)|^2 \lesssim \frac{1}{A^{1/10}} |P|.
\end{align}
But, by \eqref{e:pointwisecomp}, we may bound
\begin{align}
    |R_k^\theta(x)|^2 \lesssim |S_k^\theta(x)|^2 + \frac{A}{(K 2^k)^2 \kappa^2} \lesssim \frac{A}{(K 2^k)^2 \kappa^2},
\end{align}
from which the result follows.
 \end{proof}

\subsection{Gowers Norm Inequalities}
We collect a few standard inequalities about Gowers norms; all of these can be proven by Cauchy-Schwarz and induction on $s \geq 2$.

\begin{lemma}\label{l:gowersL2}
Suppose that $|\phi_I| \lesssim \frac{1}{|I|} \mathbf{1}_I$ is smooth, that $|f| \leq 1$, and that $|I| = N$. Then
    \begin{align}
       \| \sum_{n} \phi_I(n) f(2x-n) w(n-x) g(n) \|_{\ell^2_x(I)}^2 \lesssim \|w \|_{U^3([-N,2N])}^2 \|g \|_{\ell^2(I)}^2.
    \end{align}
\end{lemma}

\begin{lemma}\label{e:gowersUs}
Whenever $f,g$ are $1$-bounded and supported on $I$, an interval of length $\leq CN$,
    \begin{align}
        \| \frac{1}{N} \sum_{n \leq N} f(x-n) g(x+n) w(n) \|_{\ell^{2^s}}^{2^s} \lesssim_s \| w \|_{U^{s+2}([N])}^{2^s} N
    \end{align}
\end{lemma}
We will use the following immediate corollary.
\begin{cor}
    Suppose that $\phi \in \mathcal{C}_c^{\infty}([0,1])$ has $\| \phi \|_{\mathcal{C}^2(\mathbb{R})} \leq 1$. Then 
       \begin{align}
        \| \sum_{n} \phi_N(n) f(x-n) g(x+n) w(n) \|_{\ell^{2^s}}^{2^s} \lesssim_s \| w \|_{U^{s+2}([N])}^{2^s} N.
    \end{align}
\end{cor}
\begin{proof}
By the previous lemma, we just need to show that
    \begin{align}
\| \phi(n/N) w(n) \|_{U^{s+2}([N])} \lesssim \| w(n) \|_{U^{s+2}([N])};
    \end{align}
    but this follows from Fourier inversion and convexity:
    \begin{align}
        \int |(\mathcal{F}_{\mathbb{R}}\phi)(\xi)| \cdot \| e(\xi n/N) w(n) \|_{U^{s+2}([N])} \ d\xi \lesssim \| w \|_{U^{s+2}([N])},
    \end{align}
    since Gowers norms are invariant under multiplication by characters for $s \geq 2$, and we may trivially bound
    \[ \| \mathcal{F}_{\mathbb{R}}\phi \|_{L^1(\mathbb{R})} \lesssim \| \phi \|_{\mathcal{C}^2(\mathbb{R})} \leq 1.\]
\end{proof}

\subsection{Dyadic Grids}\label{ss:dyadicgrids}
We introduce a useful tool in harmonic analysis: shifted dyadic grids. Below, every interval we introduce will have side length $\geq 1$.

A grid is a collection of intervals $\{ I : I \in \mathcal{Q} \}$ so that whenever $I,I' \in \mathcal{Q}$ 
\[ I \cap I' \in \{ I,I',\emptyset\}, \]
(up to null sets). Below, all side-lengths will be of the form $K_0 2^{\mathbb{N}}$.

A standard example is the usual dyadic grid
\[ \{ 2^k \cdot \big( n + [0,1) \big): k \geq 0, \ n \in \mathbb{Z} \}.\]

Throughout the argument, we will need the flexibility to work with many different dyadic grids, so we address the more general construction.

Regarding $K_0 \in \mathbb{N}$ as fixed, define the \emph{shifted dyadic grids} to be
\[ \mathcal{D}_k^{\Delta,L} := \{ K_0 \cdot 2^k \cdot \big( n + L/\Delta + [0,1) \big)  : n\in \mathbb{Z} \},  \]
where
\[ \ L \in [\Delta] \]
and
\begin{align}\label{e:shiftgrid} \mathcal{D}_U^{\Delta,L} := \bigcup_{k \equiv U \mod \Delta-1} \mathcal{D}_k^{\Delta,L}, \end{align}
noting that $2^{\Delta-1} \equiv 1 \mod \Delta$ by Fermat's Little Theorem, since we have chosen $\Delta$ to be prime; we will only be interested in the case where $\Delta | K_0$, see \eqref{e:K0size}.
 
For each $L \in [\Delta]$, define
\begin{align}\label{e:grid0} 
\overline{\mathcal{D}_U^{\Delta,L}} &:= \bigcup_{k \equiv U \mod \Delta-1} 
\overline{\mathcal{D}_k^{\Delta,L}} \notag \\
& \qquad := \bigcup_{k \equiv U \mod \Delta-1} \{ K_0 \cdot 2^k \cdot \big( n + L/\Delta + [0,1/\Delta) \big)  : n \in \mathbb{Z}\}.
\end{align}

Note that for 
\[ x \in \overline{\mathcal{D}_U^{\Delta,L}}, \ x \in I \in \mathcal{D}_U^{\Delta,L} \] 
we have the smoothness property:
\begin{align}\label{e:dyadsmooth} \frac{|\{ I \triangle \big( x + [0, |I|) \big) \}|}{|I|} \lesssim \Delta^{-1}, 
\end{align}
where we use $\triangle$ to denote symmetric difference
\[ A \triangle B := (A \smallsetminus B) \cup (B \smallsetminus A).\]

Often, we will root the grid inside of a parental interval, $P \subset I_0$, thus
\[ \mathcal{D}(P) := \{ I \in \mathcal{D} : I \subset P \}, \; \; \; \overline{\mathcal{D}}(P) := \{ \overline{I} := I \cap \overline{\mathcal{D}} : I \in \mathcal{D}(P) \}. \]
Our estimates will be uniform in each grid, so our arguments will be flexible enough to obtain a union bound at the close.

In what follows, the smoothness property \eqref{e:dyadsmooth} will be crucial; the below lemma will be applied exclusively when $\Delta$ is as in \eqref{e:DELTA}.
\begin{lemma}\label{l:dyadicsmoothness}
    Suppose that $|f| \leq 1$, and that for each interval $I$,
    \begin{align}
        \| g \|_{\ell^2(I)}^2 \lesssim |I|.
    \end{align}
Then whenever $x \in I \in \mathcal{D}, \ x \in \overline{\mathcal{D}}$, whenever $|I| \geq 2^{Q^{1/5}}$, if
\begin{align}
    \phi_I(n) := \frac{1}{|I|} \phi(\frac{n-c_I}{|I|}), \; \; \; I = [c_I,c_I + |I|),
\end{align}
then
\begin{align}
    |\sum_n \phi_I(n) f(2x-n) w_Q(n-x) g(n) - \sum_n \phi_{|I|}(n-x) f(2x-n) w_Q(n-x) g(n)| \lesssim \Delta^{-1/4}.
\end{align}
\end{lemma}
\begin{proof}
Using the Schwartz nature of $\phi_I$, we may just use admissibility, specifically the decay
\begin{align}
    \mathbf{S}_Q \lesssim Q^{o(1) -1},
\end{align}
to estimate the difference by:
    \begin{align}
       |I|^{-1/2} \max_{|J| \leq \Delta^{-1} |I|} (\sum_{n \in J} |w_Q(n)|^2)^{1/2} \ll \Delta^{-1/2} Q^{o(1)}.
    \end{align}
\end{proof}

\section{Arithmetic Essentials}\label{s:AE}
The goal of this section is to develop an arithmetic toolkit that will be used below. Specifically, we will be interested in understanding interactions between
\begin{align}
    m \Lambda + n \Gamma_Q^{(i)}, \; \; \; |m|, |n| \leq 10,
\end{align}
where $\Lambda \subset \mathbb{Z}/M_0$ are finite subsets with
\[ |\Lambda| \lesssim V \delta^{-2}, \]
and
\[ 2^{Q^{1/5}} \leq M_0 \in K_0 2^{\mathbb{N}} \]
is a large parameter, see \eqref{e:Vdef} and \eqref{e:K0size}.

We begin by recording some elementary statistics of $\Gamma_Q^{(i)}, \ \mathcal{Q}_i, \ w_Q^{(i)}$:

\begin{enumerate}
    \item One has the following estimate on $\mathcal{Q}_i$:
    \begin{align} \label{eqn:size-of-Qi}
        \mathcal{Q}_i \leq 2^i \cdot \text{lcm}\{ q : Q/{2^{i+1}} < q \leq Q/2^i \} \leq 2^i \cdot 3^{Q/2^i}
    \end{align}
    which follows from the standard estimate for the least common multiple of the first $Q/2^i$ integers;
    \item The cardinality of $\Gamma_Q^{(i)}$ satisfies the upper bound
    \[ |\Gamma_Q^{(i)}| \leq \frac{Q^2}{2^i}, \]
    as there are $Q/2^i$ values that the denominators of elements in $\Gamma_Q^{(i)}$ can attain, and at most $Q$ values for the numerator;
    \item The weights $w_Q^{(i)}(n)$ obey the moment estimates given in Lemma \ref{heath-brown-moments}. Namely, for each integer $0 \leq i \leq \log_2 Q$, one has that:
    \begin{align} \label{eqn:heath-brown-moments-w_Q^i}
        \mathbb{E}_{n \in [\mathcal{Q}_i]} |w_Q^{(i)}(n)|^{2k} \lesssim Q^{o(1)};
    \end{align}
this can be seen by formally replacing
\begin{align}
    S(a/q) \longrightarrow S(a/q) \cdot \mathbf{1}_{\Gamma_Q^{(i)}}(a/q),
\end{align}
    which is still bounded by $\mathbf{S}_Q$.
\end{enumerate}

The next lemma concerns the statistics between sum/difference sets
\begin{align}
    \Gamma_Q, \ \Gamma_Q^{(i)} \; \; \; \text{ and } \; \; \; \Lambda.
\end{align}
We phrase these in terms of the following counting functions:

For each $0 \leq i \leq \log_2 Q$, and $|m|, |n| \leq 10$, define
\begin{align}
    D^{(i)}_{(m,n);(\Lambda,Q)}(\xi) := |\{ (\theta, a/q) \in \Lambda \times \Gamma_Q^{(i)}  : m\theta + n a/q = \xi \}|
\end{align}
    and
\begin{align}
    D_{(m,n);(\Lambda,Q)}(\xi) := |\{ (\theta, a/q) \in \Lambda \times \Gamma_Q  : m\theta + n a/q = \xi \}|.
\end{align}

\begin{lemma} \label{l:maximal-multiplicity}
The following bounds hold for each $|m|, |n| \leq 10$:
\begin{align}
\| D^{(i)}_{(m,n); (\Lambda,Q)} \|_{L^\infty(\mathbb{T})} \lesssim \min\{ 2^i K_0,|\Lambda| \} \leq \overline{\lambda}
\end{align}
and similarly
\begin{align}
    \| D_{(m,n);(\Lambda,Q)} \|_{L^\infty(\mathbb{T})} \lesssim \overline{\lambda}.
\end{align}
\end{lemma}
\begin{proof}
We begin with the case of $\Gamma_Q^{(i)}$.

Let $\xi \in m\Lambda + n\Gamma_Q^{(i)}$ be arbitrary, and pick one representation
\begin{align}
    \xi = m\theta_1 + n \frac{a_1}{q_1} \in m \Lambda + n \Gamma_Q^{(i)};
\end{align}
if this representation is unique, there is nothing to show, so suppose that
\begin{align}\label{e:xisumset}
    \xi = m \theta + n \frac{a}{q}
\end{align}
is another such representation. Then
\begin{align}
    m(\theta_1 - \theta) = n (\frac{a}{q} - \frac{a_1}{q_1}).
\end{align}
The claim is that 
\begin{align}\label{e:valuationarg}
    105 \cdot 2^i K_0 \big( \frac{a}{q} - \frac{a_1}{q_1} \big) \in \mathbb{Z},
\end{align}
which will yield the result, as this forces each $\frac{a}{q}$ in \eqref{e:xisumset} to satisfy
\begin{align}
    \frac{a}{q} \in \frac{a_1}{q_1} + \frac{1}{105 \cdot 2^i K_0}\mathbb{Z} \mod 1;
\end{align}
the same argument with $Q/2< 2^i \leq Q$ will yield the analogous statement about $D_{(m,n);(\Lambda,Q)}$; and the bound involving $|\Lambda|$ is trivial.

To see \eqref{e:valuationarg}, for each prime $p$, recall the $p$-adic valuation $v_p : \mathbb{Q} \to \mathbb{Z}$, see \eqref{e:vp};
it suffices to prove that for each prime $p$,
\begin{align}
    v_p\big( 2^i K_0 \big( \frac{a}{q} - \frac{a_1}{q_1} \big) \big) \geq - 1 \cdot \mathbf{1}_{3 \leq p \leq 7}.
\end{align}
By construction, 
\[ \min_{a/q \in \Gamma_Q^{(i)}} v_2(2^i \frac{a}{q}) \geq 0,\]
so
\begin{align}
    v_2( 2^i(\frac{a}{q} - \frac{a_1}{q_1})) \geq 0,
\end{align}
and from the identity
\begin{align}
    m(\theta_1 - \theta) = n(\frac{a}{q} - \frac{a_1}{q_1} )
\end{align}
we see that 
\[ v_2( 2^i m(\theta_1 - \theta))  \geq 0.\]
To address the case of odd $p$, since 
\[ \theta_1,\theta \in \frac{1}{K_0} 2^{- \mathbb{N}} \cdot \mathbb{Z}, \]
multiplication by $K_0$ kills the odd denominator, so
\begin{align}
    m 2^i K_0 (\theta_1-\theta) \in \mathbb{Z}.
\end{align}
Consequently, for odd $p$,
\begin{align}
    v_p\big( 2^i K_0 \big( \frac{a}{q} - \frac{a_1}{q_1} \big) \big) = v_p\big( \frac{m}{n} 2^i K_0(\theta_1-\theta) \big) \geq -1 \cdot \mathbf{1}_{3 \leq p \leq 7}.
\end{align}
\end{proof}

To describe the next lemma, we need a little notation: choose a smooth bump function
\begin{align}\label{e:chibump} 
\mathbf{1}_{\Gamma_Q + \Lambda + B(\frac{\Delta^{1/2}}{N})} \leq \chi^{N} \leq \mathbf{1}_{\Gamma_Q + \Lambda + B(\frac{10\Delta^{1/2}}{N})} 
\end{align}
so that\begin{align}
   \sup_{|\alpha| \leq A} N^\alpha |\partial^\alpha \chi^N| \lesssim_{A} 1
\end{align}
for sufficiently large $A$. 

And, we abbreviate
\begin{align}\label{e:FTw_Q}
    \widehat{w_Q^{(i)}}(\beta) := \mathbb{E}_{r\in [\mathcal{Q}_i]} w_Q^{(i)}(r) e(-\beta r).
\end{align}

Finally, we set
\begin{align}\label{e:M_0'}
    M_0' := \begin{cases} 2^{Q^{1/2+2\epsilon}} M_0 & \text{ if } Q^{1/2} \leq 2^i \leq Q \\
    2^{Q^{1+2\epsilon}} M_0 & \text{ if } 2^i < Q^{1/2}.
    \end{cases}
\end{align}

\begin{lemma} \label{variation-bound-finalizing}
    For every $\beta \in \mathbb{T}$ and every $i \leq \log_2 Q$ one has that:
    \begin{align} \label{variation-bound-finalizing-ineqn}
        \chi^{M_0}(\beta) \cdot \sum_{\theta \in \Lambda} |(\mathcal{F}_\mathbb{Z} \varphi_{M_0'})( \mathcal{Q}_i (\beta - \theta) )|^2 \cdot |\widehat{w_Q^{(i)}}(\beta - \theta)|^2 \lesssim \overline{\lambda} Q^{o(1)-2}
    \end{align}
    for any Schwartz function $\varphi$, normalized in a sufficiently high semi-norm.
\end{lemma}
\begin{proof}
    One can restrict summation in \eqref{variation-bound-finalizing-ineqn} to $\theta$ for which
    \begin{align}
        \| \mathcal{Q}_i(\beta - \theta) \| \leq \frac{Q}{M_0'}
    \end{align}
    is satisfied for some $\beta$ in the support of $\chi^{M_0}$, as otherwise each summand in \eqref{variation-bound-finalizing-ineqn} is bounded by $O_A(Q^{-A})$ and naturally 
    \[ \frac{|\Lambda|}{Q^A} \ll \overline{\lambda} Q^{-2}.\]
    For each $\theta$ that remains, pick $\frac{a_\theta}{q_\theta} \in \Gamma_Q^{(i)}$ so that:
    \begin{align}
        \|\beta - \theta - \frac{a_\theta}{q_\theta}\|
    \end{align}
    is the smallest, breaking ties arbitrarily. Therefore, we can replace $|\widehat{w_Q^{(i)}}(\beta - \theta)|$ in \eqref{variation-bound-finalizing-ineqn} by:
    \begin{align}
        |\EE_{r \in [\mathcal{Q}_i]} S \big(\frac{a_\theta}{q_\theta} \big) e \big( (\frac{a_\theta}{q_\theta} + \theta - \beta) r \big)|
    \end{align}
    without losing much, certainly $\ll 2^{-R/10}$ (say), since all the other terms are negligible by the below argument. We now prove that:
    \begin{align} \label{variation-bound-finalizing-ineqn-2}
       & \chi^{M_0}(\beta) \sum_{\theta \in \Lambda} |(\mathcal{F}_\mathbb{Z} \varphi_{ M_0'})(\mathcal{Q}_i(\beta - \theta) )|^2 \cdot \big| \EE_{r \in [\mathcal{Q}_i]} S \big(\frac{a_\theta}{q_\theta} \big) e \big( (\frac{a_\theta}{q_\theta} + \theta - \beta) r \big) \big|^2 \\
       & \lesssim \overline{\lambda} Q^{o(1)-2}.
    \end{align}
Note that the condition \[ \|\mathcal{Q}_i(\beta - \theta) \| \leq \frac{Q}{M_0'} \] implies that there exist $r_\theta \in \Z$ so that 
    \[ |\beta - \theta - \frac{r_\theta}{\mathcal{Q}_i} | \leq \frac{Q}{\mathcal{Q}_i M_0'}.\] Since both $\frac{a_\theta}{q_\theta}$ and $\frac{r_\theta}{\mathcal{Q}_i}$ can be expressed as fractions with denominator $\mathcal{Q}_i$, there exists $k_\theta \in \Z$ so that:
    \begin{align}
        \frac{a_\theta}{q_\theta} - \frac{r_\theta}{\mathcal{Q}_i} = \frac{k_\theta}{\mathcal{Q}_i}.
    \end{align}
    We will show that when $k_\theta \not \equiv 0 \mod {\mathcal{Q}_i}$, the summand corresponding to $\theta$ in \eqref{variation-bound-finalizing-ineqn-2} is smaller than 
    \[ \frac{1}{2^{Q^{1/2}}},\]
    which is negligible, even when summed over $\theta \in \Lambda$.
    Indeed, one computes:
    \begin{align}
        \EE_{r \in [\mathcal{Q}_i]} S \big(\frac{a_\theta}{q_\theta} \big) e \big( (\frac{a_\theta}{q_\theta} + \theta - \beta) r \big) &= \EE_{r \in [\mathcal{Q}_i]} S \big(\frac{a_\theta}{q_\theta} \big) e \big( (\frac{k_\theta}{\mathcal{Q}_i} + \frac{r_\theta}{\mathcal{Q}_i} + \theta - \beta) r \big)
    \end{align}
    Using the Lipschitz nature of the exponential,
    \begin{align}
        \bigg| e \big( (\frac{k_\theta}{\mathcal{Q}_i} + \frac{r_\theta}{\mathcal{Q}_i} + \theta - \beta) r \big) - e \big( \frac{k_\theta r}{\mathcal{Q}_i} \big) \bigg| &\leq 2 \pi \bigg|\frac{r_\theta}{\mathcal{Q}_i} + \theta - \beta \bigg| \mathcal{Q}_i\\
        &\leq \frac{2 \pi Q}{M_0'} \leq \frac{1}{2^{Q^{1/2}}},
    \end{align}
    and the cancellation arising from summing roots of unity, we bound
    \begin{align}
        \EE_{r \in [\mathcal{Q}_i]} S \big(\frac{a_\theta}{q_\theta} \big) e \big( (\frac{k_\theta}{\mathcal{Q}_i} + \frac{r_\theta}{\mathcal{Q}_i} + \theta - \beta) r \big) &= \EE_{r \in [\mathcal{Q}_i]} S \big(\frac{a_\theta}{q_\theta} \big) e \big( \frac{k_\theta r}{\mathcal{Q}_i} \big) + O \big(\frac{1}{2^{Q^{1/2}}} \big) \\
        & = O \big(\frac{1}{2^{Q^{1/2}}} \big).
    \end{align}
    The upshot is that for every $\theta$ that remains in the \eqref{variation-bound-finalizing-ineqn-2} one has that:
    \begin{align}
        \|\beta - \theta - \frac{a_\theta}{q_\theta} \| \leq \frac{Q}{M_0'}.
    \end{align}
    We next observe that
    \begin{align}
        m\Lambda + n\Gamma_Q^{(i)}
    \end{align}
is $\gtrsim \frac{1}{Q^2 M_0}$-separated: if
\begin{align}
    m \theta + n \frac{a}{q} = m \theta' + n \frac{a'}{q'} + \eta, \; \; \; |\eta| \ll  \frac{1}{Q^2 M_0}, 
\end{align}
then because
\begin{align}
    \theta - \theta' \in \mathbb{Z}/M_0
\end{align}
and $\frac{a}{q} - \frac{a'}{q'}$ has denominator $\leq Q^2$, we must have $\eta = 0$. Consequently, there is at most one value 
\[ \xi \in m \Lambda + n \Gamma_Q^{(i)}  \]
in each $\frac{Q}{M_0'}$ neighborhood of $\beta$, and each such value has at most $\overline{\lambda}$ many representations by Lemma \ref{l:maximal-multiplicity} above; the gain of $Q^{o(1) -2}$ just follows from admissibility, namely the bound $\mathbf{S}_Q \lesssim Q^{o(1)-1}$.
\end{proof}

\section{Multi-Frequency Analysis}\label{s:MultiFreq}

In this section, we develop the main analytic tool that we will use below, refined variants of Bourgain's multi-frequency maximal theory \cite{B2}. The set up is as follows:

Suppose that 
\begin{align}
    |\Psi(t)| \lesssim |t|
\end{align}
is $1$-Lipschitz, that $\Lambda \subset \mathbb{T}$ is a finite set of 
\[ |\Lambda| = K\] frequencies with
\begin{align}
    \min_{\theta \neq \theta' \in \Lambda} |\theta - \theta'| \approx 2^{-l_0},
\end{align}
that $\mathbf{r}:\mathbb{Z} \to [-2^{l_0},2^{l_0}]$ is an arbitrary function, and that 
\begin{align}
    \max_{\theta \in \Lambda} |v(\theta)| \leq \mathbf{v} 
\end{align}
is an arbitrary weight. Suppose that $\{ I_k = [0,L_k) \}$ are a collection of intervals with lacunarily increasing lengths,
\begin{align}
    \frac{L_k}{L_{k-1}} \geq \lambda > 1,
\end{align}
and define
\begin{align}
    l_1 := l_0 + 100 \log K.
\end{align}
The object of study will be the operators
\begin{align}
    V^r_\mathbf{r}(x)  := \mathcal{V}^r\big(\sum_{\theta \in \Lambda} e(\theta x) v(\theta) \Psi( \mathbb{E}_{z \in I_k+\mathbf{r}(x)} \text{Mod}_{-\theta} f(x-z) ) : k \gg l_1 \big),
\end{align}
which we will investigate by way of comparison with
\begin{align}
    V^r(x) := \mathcal{V}^r\big(\sum_{\theta \in \Lambda} e(\theta x) v(\theta) \Psi( \mathbb{E}_{z \in I_k} f_\theta(x-z) )  : k \geq l_1 \big),
\end{align}
where $f_\theta$ is defined via its Fourier transform,
\[ (\mathcal{F}_{\mathbb{Z}}{f_\theta})(\beta) := \chi(2^{l_0}\beta) \mathcal{F}_{\mathbb{Z}}{f}(\beta + \theta),\] where $\chi$ is Schwartz and supported inside $[-1/2,1/2]$.

All expectations will be taken with respect to the $z$ variable below.

\begin{proposition}\label{p:mfvarprop}
    The following estimate holds, uniformly in $\mathbf{r}$ and $\Psi$:
    \begin{align}
        \| V^r_{\mathbf{r}} \|_{\ell^2} \lesssim (\frac{r}{r-2})^2 \log^2 |\Lambda| \cdot \mathbf{v} \cdot \|f \|_{\ell^2}.
    \end{align}
\end{proposition}

It will be convenient to have the flexibility to \emph{choose} our shift $\mathbf{r}$; we use the following lemma to do so.

\begin{lemma}\label{l:transfreely}
    The following holds pointwise, uniformly in choice of $\mathbf{r}$:
    \begin{align}
        |V^r_\mathbf{r} - V^r| \lesssim |\Lambda|^{-10} \cdot \mathbf{v} \cdot M_{\text{HL}} f(x)
    \end{align}
\end{lemma}
\begin{proof}
This follows directly from the Lipschitz nature of $\Psi$ and the fact that we are only interested in large scales. First, we compare
    \begin{align}
        |\Psi( \mathbb{E}_{I_k} \text{Mod}_{-\theta} f(x-z) ) - \Psi( \mathbb{E}_{I_k+\mathbf{r}(x)} \text{Mod}_{-\theta} f(x-z) )| \lesssim_\lambda 2^{-k/2} \cdot |\Lambda|^{-20} \cdot M_{\text{HL}} f(x);
    \end{align}
    and similarly
\begin{align}
        &|\Psi( \mathbb{E}_{I_k} \text{Mod}_{-\theta} f(x-z)) - \Psi( \mathbb{E}_{I_k} ((\mathcal{F}_{\mathbb{R}}^{-1} \chi)*\text{Mod}_{-\theta} f)(x-z))| \\
        & \lesssim |\mathbb{E}_{I_k} \text{Mod}_{-\theta} f(x-z) - \mathbb{E}_{I_k} \sum_t \text{Mod}_{-\theta} f(x-t-z) (\mathcal{F}_{\mathbb{R}}^{-1} \chi)(t)| \\
        & \lesssim_\lambda 2^{-k/2} |\Lambda|^{-20} \cdot  M_{\text{HL}} f(x).
    \end{align}
    \end{proof}

Henceforth, all estimates for $V^r$ will transfer to $V^r_{\mathbf{r}}$, and in particular we will have the flexibility to \emph{proscribe} our translation function, $\mathbf{r}$.

Noting that
\begin{align}\label{e:orthog}
    \| (\sum_{\theta \in \Lambda} |f_\theta|^2)^{1/2} \|_{\ell^2} \lesssim \|f \|_{\ell^2},
\end{align}
we adopt a vector-valued perspective below, and define the following three operators, see \S \ref{ss:Lep} above.

\begin{definition}
Define the jump-counting function at altitude $\lambda > 0$:
\begin{align}\label{e:vecJump}
    \vec{N}_\lambda(x) &:= \sup\{ M : \text{there exist } k_0 < k_1 < \dots < k_M : \\
    & \qquad \| \mathbb{E}_{I_{k_i}} f_\theta(x-z) - \mathbb{E}_{I_{k_{i-1}}} f_\theta(x-z) \|_{\ell^2(\Lambda)} \geq \lambda \};
\end{align}
the vector-valued maximal function;
\begin{align}\label{e:vecMax}
    \mathcal{M}(x) := (\sum_{\theta \in \Lambda} \sup_{k \geq l_1} |\mathbb{E}_{I_k} f_\theta(x-z)|^2)^{1/2}
\end{align}
and the vector-valued variational operator
\begin{align}\label{e:vecVr}
\mathcal{V}^r(x) := 
    \sup \big( \sum_{i} \| \mathbb{E}_{I_{k_i}}f_\theta(x-z) - \mathbb{E}_{I_{k_{i-1}}} f_\theta(x-z) \|_{\ell^2(\Lambda)}^r \big)^{1/r}
\end{align}
where the supremum runs over all finite increasing subsequences. 
\end{definition}

Consolidating Lemma \ref{l:LEP} and \eqref{e:orthog}, we obtain the following.
\begin{lemma}
For each $\lambda > 0, \ r > 2$,
\begin{align}
   \frac{r-2}{r} \| \mathcal{V}^r \|_{\ell^2} + \| \mathcal{M} \|_{\ell^2} + \| \lambda \vec{N}_\lambda^{1/2} \|_{\ell^2} \lesssim \|f \|_{\ell^2}.
\end{align}
\end{lemma}

With these preliminaries in mind, we turn to the proof.

\subsection{The Proof of Proposition \ref{p:mfvarprop}}
We work locally, and estimate
\begin{align}
    \sum_{|I| = 2^{l_0}} \| V^r f \|_{\ell^2(I)}^2.
\end{align}

Fix one interval $I$; by arguing as in Lemma \ref{l:transfreely}, we can replace $V^r$ with $V^r_{\mathbf{r}_I}$, where
\begin{align}
    {V}^r_{\mathbf{r}_I}(x) &:= \mathcal{V}^r\big(\sum_{\theta \in \Lambda} e(\theta x) \Psi( \mathbb{E}_{I_k} f_\theta(c_I-z) ) v(\theta) : k \geq l_1 \big),
\end{align}
where $c_I \in I$ is a point to be determined later.

We apply a standard metric chaining argument:

With $c_I$ to be determined later, we set
\begin{align}\label{e-ftheta} \mathcal{X}(I) := \mathcal{X}(c_I) := \{ (\mathbb{E}_{I_k}f_\theta(c_I-z))_{\theta \in \Lambda} : k \geq l_1 \} \subset \ell^2(\Lambda) \end{align}
and for each $u$ so that 
\[ 2^{-u} \leq   \text{diam}(\mathcal{X}(I)) \leq 2 \mathcal{M}(c_I),
\]
see \eqref{e:vecMax}, define $\mathcal{X}_u(I)$ to be a collection of intervals $\{ I_k = I_k(I)\}$ so that
\begin{align}
\mathcal{X}(I) \subset \bigcup_{I_k \in \mathcal{X}_u(I)} B\big( (\mathbb{E}_{I_k} f_\theta (c_I-z))_{\theta \in \Lambda}, 2^{-u} \big),
\end{align}
subject to the constraint that $|\mathcal{X}_u(I)|$ is minimal; the cardinality is essentially the $2^{-u}$-\emph{entropy} of the set, and note that for $u$ in the proscribed range, we may bound
\[ |\mathcal{X}_u(I)| \leq \vec{N}_{2^{-u}}(c_I).\]
Above,
\[ B\big((\mathbb{E}_{I_k} f_\theta (c_I-z))_{\theta \in \Lambda},2^{-u} \big) := \{ (b_{\theta})_{\theta \in \Lambda} : \| b_{\theta} - \mathbb{E}_{I_k} f_{\theta}(c_I-z) \|_{\ell^2(\Lambda)} \leq 2^{-u} \} \]
are balls with respect to the $\ell^2(\Lambda)$-norm.

For each $I_k \in \mathcal{X}_u(c_I)$, define the \emph{predecessor} of $I_k$, $\varrho(I_k)\in \mathcal{X}_{u-1}(c_I)$ to be the smallest interval, $I_k'$, so that 
\begin{align}\label{e-int}
B\big( (\mathbb{E}_{I_k} f_\theta(c_I-z))_{\theta \in \Lambda}, 2^{-u} \big) \cap B \big( 
(\mathbb{E}_{I_k'} f_\theta(c_I-z))_{\theta \in \Lambda}, 2^{1-u} \big) \neq \emptyset.
\end{align}

Then Proposition \ref{p:mfvarprop} will follow directly from applying the following lemma and square-summing over $\{ I \}$.

\begin{lemma}\label{l:localized}
    Whenever $|I| = 2^{l_0}$, for any $\frac{2}{r} < u_0 < 1$, we may bound 
    \begin{align}\label{e:localized}
    \| {V}^r \|_{\ell^2(I)} &\lesssim \mathbf{v} \cdot \frac{1}{1-u_0} \log K \min_{c_I \in I} \, \mathcal{V}^s(c_I) \cdot |I|^{1/2}\\
    & \qquad + \mathbf{v} \cdot |\Lambda|^{-10} \min_{c_I \in I} \, \mathcal{M}(c_I)^{1-u_0} \cdot \mathcal{V}^{ru_0}(c_I)^{u_0} \cdot |I|^{1/2} \\
    & \qquad \qquad + \mathbf{v} \cdot |\Lambda|^{-10} \min_{c_I \in I}  M_{\text{HL}} f(c_I) \cdot |I|^{1/2},
\end{align}
where $2 < s < r$ can be chosen to satisfy 
\[\frac{1}{s-2} \lesssim \frac{\log K}{r-2}.\]
\end{lemma}
\begin{proof}[Proof of Lemma \ref{l:localized}]
As above, we can replace $V^r$ with $V^r_{\mathbf{r}_I}$ for any $\mathbf{r}_I$ we wish; we will make an appropriate choice below.

We apply the metric chaining mechanism to the set $\mathcal{X}(c_I) = \mathcal{X}(I)$, and for $I_k \in \mathcal{X}_u(I)$ we define
\[ \Delta_{I_k} f_\theta(c_I) := \Psi(\mathbb{E}_{I_k} f_\theta(c_I-z)) - \Psi(\mathbb{E}_{{\varrho(I_k)}}f_\theta(c_I-z)) \]
so that
\begin{align}
\| \Delta_{I_k} f_\theta(c_I) \|_{\ell^2(\Lambda)} \lesssim 2^{-u}    
\end{align}
by the Lipschitz nature of $\Psi$; then, by telescoping along predecessors, we bound
\begin{align}
    &\mathcal{V}^r\big( \sum_{\theta \in \Lambda} e(\theta x) v(\theta) \Psi(\mathbb{E}_{I_k} f_\theta(c_I-z)) : k \geq l_1 ) \\
    &\leq \sum_{2^{-u} \leq 2 \mathcal{M}(c_I)} \mathcal{V}^r\big( \sum_{\theta \in \Lambda} e(\theta x) v(\theta) \Delta_{I_k} f_\theta(c_I) : I_k \in \mathcal{X}_u(c_I) \big) \\
    & \lesssim \sum_{2^{-u} \leq 2 \mathcal{M}(c_I)} \big( \sum_{I_k \in \mathcal{X}_u(c_I)} \big| \sum_{\theta \in \Lambda} e(\theta x) v(\theta)  \Delta_{I_k} f_\theta(c_I)|^r \big)^{1/r}.  
\end{align}
For each $u$, we bound
\begin{align}
    &\| (\sum_{I_k \in \mathcal{X}_u(c_I)} \big| \sum_{\theta \in \Lambda} e(\theta x) v(\theta) \Delta_{I_k} f_\theta(c_I) \big|^r)^{1/r} \|_{\ell^2_x(I)} \\
    &\lesssim \mathbf{v} \cdot 2^{-u} \min\{ |\Lambda|^{1/2} |\mathcal{X}_u(c_I)|^{1/r}, |\mathcal{X}_u(c_I)|^{1/2} \} \cdot |I|^{1/2} \\
    &  \lesssim \mathbf{v} \cdot 2^{-u} \min\{ |\Lambda|^{1/2}\vec{N}_{2^{-u}}(c_I)^{1/r}, \vec{N}_{2^{-u}}(c_I)^{1/2}\} \cdot |I|^{1/2}:
\end{align} 
the first inequality is just a pointwise estimate, which follows from applying Cauchy Schwarz in the inner sum in $\theta$; for the second, we bound
\begin{align}
    &\| (\sum_{I_k \in \mathcal{X}_u(c_I)} \big| \sum_{\theta \in \Lambda} e(\theta x) v(\theta) \Delta_{I_k} f_\theta(c_I) \big|^r)^{1/r} \|_{\ell^2(I)} \\ &\leq \| (\sum_{I_k \in \mathcal{X}_u(c_I)} \big| \sum_{\theta \in \Lambda} e(\theta x) v(\theta)  \Delta_{I_k} f_\theta(c_I) \big|^2)^{1/2} \|_{\ell^2(I)} \\
    & \lesssim \mathbf{v} \cdot 2^{-u} |\mathcal{X}_u(c_I)|^{1/2} \cdot |I|^{1/2} \\
    & \lesssim \mathbf{v} \cdot 2^{-u} \vec{N}_{2^{-u}}(c_I)^{1/2} \cdot |I|^{1/2},
\end{align}
using the fact that $\{ \theta \}$ are $2^{l_0}$ separated, and the elementary inequality
\begin{align}
    \| \sum_{\theta \in \Lambda} e(\theta x) a_\theta \|_{\ell^2(I)} \lesssim \| \sum_{\theta \in \Lambda} e(\theta x) a_\theta w_I(x) \|_{\ell^2} \lesssim (\sum_{\theta \in \Lambda} |a_\theta|^2)^{1/2} \cdot |I|^{1/2}
\end{align}
where 
\[ \mathbf{1}_{I} \leq w_I \lesssim (1 + \text{dist}(x,I))^{-100}\]
has a Fourier transform supported inside $[-2^{-l_0-2},2^{-l_0-2}]$.

To conclude, with $A=A(u_0) = O(\frac{1}{1-u_0})$, we bound
\begin{align}
    &\sum_{2^{-u} \leq 2 \mathcal{M}(c_I)} 
     2^{-u} \min\{ |\Lambda|^{1/2}\vec{N}_{2^{-u}}(c_I)^{1/r}, \vec{N}_{2^{-u}}(c_I)^{1/2}\} \cdot |I|^{1/2} \\
    &\leq \sum_{ \mathcal{M}(c_I)/|\Lambda|^A \leq 2^{-u} \leq 2\mathcal{M}(c_I)} 2^{-u} \vec{N}_{2^{-u}}(c_I)^{1/s} + |\Lambda|^{1/2} \sum_{2^{-u} \leq \mathcal{M}(c_I)/|\Lambda|^A} 2^{-u(1-u_0)} \big(2^{-u}  \vec{N}_{2^{-u}}(c_I)^{1/ru_0} \big)^{u_0} \\
    & \ \lesssim \frac{1}{1-u_0} \cdot \log K \cdot \mathcal{V}^s(c_I) + |\Lambda|^{-10} \cdot \mathcal{M}(c_I)^{1-u_0} \cdot \mathcal{V}^{ru_0}(c_I)^{u_0}, 
\end{align}
completing the proof, since all operators are smooth at scales $2^{l_0}$.
\end{proof}

\begin{cor}\label{c:varschwartz}
    For any Schwartz function, $\varphi$,
    \begin{align}
        &\| \mathcal{V}^r( \sum_{\theta \in \Lambda} e(\theta x) v(\theta) \varphi_k*\text{Mod}_{-\theta} f(x) : k \geq l_1) \|_{\ell^2} \\
        & \lesssim (\frac{r}{r-2})^2 \log^2 |\Lambda| \cdot \mathbf{v} \cdot \| (1 + |t|) \varphi'(t) \|_{L^1(\mathbb{R})} \| f \|_{\ell^2}.
    \end{align}    
\end{cor}
\begin{proof}
By a dyadic decomposition, it suffices to assume that $\varphi$ is compactly supported; we address the case where $\varphi$ is simply absolutely continuous.

Thus, if we let $m$ denote the Radon-Nikodym derivative of $\varphi$, we may assume without loss of generality that $\varphi(0) = 0$ (since the variation is translation invariant), and that $\varphi$ is supported in $(0,\infty)$. By convexity, we may then express
    \begin{align}
        &\mathcal{V}^r( \sum_{\theta \in \Lambda} e(\theta x) v(\theta) \varphi_k*\text{Mod}_{-\theta} f(x) : k \geq l_1) \\
        & \leq \int \mathcal{V}^r( \sum_{\theta \in \Lambda} e(\theta x) v(\theta) \mathbb{E}_{[0,t2^k]} \text{Mod}_{-\theta} f(x-z) : k \geq l_1) \cdot |t| \ dm(t) \\
        & \leq \sum_{0 \leq j \leq l_0} \sum_{i \leq 2^j} \int_{|t| \approx 2^{-j}} \mathcal{V}^r( \sum_{\theta \in \Lambda_{j,i}} e(\theta x) v(\theta) \mathbb{E}_{[0,t2^k]} \text{Mod}_{-\theta} f(x-z) : k \geq l_1) \cdot |t| \ dm(t) \\
        & + \int_{|t| \geq 1} \mathcal{V}^r( \sum_{\theta \in \Lambda} e(\theta x) v(\theta) \mathbb{E}_{[0,t2^k]} \text{Mod}_{-\theta} f(x-z) : k \geq l_1) \cdot |t| \ dm(t),
    \end{align}
where $\Lambda_{j,i} \subset \Lambda$ are $2^{j-l_0}$ separated. Taking $\ell^2$ norms yields the result.
\end{proof}

\begin{cor}\label{c:mfproj}
    Let $R_k := \Lambda + B(L_k^{-1})$ where $L_k$ are lacunarily increasing as above, and define
    \begin{align}
        \mathcal{F}_{\mathbb{Z}}( \Pi_k f ) := \mathcal{F}_{\mathbb{Z}}{f} \cdot \mathbf{1}_{R_k}.
    \end{align}
Then
\begin{align}
   \| \mathcal{V}^r(\Pi_k f ) \|_{\ell^2} \lesssim (\frac{r}{r-2})^2 \log^2 |\Lambda| \|f \|_{\ell^2}.
\end{align}
\end{cor}
\begin{proof}
By monotone convergence, we may restrict to scales $|k| \leq M$, provided our estimates do not depend on $M$.

By a square function argument and Corollary \ref{c:varschwartz}, we may bound
\begin{align}
   \| \mathcal{V}^r(\Pi_k f : k \geq l_1 ) \|_{\ell^2} \lesssim (\frac{r}{r-2})^2 \log^2 K \|f \|_{\ell^2}.
\end{align}
By the Rademacher-Menshov inequality, for any finite $E \subset \mathbb{Z}$, we may also bound
\begin{align}\label{e:V2}
   \| \mathcal{V}^2(\Pi_k f : k \in E ) \|_{\ell^2} \lesssim \log |E| \|f \|_{\ell^2}.
\end{align}    

So, let 
\begin{align}\label{e:E} E := \{ |k| \leq M : \text{ there exists } \theta \neq \theta' \in \Lambda : 2^{-k - K^{200}} \leq |\theta - \theta'| \leq 2^{-k + K^{200}} \} \end{align}
which has size $|E| \lesssim K^{500}$; we estimate the contribution to the jump counting function coming from times in $E$ using \eqref{e:V2}.

Now, let $I_n := [a_n,b_n]$ be such that for each $k \in I_n$, $R_k$ has $n$ connected components, and $[a_n,b_n] \cap E = \emptyset$. 

Note that since $I_n \cap E = \emptyset$, if for $k \in I_n$ we express $R_k$ as a disjoint union of intervals,
\[ R_k = \bigcup_{l \leq n} (c_l-2^{-k},d_l+2^{-k}) = \bigcup_{l \leq n} \big( \frac{c_l+d_l}{2} + (-2^{-k} - \frac{d_l - c_l}{2}, 2^{-k} + \frac{d_l-c_l}{2}) \big).\]

We majorize
\begin{align}
    \mathcal{V}^r(  \Pi_k f :k) \lesssim \big( \sum_{n \leq K} \mathcal{V}^r( \Pi_k f : k \in I_n)^2 \big)^{1/2} + \mathcal{V}^2(\Pi_{b_n} f : n \leq K),
\end{align}
by \cite{JSW}, 
and bound
\[ \| \mathcal{V}^2(\Pi_{b_n}f : n \leq K) \|_{\ell^2} \lesssim \log K \|f \|_{\ell^2}\]
by \eqref{e:V2}, so we focus on the first term. But, we have the equality
\begin{align}
\mathcal{V}^r( \Pi_k f : k \in I_n) \equiv \mathcal{V}^r( \Pi_k ( \Pi_{a_n} f - \Pi_{b_n} f ) : k \in I_n) 
\end{align}
so summing yields the estimate
\begin{align}
    \| \big( \sum_{n \leq K} \mathcal{V}^r( \Pi_k f : k \in I_n)^2 \big)^{1/2} \|_{\ell^2}^2 &\lesssim (\frac{r}{r-2})^4 \log^4 K \cdot \sum_n \| \Pi_{a_n} f - \Pi_{b_n} f \|_{\ell^2}^2 \\
    & \qquad \lesssim (\frac{r}{r-2})^4 \log^4 K \| f \|_{\ell^2}^2.
\end{align}
\end{proof}

\section{Polynomials ``Split"}\label{s:polysplit}
The goal of this section is to establish certain orthogonality properties of polynomials defined over finite (nested) cyclic subgroups. The situation is as follows:

Suppose $I = \bigcup_\alpha I_\alpha$ is a disjoint union, with $|I_\alpha| \leq |I|/10$ and $|I_\alpha| \in 2^{-k} |I| \in \mathbb{N}$, and that each $|I_\alpha| \geq 2^{Q^{1/5}}$. Then, by orthogonality, we have the exact relationship:
\begin{align}
    \mathbb{E}_{\mathbb{Z}/|I|} |\sum_{n \in I} g(n) e(-n \beta)|^2 &=     \mathbb{E}_{\mathbb{Z}/|I|} |\sum_\alpha \sum_{n \in I_\alpha} g(n) e(-n \beta)|^2\\
    &= \sum_\alpha \mathbb{E}_{\mathbb{Z}/|I_\alpha|} |\sum_{n \in I_\alpha} g(n) e(-n \beta)|^2.
\end{align}
Below, we will show that -- up to lower order errors -- a similar orthogonality persists when one localizes the summation in a way that respects the uncertainty principle; the following Lemma is a discrete analogue of \cite[Lemma 5.1]{B3}. For notational ease, we will let
\begin{align}\label{e:duall2}
    \| h \|_{\ell^2(\mathbb{Z}/M)}^2 := \frac{1}{M} \sum_{\xi \in \mathbb{Z}/M} |h(\xi)|^2
\end{align}
denote normalized $\ell^2$ norms.

\begin{lemma}\label{l:polysplit}
Suppose $I = \bigcup_\alpha I_\alpha$ is as above, that $|g(n)| \leq 1$, and that $\Lambda \subset [0,1/10]$. Then

\begin{align}\label{e:uncertaintyprinc}
    &\mathbb{E}_{\mathbb{Z}/|I|} \mathbf{1}_{\Lambda + O(R/|I|)}(\beta) \cdot  |\sum_{n \in I} g(n) e(- n \beta)|^2 \\
    & \qquad \leq \sum_\alpha \mathbb{E}_{\mathbb{Z}/|I_{\alpha}|} \mathbf{1}_{\Lambda + O(R/|I_{\alpha}|)}(\beta) \cdot |\sum_{n \in I_\alpha} g(n) e(- n \beta)|^2 \\
    & \qquad \qquad \qquad  + |I| \big( O_A(R^{-A} |\Lambda|) + O(R^{-1/2}) \big).
\end{align}
\end{lemma}
\begin{proof}
    It suffices to remove the dilation by $R$ on the left hand side,
    \begin{align}
        \mathbf{1}_{\Lambda + O(R/|I|)} \longrightarrow \mathbf{1}_{\Lambda};
    \end{align}
    to do so, we choose 
    \[ \Lambda' := \{ \xi \in \mathbb{Z}/|I| : \text{dist}(\xi,\Lambda) \leq R/|I|\} \]
and then just observe that
\[ \Lambda' + O(R/|I_\alpha|) \subset \Lambda + O(2R/|I_\alpha|) \]
    since $|I_\alpha| \leq |I|/10$, so one can simply replace
    \begin{align}
\Lambda \longrightarrow \Lambda', \; \; \; R \longrightarrow 2R 
\end{align}
and then relabel appropriately.

To estimate \eqref{e:uncertaintyprinc}, we express the square root of the left hand side, for an appropriate $h$ with
\begin{align}
\| h \|_{\ell^2(\mathbb{Z}/|I|)} = 1
\end{align}
that vanishes off $\Lambda$, as
\[ \sum_{n \in I} g(n) (\mathcal{F}_I^{-1} h)(n) = \sum_{\alpha} \sum_{n \in I_\alpha} g(n) (\mathcal{F}_I^{-1} h)(n). \]
Freezing $\alpha$, we let $\varphi_{\alpha}$ be a Schwartz function with \[ \mathbf{1}_{\{ n \in I_\alpha : \text{dist}(n,I_\alpha^c) \geq \gamma |I_\alpha| \}} \leq \varphi_\alpha(n) \leq \mathbf{1}_{I_\alpha}(n) \]
so that
\[ |\mathcal{F}_{\mathbb{R}} \varphi_\alpha(\xi)| \lesssim_A |I_\alpha|(1 +  \gamma |I_\alpha||\xi| )^{-A} \]
pointwise; note that $\mathcal{F}_{\mathbb{Z}} \varphi_\alpha$ satisfies
\[ |\mathcal{F}_{\mathbb{Z}} \varphi_\alpha(\xi)| \lesssim_A |I_\alpha|(1 +  \gamma |I_\alpha|\|\xi\|_{\mathbb{T}} )^{-A} \]
by Poisson summation. Since we are interested in the case where $|\xi| \leq 1/5$, there is no difference between the norm and norm $\mod 1$.

Then
\[ \sum_{I_\alpha} (1 - \varphi_\alpha)(n) g(n) (\mathcal{F}_I^{-1} h)(n) = O(\gamma^{1/2} |I_\alpha|^{1/2} \| (\mathcal{F}_I^{-1} h) \|_{\ell^2(I_\alpha)}), \]
and so a sum over $\alpha$ yields a bound of
\[ \gamma^{1/2} |I|^{1/2} \| h \|_{\ell^2(\mathbb{Z}/|I|)} = \gamma^{1/2} |I|^{1/2},\]
by Cauchy-Schwarz.

We turn to the main contribution. By Fourier inversion in $\mathbb{Z}/|I_\alpha|$,
\begin{align}
&\sum_{n \in I_\alpha} \varphi_\alpha(n) g(n) (\mathcal{F}_I^{-1} h)(n) \\
&= \mathbb{E}_{\beta \in \mathbb{Z}/|I_{\alpha}|} \mathbb{E}_{\xi \in \mathbb{Z}/|I|} h(\xi) (\mathcal{F}_{I_\alpha}g)(\beta) \sum_n \varphi_\alpha(n) e(-(\xi - \beta)n) \\
&  = \mathbb{E}_{\beta \in \mathbb{Z}/|I_{\alpha}|} \big( \mathbb{E}_{\xi \in \mathbb{Z}/|I|} h(\xi) (\mathcal{F}_{\mathbb{Z}} \varphi_\alpha)(\xi - \beta) \big) \cdot (\mathcal{F}_{I_\alpha}g)(\beta) \\
&  = \mathbb{E}_{\beta \in \mathbb{Z}/|I_{\alpha}|} \big( \mathbb{E}_{\xi \in \mathbb{Z}/|I|} h(\xi) (\mathcal{F}_{\mathbb{Z}} \varphi_\alpha)(\xi - \beta) \big) \cdot (\mathcal{F}_{I_\alpha}g)(\beta) \cdot \mathbf{1}_{\Lambda + B(R/|I_\alpha|)}(\beta) \\
& \qquad \qquad + \mathbb{E}_{\beta \in \mathbb{Z}/|I_{\alpha}|} \big( \mathbb{E}_{\xi \in \mathbb{Z}/|I|} h(\xi) (\mathcal{F}_{\mathbb{Z}} \varphi_\alpha)(\xi - \beta) \big) \cdot (\mathcal{F}_{I_\alpha}g)(\beta) \cdot \mathbf{1}_{(\Lambda + B(R/|I_\alpha|))^c}(\beta).
\end{align}
We begin with the main term; by Cauchy-Schwarz
\begin{align}
    &\sum_\alpha |\mathbb{E}_{\beta \in \mathbb{Z}/|I_{\alpha}|} \big( \mathbb{E}_{\xi \in \mathbb{Z}/|I|} h(\xi) (\mathcal{F}_{\mathbb{Z}} \varphi_\alpha)(\xi - \beta) \big) \cdot (\mathcal{F}_{I_\alpha}g)(\beta) \cdot \mathbf{1}_{\Lambda + B(R/|I_\alpha|)}(\beta)|^2 \\
    & \leq \sum_\alpha \mathbb{E}_{\beta \in \mathbb{Z}/|I_{\alpha}|} |(\mathcal{F}_{I_\alpha} g)(\beta) \cdot \mathbf{1}_{\Lambda + B(R/|I_\alpha|)}(\beta)|^2 \times \Sigma,
\end{align}
where
\[ \Sigma := \sum_\alpha \mathbb{E}_{\beta \in \mathbb{Z}/|I_{\alpha}|}| ( \mathbb{E}_{\xi \in \mathbb{Z}/|I|} h(\xi) (\mathcal{F}_{\mathbb{Z}} \varphi_\alpha)(\xi - \beta) \big)|^2.  \]
We show that $\Sigma \leq 1$ by computing, for each $\alpha$,
\begin{align}
 &\mathbb{E}_{\beta \in \mathbb{Z}/|I_{\alpha}|}| ( \mathbb{E}_{\xi \in \mathbb{Z}/|I|} h(\xi) (\mathcal{F}_{\mathbb{Z}} \varphi_\alpha)(\xi - \beta) \big)|^2\\
    & =  \mathbb{E}_{\xi,\zeta \in \mathbb{Z}/|I|} h(\xi) \overline{h}(\zeta) \sum_{n,m} \varphi_\alpha(n) \overline{\varphi_\alpha}(m) e(m \xi - n \zeta) \mathbb{E}_{\beta \in \mathbb{Z}/|I_{\alpha}|} e(-(m-n)\beta) \\
    & = \sum_{n} |\varphi_\alpha(n)|^2 \mathbb{E}_{\xi,\zeta \in \mathbb{Z}/|I|} h(\xi) \overline{h}(\zeta) e((\xi - \zeta) n) \\
    & = \sum_n |\varphi_\alpha(n)|^2 |(\mathcal{F}_{I}^{-1} h)(n)|^2;
\end{align}
summing over $\alpha$ yields an upper bound of 
\[ \Sigma \leq \| h \|_{\ell^2(\mathbb{Z}/|I|)}^2 = 1.\]
In the above argument, we emphasize that $\varphi_\alpha$ is supported inside $I_\alpha$, so for $m,n$ in the support of $\varphi_\alpha$
\begin{align}
    \mathbb{E}_{\beta \in \mathbb{Z}/|I_\alpha|} e(-(m-n) \beta) = \mathbf{1}_{m =n}.
\end{align}

Thus, it suffices to efficiently bound
\[ \mathbb{E}_{\beta \in \mathbb{Z}/|I_{\alpha}|} \big( \mathbb{E}_{\xi \in \mathbb{Z}/|I|} h(\xi) (\mathcal{F}_{\mathbb{Z}} \varphi_\alpha)(\xi - \beta) \big) \cdot (\mathcal{F}_{I_\alpha}g)(\beta) \cdot \mathbf{1}_{(\Lambda + B(R/|I_\alpha|))^c}(\beta) \]
for each $\alpha$.

For each $\beta \notin \Lambda + B(R/|I_\alpha|)$, we bound
\begin{align}
    &|\mathbb{E}_{\xi \in \mathbb{Z}/|I|} h(\xi) (\mathcal{F}_{\mathbb{Z}} \varphi_\alpha)(\xi - \beta)|^2 \\
    &\leq \mathbb{E}_{\xi \in \mathbb{Z}/|I|} \mathbf{1}_{\Lambda}(\xi) \cdot  |I_\alpha|^2 (1 + \gamma |I_\alpha|\| \beta - \xi \|)^{-2A} \\
    & \qquad \lesssim \gamma^{-1} (\gamma R)^{-A} |I_\alpha| \cdot \mathbb{E}_{\xi \in \mathbb{Z}/|I|} \mathbf{1}_{\Lambda}(\xi) \cdot \gamma |I_\alpha| \cdot ( 1 + \gamma |I_\alpha| \|\beta - \xi\| )^{-A}
\end{align}
so 
\begin{align}
&\big| \mathbb{E}_{\beta \in \mathbb{Z}/|I_{\alpha}|} \big( \mathbb{E}_{\xi \in \mathbb{Z}/|I|} h(\xi) (\mathcal{F}_{\mathbb{Z}} \varphi_\alpha)(\xi - \beta) \big) (\mathcal{F}_{I_\alpha} g)(\beta) \cdot \mathbf{1}_{(\Lambda + B(R/|I_\alpha|))^c}(\beta) \big|  \\
&  \leq \| g \|_{\ell^2(I_{\alpha})} \cdot (\gamma^{-1} (\gamma R)^{-A} |I_\alpha|)^{1/2} \cdot  \big( \mathbb{E}_{\beta \in \mathbb{Z}/|I_{\alpha}|} \mathbb{E}_{\xi \in \mathbb{Z}/|I|} \mathbf{1}_{\Lambda}(\xi) \gamma |I_\alpha| \cdot ( 1 + \gamma |I_\alpha| \| \beta - \xi \|)^{-A} \big)^{1/2} 
\\
& \leq (\gamma^{-1} (\gamma R)^{-A})^{1/2} |I_\alpha| \cdot ( \mathbb{E}_{\xi \in \mathbb{Z}/|I|} \mathbf{1}_{\Lambda}(\xi) )^{1/2}
\\
& \leq (\gamma^{-1} (\gamma R)^{-A})^{1/2} |I_\alpha| \cdot (\frac{|\Lambda|}{|I|})^{1/2}.
    \end{align} 
Summing over $\alpha$ therefore yields an upper bound of
\[ R^{-A/8} |I|^{1/2} |\Lambda|^{1/2}\]
upon specializing $\gamma = R^{-1/2}$, completing the proof.
\end{proof}

In the following section, we will use this splitting behavior to develop a stopping time algorithm that will eventually allow us to reduce to the case where we consider interval coming from a single set
\begin{align}
    \mathcal{B}([\Lambda]),
\end{align}
as discussed in the proof overview, see \eqref{e:branches}.

\section{Energy Pigeon-Holing}\label{s:ENERGY}
Our task here is to organize our collection of dyadic intervals, $\mathcal{D}$, into scales and locations on which $g$ ``acts like" a fixed linear combination of characters; we freely discard small pointwise errors, exceptional sets, and exceptional ranges of scales  from our analysis. This will allow us to employ the analysis using the technology developed in \S \ref{s:MultiFreq}.

We proceed by energy pigeon-holing, namely by exploiting orthogonality in phase space; the following elementary lemma allows us to convert statements about ``low energy" -- i.e.\ an $\ell^2$ statistic -- to pointwise control.

\begin{lemma}\label{l:l2topointwise}
    Suppose that $\|g \|_{\ell^2(I)}^2 \leq t^2 |I|$, that $|f| \leq 1$, and that $|I| \geq 2^{Q^{1/5}}$. Then
    \begin{align}
        |\sum_n \phi_I(n) f(2x-n) w_Q(n-x) g(n)| \lesssim t^{1/2}
    \end{align}
\end{lemma}
\begin{proof}
    This follows since $\mathbb{E}_{n \in [N]} |w_Q(n)|^2 \lesssim Q^{\kappa^2/100}$ for $|I| = N$.
\end{proof}

Below, we will view $Q,\delta$ as fixed, noting that whenever $Q \geq \delta^{-1/1000}$, $R \leq Q^{2^{20}\kappa}$, see \eqref{e:R00}; we recall the normalized $\ell^2$-norms \eqref{e:duall2}, and emphasize that each time $M \in K_0 2^{\mathbb{N}}$, where $K_0$ is as in \eqref{e:K0size}, see Lemma \ref{l:lacred}.

\subsection{Trees}
Set $I_0 := [J]$ and let
\[ \mathcal{I} : I_0 \to \mathcal{D}(I_0) \]
be an arbitrary selector (which we secretly think of as a linearizing function for our pertaining suprema), 
where $\mathcal{D}(I_0) := \{ I \in \mathcal{D} : I \subset I_0\}$, and we regard $\mathcal{D}$ as an arbitrary but fixed dyadic grid. 

The structure we introduce will depend on $\mathcal{I}$, but all estimates will be uniform.

For a collection of intervals $E := \{ I \}$, we define the \emph{shadow} of $E$,
\begin{align}
    \text{sh}(E) := \bigcup I.
\end{align}

We will exhibit a partition of $\mathcal{D}(I_0)$ into $V$
many subsets, ``trees,"
\[ \mathcal{D}_j(I_0), \ 0 \leq j < V \]
and an exceptional collection of intervals, $\mathcal{D}_V(I_0)$, so that we can express
\[ \text{sh}(\mathcal{D}_j(I_0)) = \bigcup_{T_j^{\text{max}}} I_j \]
as a disjoint union of maximal dyadic intervals, ``tree tops," so that for each $I_j \in T_j^{\text{max}}$ there exists a unique $I_{j-1} \in T_{j-1}^{\text{max}}$ with $I_j \subset I_{j-1}$; $I_0$ anchors the inductive construction. We use the notation
\[ T_j^{\text{max}}(I_{j-1}) := \{ I \in T_j^{\text{max}} : I \subset I_{j-1} \}. \]

The intervals collected in 
\[ \mathcal{D}_V(I_0)  \]
  will be very localized,
\begin{align}\label{e:loc} |\text{sh}(\mathcal{D}_V(I_0))| = \sum_{T_V^{\text{max}}} |I_V| \lesssim t^{-2} \cdot  V^{-1} \cdot  |I_0|; 
\end{align}
 these intervals will only effect the argument minimally.

\begin{definition}
    A $V$-\emph{forest} consists of disjoint collections of intervals 
    \[ \{ \mathcal{D}_j(I_0) : 0 \leq j\leq V\},\] equipped with nested ``tree tops" $T_j^{\text{max}} \subset \mathcal{D}_j(I_0)$ as above, so that each $I_j \in T_j^{\text{max}}$ comes equipped with a finite set of frequencies, $\Lambda_j(I_j) \subset \mathbb{Z}/|I_j|$, satisfying the following properties:
    \begin{itemize}
    \item The forest exhausts most of $I_0$: \eqref{e:loc} holds;
    \item There are not too many frequencies involved: $|\Lambda_0(I_0)| \lesssim \delta^{-2}$, and in general
      \[ \sup_{j < V} \sup_{I_j \in T_{j}^{\text{max}}} |\Lambda_j(I_j)| \lesssim V  \cdot  \delta^{-2}; \]
      \item The frequencies respect nesting: if $I_j \in T_j^{\text{max}}(I_{j-1})$ then we have the containment
      \[ \Lambda_{j-1}(I_{j-1}) \subset \Lambda_j(I_j);\]
    \item We have the bound
\[ \sup_{I \in \mathcal{D}_j(I_j), \ I_j \in T_j^{\text{max}}} \| \sum_{\xi \in \mathbb{Z}/|I| \smallsetminus \big( \Lambda_j(I_j) + B(R/|I|) \big)} \Psi_\delta\big( \frac{\mathcal{F}_Ig(\xi)}{|I|} \big)  e(n \xi) \|_{\ell^2(\mathbb{Z}/|I|)} \ll t, \]
see \eqref{e:Psidelta}.
    \end{itemize}
    \end{definition}

\begin{definition}
Given a collection of frequencies $\Lambda \subset \mathbb{T}$, we say that an interval, $I$, is \emph{$(R,t)$-localized} with respect to $\Lambda$ if
\begin{align}
\| \sum_{\xi \in \mathbb{Z}/|I| \smallsetminus \Lambda_R(I)} \Psi_\delta\big( \frac{\mathcal{F}_I {g}(\xi)}{|I|} \big) e(\xi x) \|_{\ell^2(I)}^2 \ll t^2 |I|
\end{align}
where
\[ \Lambda_R(I) := \{ \xi \in \mathbb{Z}/|I| : [\xi-R/|I|,\xi + R/|I|) \cap \Lambda \neq \emptyset \}. \]
Otherwise, we say that $I$ is \emph{$(R,t)$-diffuse} with respect to $\Lambda$.
\end{definition}

Below, we use the notation
\begin{align}\label{e:Specdelta}
    \text{Spec}_\delta(I) := \{ \xi \in \mathbb{Z}/|I| : |\mathcal{F}_I g(\xi)| \approx \delta |I| \}
\end{align}
to denote the large spectrum of $\delta$ at each scale and location; note that if $I$ is \emph{$(R,t)$-diffuse} with respect to $\Lambda$, then
\begin{align}
&\| \frac{1}{|I|} \sum_{\text{Spec}_\delta(I) \smallsetminus \Lambda_R(I)} \mathcal{F}_I {g}(\xi) e(\xi x) \|_{\ell^2(I)}^2 \\
& \qquad = \| \mathcal{F}_I{g} \cdot \mathbf{1}_{\text{Spec}_\delta(I) \smallsetminus \Lambda_R(I)} \|_{\ell^2(\mathbb{Z}/|I|)}^2 \gtrsim t^2 |I|.
\end{align}

Below, regarding $\delta > 0$ as fixed, we use the notation
\[ \Sigma(I) := \Sigma_\delta(I) := \text{Spec}_\delta(I) \subset \mathbb{Z}/|I|.\]

So, at time $t=0$, initiate 
\[ \Lambda_0(I_0) := \Sigma(I_0),\]
let
\[ X_1 := \{ r \in I_0 : \mathcal{I}(r) \text{ is diffuse with respect to $\Lambda_0(I_0)$} \}, \]
and let $T_1^{\text{max}}$ denote the maximal (with respect to inclusion) dyadic sub-intervals inside $X_1$.
Thus
\[ \text{sh}(X_1) = \bigcup_{T_1^{\text{max}}} I.\]
Note that for $r \in I_1 \in T_1^{\text{max}}$
\[ \mathcal{I}(r) \subset I_1.\]

For each $I_1 \in T_1^{\text{max}}$, set
\[ \Lambda_1(I_1) := \big\{ \xi \in \mathbb{Z}/|I_1| : \text{dist}(\xi,\Lambda_0(I_0)) \leq |I_1|^{-1} \big\} \cup \Sigma(I_1) \subset \mathbb{Z}/|I_1|. \]

Next, for each $I_1 \in T_1^{\text{max}}$, let
\[ X_2(I_1) := \{ r \in I_1 : \mathcal{I}(r)  \text{ is diffuse with respect to $\Lambda_1(I_1)$} \}, \]
and consolidate
\[ X_2 := \bigcup_{T_1^{\text{max}}} X_2(I_1).\]

We let $T_2^{\text{max}}(I_1)$ denote the maximal dyadic sub-intervals of $X_2(I_1)$
and collect
\[ T_2^{\text{max}} = \bigcup_{I_1 \in T_1^{\text{max}}} T_2^{\text{max}}(I_1). \]

For each $I_2 \in T_2^{\text{max}}(I_1)$, let
\[ \Lambda_2(I_2) := \big\{ \xi \in \mathbb{Z}/|I_2| : \text{dist}(\xi,\Lambda_1(I_1)) \leq |I_2|^{-1} \big\} \cup \Sigma(I_2) \subset \mathbb{Z}/|I_2|.\]

Note that for $r \in I_2 \in T_2^{\text{max}}(I_1)$
\[ \mathcal{I}(r) \subset I_2 \subset I_1.\]

We continue this process inductively up to time $V$; we prove the following localization estimate.

\begin{proposition}
The following bound holds:
\[ |\text{sh}(X_V)| \lesssim t^{-2} V^{-1} |I_0|.\]
\end{proposition}
\begin{proof}
We begin by sparsifying our scales into (say) $10$ subclasses, so that within each class
\[ |I| < |I'| \Rightarrow |I| \leq 2^{-10} |I'|; \]
we will restrict all intervals in question to a single such subclass, and then use a union bound to conclude.

Suppose that for all $w \leq v $,
\[ |\text{sh}(X_w)| \geq |\text{sh}(X_v)| \geq D^{-1} |I_0| \geq R^{-A_0} |I_0|.\]
We will show that $v \lesssim t^{-2} D$. Since our tree tops are disjoint and $g$ is one-bounded, we observe the energy bound
\begin{align}\label{e:totalenergy}
|I_0| \geq \sum_{T_{v-1}^{\text{max}}} \sum_{I_v \in T_v^{\text{max}}(I_{v-1})} \| \sum_{n \in I_v} g(n) e(-n \beta) \|_{\ell^2(\mathbb{Z}/|I_v|)}^2;
\end{align}
since our tree tops are \emph{diffuse}, we bound
\begin{align}
\eqref{e:totalenergy}
& \geq \sum_{T_{v-1}^{\text{max}}} \sum_{I_v \in T_v^{\text{max}}(I_{v-1})} \big( \| \sum_{n \in I_v} g(n) e(-n \beta) \cdot \mathbf{1}_{\Lambda_{v-1}(I_{v-1}) + B(R/|I_v|)}(\beta) \|_{\ell^2(\mathbb{Z}/|I_v|)}^2 + c t^2 |I_v| \big) \\
& \geq \sum_{T_{v-1}^{\text{max}}} \big( \sum_{I_v \in T_v^{\text{max}}(I_{v-1})} \| \sum_{n \in I_v} g(n) e(-n \beta) \cdot \mathbf{1}_{\Lambda_{v-1}(I_{v-1}) + B(R/|I_v|)}(\beta) \|_{\ell^2(\mathbb{Z}/|I_v|)}^2 \big) + c t^2 \cdot |\text{sh}(X_v)| \\
& \geq \sum_{T_{v-1}^{\text{max}}}  \| \sum_{n \in I_{v-1} } g(n) e(-n \beta) \cdot \mathbf{1}_{\Lambda_{v-1}(I_{v-1}) + B(R/|I_{v-1}|)}(\beta) \|_{\ell^2(\mathbb{Z}/|I_{v-1}|)}^2 \\
& \qquad + c t^2 \cdot  |\text{sh}(X_v)|   - \sum_{T_{v-1}^{\text{max}}} O_A(R^{-A} v \delta^{-2} + R^{-1/2})) |I_{v-1}|  \\
& \geq \sum_{T_{v-1}^{\text{max}}}  \| \sum_{n \in I_{v-1}} g(n) e(-n \beta) \cdot \mathbf{1}_{\Lambda_{v-1}(I_{v-1}) + B(R/|I_{v-1}|)} (\beta)\|_{\ell^2(\mathbb{Z}/|I_{v-1}|)}^2 + c_0 t^2/D |I_0|,
\end{align}
provided we choose $A$ so large that
\[ R^{-A} V \delta^{-2} \leq R^{-1},\]
say; we deduced the above lower bound via Lemma \ref{l:polysplit}, which applies since our tree tops are disjoint, nested, and have widely separated sizes by our initial sparsification.

We express
\begin{align}
&\sum_{T_{v-1}^{\text{max}}}  \| \sum_{n \in I_{v-1}} g(n) e(-n \beta) \cdot \mathbf{1}_{\Lambda_{v-1}(I_{v-1}) + B(R/|I_{v-1}|)}(\beta) \|_{\ell^2(\mathbb{Z}/|I_{v-1}|)}^2 \\
& \qquad = \sum_{T_{v-2}^{\text{max}}} \sum_{I_{v-1} \in T_{v-1}^{\text{max}}(I_{v-2})}  \| \sum_{n \in I_{v-1}} g(n) e(-n \beta) \cdot \mathbf{1}_{\Lambda_{v-1}(I_{v-1}) + B(R/|I_{v-1}|)}(\beta) \|_{\ell^2(\mathbb{Z}/|I_{v-1}|)}^2 \\
& \qquad \geq \sum_{T_{v-2}^{\text{max}}} \sum_{I_{v-1} \in T_{v-1}^{\text{max}}(I_{v-2})} \Big( \| \sum_{n \in I_{v-1}} g(n) e(-n \beta) \cdot \mathbf{1}_{\Lambda_{v-2}(I_{v-2}) + B(R/|I_{v-1}|)}(\beta) \|_{\ell^2(\mathbb{Z}/|I_{v-1}|)}^2 \\
& \qquad \qquad  + 
\| \sum_{n \in I_{v-1}} g(n) e(-n \beta) \cdot \mathbf{1}_{\text{Spec}_\delta(I_{v-1}) \smallsetminus \big( \Lambda_{v-2}(I_{v-2}) + B(R/|I_{v-1}|)\big) }(\beta) \|_{\ell^2(\mathbb{Z}/|I_{v-1}|)}^2 \Big),
\end{align}
since
\[ \Lambda_{v-1}(I_{v-1}) \supset \text{Spec}_{\delta}(I_{v-1}).\]
So,
\begin{align}
&\sum_{T_{v-1}^{\text{max}}}  \| \sum_{n \in I_{v-1}} g(n) e(-n \beta) \cdot \mathbf{1}_{\Lambda_{v-1}(I_{v-1}) + B(R/|I_{v-1}|)}(\beta) \|_{\ell^2(\mathbb{Z}/|I_{v-1}|)}^2 \\
& \geq \sum_{T_{v-2}^{\text{max}}} \big(\sum_{I_{v-1} \in T_{v-1}^{\text{max}}(I_{v-2})} \| \sum_{n \in I_{v-1}} g(n) e(-n \beta) \cdot \mathbf{1}_{\Lambda_{v-2}(I_{v-2}) + B(R/|I_{v-1}|)}(\beta) \|_{\ell^2(\mathbb{Z}/|I_{v-1}|)}^2\big) \\
& \qquad + c t^2 \cdot |\text{sh}(X_{v-1})| \\
& \geq \sum_{T_{v-2}^{\text{max}}}  \big( \| \sum_{n \in I_{v-2}} g(n) e(-n \beta) \cdot \mathbf{1}_{\Lambda_{v-2}(I_{v-2}) + B(R/|I_{v-2}|)}(\beta) \|_{\ell^2(\mathbb{Z}/|I_{v-2}|)}^2 \big) + c_0 t^2 \cdot |\text{sh}(X_{v-1})|
\end{align}
or
\[ |I_0| \geq \sum_{T_{v-2}^{\text{max}}}  \| \sum_{n \in I_{v-2}} g(n) e(-n \beta) \cdot \mathbf{1}_{\Lambda_{v-2}(I_{v-2}) + B(R/|I_{v-2}|)}(\beta) \|_{\ell^2(\mathbb{Z}/|I_{v-2}|)}^2 + 2 c_0 t^2/D |I_0|;\]
iterating this forces $v \lesssim t^{-2} D$.
\end{proof}

\subsection{From Trees to Branches}
In this subsection, we ``prune" each tree in our $V$-forest into ``branches," which have the additional feature that the frequencies linked to each branch are ``widely separated" relative to scale. 

We begin with the following definition.

\begin{definition}
Regarding a $V$ forest as fixed, a $U$-\emph{branch} consists of a (possibly empty) disjoint collection of intervals $\mathcal{B}_s(I_j) \subset \mathcal{D}_j(I_j), \ I_j \in T_j^{\text{max}}, \ s \leq U$ and collections of frequencies $\Lambda_{j,s}(I_j) \subset \mathbb{T}$, so that
    \begin{itemize}
        \item There are not too many frequencies involved:
        \[ \Lambda_{j,s}(I_j) \subset \Lambda_j(I_j); \]
        \item The frequencies are widely separated relative to scale:
        \[  
\min_{\theta \neq \theta' \in \Lambda_{j,s}(I_j)} \|\theta - \theta'\| \geq 2^{O(R)} \cdot \max_{I \in \mathcal{B}_s(I_j)} |I|^{-1}; \]
\item We may localize each family
\[ \big\{ \sum_{\xi \in \text{Spec}_\delta(I)} \mathcal{F}_{\mathbb{Z}}g(\xi) e(\xi n) : I \in \mathcal{B}_s(I_j) \big\}\]
to a small neighborhood of $\Lambda_{j,s}(I_j)$
\[ \| \sum_{\xi \in \mathbb{Z}/|I| \smallsetminus \big(\Lambda_{j,s}(I_j) + B(2R/|I|) \big)} \Psi_\delta\big( \frac{\mathcal{F}_I g(\xi)}{|I|} \big) e(\xi x) \|_{\ell^2(I)}^2 \ll t^2 |I|  \]
for each $I \in \mathcal{B}_s(I_j)$. 
    \end{itemize}
\end{definition}

We isolate the main structural decomposition we will achieve in the following Lemma.

\begin{lemma}[Pruning a Tree]\label{l:branches}
    For each $U \geq R t^{-2}$, for each $I_j \in T_j^{\text{max}}, \ j < V$, we may decompose
    \[ \mathcal{D}_j(I_j) = \bigcup_{s \leq UV} \mathcal{B}_s(I_j) \cup \mathcal{B}_{\infty}(I_j) \]
where each $\mathcal{B}_s(I_j)$ is a (possibly empty) ``branch," and there are only a few scales represented in $\mathcal{B}_\infty(I_j)$, the ``boundary" branch:
\[ |\{ |I| : I \in \mathcal{B}_\infty(I_j)\}| \lesssim RU V.\]
\end{lemma}
For the remainder of the paper, we will use ``branch" to refer to ``non-boundary branch," thus will distinguish between branches and boundary branches.

\begin{proof}
For each $\epsilon > 0$, let
\[ \text{Ent}(\Lambda_0,\epsilon) \]
denote the cardinality of the maximal $\epsilon$-separated subset of $\Lambda_0$, where we abbreviate $\Lambda_0 := \Lambda_j(I_j)$. We select an increasing sequence of \emph{non-positive} integers 
\[ m_0 < m_1 < \dots < m_{UV} \leq 0\] 
so that $|I_j|^{-1} = 2^{m_0}$, and for each $s \geq 1$, $m_s$ is the maximal integer satisfying
\[ \text{Ent}(\Lambda_0,2^{m_s}) \geq \text{Ent}(\Lambda_0,2^{m_{s-1}}) - \delta^{-2}/U; \]
since $|\Lambda_0| \leq V \delta^{-2}$, this process is guaranteed to stop after $\leq UV$ many steps.

We now choose $\Lambda_1 \subset \Lambda_0$ a maximal $2^{m_1}$-separated subset, so that
\[ \Lambda_0 \subset \Lambda_1 + B(2^{m_1}),\]
and inductively select
\[ \Lambda_s \subset \Lambda_{s-1}\]
a maximal $2^{m_s}$-separated subset, so that
\[ \Lambda_{s-1} \subset \Lambda_s + B(2^{m_s}).\]

We collect the boundary scales lieing to close to an entropy jump,
\[ \mathcal{B}_{\infty}(I_j) := \{ I \in \mathcal{D}_j(I_j) : |I|^{-1} = 2^{m_s \pm O(R)} \text{ for some } s \leq UV \}, \]
and
\[ \mathcal{B}_s(I_j) := \{ I \in \mathcal{D}_j(I_j) \smallsetminus \mathcal{B}_\infty(I_j) : 2^{m_{s-1}} \leq |I|^{-1} \leq 2^{m_s} \}, \]
so that $I \in \mathcal{B}_s(I_j)$ satisfy
\begin{align}
    2^{m_s} \gg 2^{O(R)}/|I|.
\end{align}

It remains to show that for all $I \in \mathcal{B}_s(I_j)$

\[ \| \sum_{\xi \in \mathbb{Z}/|I| \smallsetminus \big(\Lambda_{j,s}(I_j) + B(2R/|I|) \big)} \Psi_\delta\big( \frac{\mathcal{F}_I g(\xi)}{|I|} \big) e(\xi x) \|_{\ell^2(I)}^2 \ll t^2 |I|.  \]

Abbreviating 
\[ \Lambda_0 \supset \Lambda_{s-1} := \Lambda_{j,s-1}(I_j) \supset \Lambda_s := \Lambda_{j,s}(I_j), \]
we bound  
\begin{align*}
&\sum_{\xi \in \mathbb{Z}/|I| \smallsetminus \big(\Lambda_{j,s}(I_j) + B(2R/|I|) \big)} |\Psi_\delta\big( \frac{\mathcal{F}_I g(\xi)}{|I|} \big)|^2 \\
&\leq \sum_{\xi \in \mathbb{Z}/|I| \smallsetminus \big(\Lambda_0 + B(R/|I|) \big)} |\Psi_\delta\big( \frac{\mathcal{F}_I g(\xi)}{|I|} \big)|^2  + \sum_{\xi \in \big(\Lambda_0 + B(R/|I|) \big) \smallsetminus \big( \Lambda_s + B(2R/|I|) \big)} |\Psi_\delta\big( \frac{\mathcal{F}_I g(\xi)}{|I|} \big)|^2;
\end{align*}

since $I$ is localized with respect to $\Lambda_0$, we may bound
\begin{align}
    \sum_{\xi \in \mathbb{Z}/|I| \smallsetminus \big(\Lambda_0 + B(R/|I|) \big)} |\Psi_\delta\big( \frac{\mathcal{F}_I g(\xi)}{|I|} \big)|^2 \ll t^2 \cdot |I|;
\end{align}  
the second term is of a simpler nature: since $R/|I| \geq 1/|I| \gg 2^{m_{s-1}}$, and
\[ |\Psi_\delta | \approx \delta,\]
see \eqref{e:Psidelta},
we may  contain
\begin{align*}
&\big(\Lambda_0 + B(R/|I|) \big) \smallsetminus \big( \Lambda_s + B(2R/|I|) \big) \\
& \qquad \subset \big( \Lambda_{s-1} + B(2^{m_{s-1}}) + B(R/|I|) \big) \smallsetminus \big( \Lambda_s + B(2R/|I|) \big) \\
& \qquad \qquad \subset \big( \Lambda_{s-1} + B(2R/|I|) \big) \smallsetminus \big( \Lambda_s + B(2R/|I|) \big) \\
& \qquad \qquad \qquad \subset \big( \Lambda_{s-1} \smallsetminus \Lambda_s \big) + B(2R/|I|) , \end{align*}
which leads to the estimate
\begin{align}\label{e:aboutdelta}
    &\sum_{\xi \in \big(\Lambda_0 + B(R/|I|) \big) \smallsetminus \big( \Lambda_s + B(2R/|I|) \big)} |\Psi_\delta\big( \frac{\mathcal{F}_I g(\xi)}{|I|} \big)|^2 \\
    & \qquad \qquad \qquad \lesssim \delta^2 |I| \cdot |\Lambda_{s-1} \smallsetminus \Lambda_s|  \\
    & \qquad \qquad \qquad = \delta^2 |I| \cdot ( |\Lambda_{s-1}| - |\Lambda_s|)  \\
    &\qquad \qquad \qquad \leq R/U \cdot |I| \ll t^2 \cdot |I|. \notag
\end{align}
\end{proof}

With these preparations in mind, we proceed to the proof of Theorem \ref{t:main}.

\section{The Opening: Quantifying Convergence and Transference}\label{s:trans}
Let $(X,\mu,T)$ be arbitrary, and fix $f,g \in L^\infty(X)$; we first reduce to the case where $g \in \mathcal{K}(T)^{\perp}:$

\begin{lemma}
    In proving Theorem \ref{t:main}, it suffices to assume that $g \in \mathcal{K}(T)^{\perp}$.
\end{lemma}
\begin{proof}
    We quickly verify that for each measure-preserving system, $(X,\mu,T)$, and each $f \in L^\infty(X)$, the weighted averages
    \begin{align}\label{e:convlin}
        \frac{1}{N} \sum_{n \leq N} w(n) T^n f
    \end{align}
    converge almost surely, as addressing the case where $g$ is an eigenfunction follows from establishing the convergence of \eqref{e:convlin} in a product system; as is standard, we can assume that our sequence of times is lacunarily increasing.

So, if $s$ is as in the definition of admissibility, then setting $g = \mathbf{1}_X$, applying Calder\'{o}n's Transference Principle \cite{C1} and Lemma \ref{e:gowersUs}, we deduce that the following function is integrable,
\begin{align}
    \sum_{N \text{ lacunary}} |\frac{1}{N} \sum_{n \leq N} \mathcal{E}_{w;N}(n) T^n f|^{2^s} \in L^1(X),
\end{align}
where 
\begin{align}\label{e:errorweight}
    \mathcal{E}_{w;N} := w - \sum_{Q \leq M} w_Q, \; \; \; M=M(N)
\end{align}
for $M$ as in the definition of admissibility. So, it suffices to show that, almost surely,
\begin{align}
    \frac{1}{N} \sum_{n \leq N} \sum_{Q \leq M} w_Q(n) T^n f 
\end{align}
converges. To do so, we just observe that whenever $w$ is admissible, and $D \geq 1$:
    \begin{align}
        \| \mathcal{V}^r\big(\frac{1}{N} \sum_{n \leq N} w_Q(n) T^n f : N \geq Q^{100}, \ N \in \lfloor 2^{m/D} \rfloor\big) \|_{L^2(X)} \lesssim_{D} (\frac{r}{r-2})^2 Q^{o(1)-1} \| f\|_{L^2(X)},
    \end{align}
by Calder\'{o}n's Transference Principle \cite{C1} and Proposition \ref{p:mfvarprop}.
\end{proof}

So, in what follows, it suffices to assume that $g \in \mathcal{K}(T)^{\perp}$ is orthogonal to the Kronecker factor, with the aim of showing that
\begin{align}
    \frac{1}{N} \sum_{n \leq N} w(n) T^n f \cdot T^{-n} g
\end{align}
converges to $0$ pointwise (as it already does so in norm); as is standard, below we will assume that all times $N$ derive from a lacunary sequence.

With this in mind, seeking a contradiction, suppose that there exists some $\mathbf{c} > 0$ so that
\begin{align}
    \int_X \limsup|\frac{1}{N} \sum_{n \leq N} w(n) T^n f \cdot T^{-n} g | \ d\mu \gg \mathbf{c} > 0;
\end{align}
since $w$ is $\ell^1$-normalized, the above implies that whenever
\[ 0 \leq \varphi \in \mathcal{C}_c^\infty([0,1]) \; \; \; \text{ with } \; \; \; \| \varphi - 1 \|_{L^1([0,1])} \ll \mathbf{c} \]
necessarily
\begin{align}
    \int_X \limsup|\sum_n \varphi_N(n) w(n) T^n f \cdot T^{-n} g | \ d\mu \gg \mathbf{c} > 0,
\end{align}
as well; by the ergodic decomposition we can assume that $T$ is ergodic. For the remainder of the paper, we will view $\mathbf{c}$ as fixed, and will allow all constants to depend implicitly on this value.

Choose 
\[ \delta_0 \ll \exp(-\mathbf{c}^{-100}), \; \; \;  N_0 \gg \exp(\delta_0^{-100})\] so that
\begin{align}\label{e:Xdelta0}
    X_{\delta_0} := \{ \omega : \sup_{|I| \geq N_0^{1/2}, \ 0 \in I} \, \sup_{\beta \in \mathbb{T}} \, |\mathbb{E}_{n \in I} g(T^{-n} \omega) e(n\beta)| \leq \delta_0 \}
\end{align}
has $\mu(X_{\delta_0}) \geq 1 - \delta_0$. By Egorov's theorem and the ergodicity of $T$, there exists \[ X_{\delta_0}' \subset X_{\delta_0} \; \; \; \text{  with } \; \;\;\mu(X_{\delta_0}') \geq 1 - 2 \delta_0, \; \; \;  J_0 < \infty\] so that for all $J \geq J_0$, and $\omega \in X_{\delta_0}'$
\begin{align}
    \mathbb{E}_{[J]} \mathbf{1}_{X_{\delta_0}}(T^j \omega) \geq 1 - 3 \delta_0.
\end{align}

This implies that there exists $X_{0}$ with \[ \mu(X_0) \geq 1 - 4\delta_0, \;\;\; N_1 \gg N_0, \; \;\;  J' \geq J_0\] so that for all $\omega \in X_0$, and all $J \geq J'$
\[ \mathbb{E}_{k \in [J]} \sup_{N_0 \leq N \leq N_1} |\sum_n \varphi_N(n) w(n) f(T^{n-k} \omega)  g(T^{-n-k} \omega)| \gg \mathbf{c};\] and, if 
    \[ E_\omega := \{ k \in [J] : T^{-k} \omega \in X_{\delta_0}\},\] 
    then 
    \[ |E_\omega| \geq J(1-5 \delta_0).\]

In particular, setting
\begin{align}
    F_\omega(m) := f(T^{-m} \omega), \; \; \; G_\omega(m) := g(T^{-m} \omega),
\end{align}
it suffices to show that for $J \gg N_1, J'$
\begin{align}\label{e:firstlowerbound}
    & \qquad |\{  k \in  E_\omega : \sup_{N_0 \leq N \leq N_1} |\sum_n \varphi_N(n) w(n) F_\omega(k-n) G_\omega(k+n)| \gg \log^{-1}(1/\delta_0) \}|  \\
    & \qquad \qquad  \ll \log^{-1}(1/\delta_0) J
\end{align}
(say).

With $\mathcal{E}_{w;N}$ as in \eqref{e:errorweight} above, and
\begin{align}
    \mathcal{E}_{w;N}' := \sum_{(\log N)^5 \leq Q \leq M} w_Q,
\end{align}
we bound 
\begin{align}
    &\text{LHS }\eqref{e:firstlowerbound} \\
    & \leq     |\{  k \in  E_\omega : \sup_{N_0 \leq N \leq N_1} |\sum_n \varphi_N(n) (w- \mathcal{E}_{w;N})(n) 
    F_\omega(k-n) G_\omega(k+n)| \gg \log^{-1}(1/\delta_0) \}| \\
    & + |\{k \in   E_\omega : \sup_{N_0 \leq N \leq N_1} |\sum_n \varphi_N(n)\mathcal{E}_{w;N}(n)  F_\omega(k-n) G_\omega(k+n)| \gg \log^{-1}(1/\delta_0) \}| \\
    & \leq \sum_{Q} |\{ k \in   E_\omega : \sup_{\substack{ N_0 \leq N \leq N_1 \\ N \geq 2^{Q^{1/5}}}} |\sum_n  \varphi_N(n) w_Q(n) F_\omega(k-n) G_\omega(k+n)| \gg \log^{-1}(1/\delta_0) \log^{-2} Q \}| \\
    & + \log^4(1/\delta_0) \sum_{N_0}^{N_1} \| \mathcal{E}_{w;N}' \|_{U^3([N])}^4 \cdot J \\
    & + \log^{2^{s}}(1/\delta_0) \sum_{N_0}^{N_1} \| \mathcal{E}_{w;N} \|_{U^{s+2}([N])}^{2^s} \cdot J.
\end{align}
By admissibility (namely: Lemma \ref{p:U3est}), 
\begin{align}
    \| \mathcal{E}_{w;N}' \|_{U^3([N])}^4 \lesssim \log^{-14/13} N
\end{align}
and
\begin{align}
    \| \mathcal{E}_{w;N} \|_{U^{s+2}([N])}^{2^s} \lesssim \log^{-\nu/4} N,
\end{align}
so there exists some absolute $c = c(\nu) > 0$ so that
\begin{align}
&    \log^4(1/\delta_0) \sum_{N_0}^{N_1} \| \mathcal{E}_{w;N}' \|_{U^3([N])}^4  +
\log^{2^{s}}(1/\delta_0) \sum_{N_0}^{N_1} \| \mathcal{E}_{w;N} \|_{U^{s+2} ([N])}^{2^s} \\
& \qquad \lesssim \log^{2^{s}}(1/\delta_0) \cdot (\log N_0)^{-c(\nu)}   \ll \delta_0^{90}. 
\end{align}
Thus, by Lemma \ref{l:lacred} and the decay
\begin{align}
    \mathbf{S}_Q \lesssim Q^{o(1) -1},
\end{align}
it suffices to show that for each $\omega \in X_0$, and each $Q \geq 1$, whenever
\begin{align}\label{e:N0}
N_0 \geq \max\{ \exp( \delta_0^{-100}), 2^{Q^{1/5}} \}
\end{align}
is as in \eqref{e:Xdelta0},
\begin{align}\label{e:goal00}
    & \qquad   |\{  k \in  E_\omega : \sup_{N_0 \leq N \leq J} |\sum_n \varphi_N(n) w_Q(n) F_\omega(k-n) G_\omega(k+n)| \gg \delta_0^{\tilde{\epsilon}/100} Q^{-\tilde{\epsilon}/100} \}| \\
    & \qquad \lesssim \delta_0^{\tilde{\epsilon}/10} Q^{-\tilde{\epsilon}/10} J
\end{align}
where in \eqref{e:goal00} and for the rest of the paper, we can and will assume that all times are of the form
\begin{align}
    K_0 2^{\mathbb{N}},
\end{align}
see \eqref{e:K0size}.


Explicitly, we will spend the remainder of the paper focused on proving the following proposition: it asserts that if a bounded sequence $G$ is locally Fourier-uniform, then the bilinear averages against a single major arc piece, $w_Q$, are pointwise small outside a quantitatively sparse exceptional set.

\begin{proposition}\label{p:key0}
Let $Q \geq 1$, and
\begin{align}
    w_Q(n) := \sum_{a/q \in \Gamma_Q} S(a/q) e(na/q), \; \; \; |S(a/q)| \lesssim q^{o(1) -1}.
\end{align}
Suppose that $0 < \delta_0 \ll 1$ is very small but fixed, and that $N_0$ is restricted by \eqref{e:N0}.

Then, whenever $|F|, |G| \leq 1$, the following estimate holds:
\begin{align}
   & |\{ x \in E : \sup_{N_0 \leq N \leq J} |\sum_n  \varphi_N(n) w_Q(n) F(x-n) G(x+n)| \gg \delta_0^{\tilde{\epsilon}/5} Q^{-\tilde{\epsilon}/5} \}| \\
   & \qquad \lesssim \delta_0^{\tilde{\epsilon}/5} Q^{-\tilde{\epsilon}/5} J,
\end{align}
where
    \begin{align}
    E := \{ k \in [J] : \sup_{I : k \in I, \ N_0^{1/2} \leq |I| \leq J}  \sup_{\beta} \,  |\mathbb{E}_{n \in I} G(n) e(n \beta)| \ll \delta_0 \}.
    \end{align}
\end{proposition}

In our proof of Proposition \ref{p:key0}, it will be convenient to proceed using tools from dyadic harmonic analysis; we make use of our dyadic grid machinery, see \S \ref{ss:dyadicgrids}.

Thus, regarding 
\[ \mathcal{D} = \mathcal{D}_U^{\Delta_0,L} \]
and
$Q$ as fixed, define
\begin{align}\label{e:AI}
    A_I(f,g)(x) &:= A_{I;\mathcal{D}}(f,g)(x) \\
    & \qquad := \sum_n \phi_I(n) f(2x-n) w_Q(n-x) g(n) \mathbf{1}_{I \cap \overline{\mathcal{D}}}(x).
\end{align}
By Lemma \ref{l:dyadicsmoothness}, we may reduce Proposition \ref{p:key0} to the following.

\begin{proposition}\label{p:key1}
The following estimate holds for each $\mathcal{D} = \mathcal{D}_U^{\Delta_0,L}$:
\begin{align}
    |\{ k \in E \cap \overline{\mathcal{D}} : \sup_{I \in \mathcal{D}} |A_I(f,g)(k)| \gg \delta_0^{\tilde{\epsilon}/2} Q^{-\tilde{\epsilon}/2}  \}| \ll \Delta_0^{-8} J.
\end{align} 
\end{proposition}
Below, we will implicitly restrict the supremum over $I$ to elements of $\mathcal{D}$.

\medskip

We now decompose $g$ according to the level sets of its spectrum at each scale and location.

\subsection{Decomposing $g$}
For each interval, $I$, recall \eqref{e:Specdelta}:
\begin{align}
    \text{Spec}_\delta(I) := \{ \xi \in \mathbb{Z}/|I| : |\mathcal{F}_I g(\xi)| \approx \delta |I| \},
\end{align}
and note that on $E$, we may decompose
\begin{align}
\sup_I |A_I(f,g)| \leq \sum_{\delta \leq \delta_0} \sup_I |A_I^\delta(f,g)|
\end{align}
where
\begin{align}
    A_I^\delta(f,g)(x) := A_I(f,g_{\delta,I})(x)
\end{align}
for 
\[ g_{\delta,I}(x) := \sum_{\xi \in \mathbb{Z}/|I|} \Psi_\delta\big(\frac{\mathcal{F}_I g(\xi)}{|I|}\big) e(\xi x) \mathbf{1}_I(x),\]
and, as above, $\Psi_\delta$ is as in \eqref{e:Psidelta};\footnote{The presence of the cut-off $\mathbf{1}_I$ is for emphasis only, as the operator $A_I$ only detects the value of $g$ on $I$.}
note that
\begin{align}
    \sum_{\delta \leq \delta_0} \| g_{\delta,I} \|_{\ell^2(I)}^2 \leq |I|,
\end{align}
while $\Psi_\delta$ acts as a soft, regular threshold, selecting Fourier coefficients of magnitude $\approx \delta$, while preserving Lipschitz control.

For technical reasons, we will assume that 
\[ \text{Spec}_\delta(I) \subset P \subset \mathbb{T} \]
lives inside a single interval of length $1/10$\footnote{This can be accomplished by composing with the Fourier multipliers 
\[ \{ \psi(20 \beta -l) (1 - |\beta|)_+ : |l| \leq 100\}\]
for an appropriate bump function $\psi$ at unit scales} (say), and by modulating $f$ appropriately we can assume that $P=[0,1/10]$. We can also assume that $\mathcal{F}_{\mathbb{Z}}{f}$ is supported in a single interval of length $1/10$.

\medskip

\textbf{Thus, for the remainder of this paper, we will always assume that for each interval, $I \in \mathcal{D}$,
\[ \text{Spec}_\delta(I) \subset [0,1/10].\]}

\medskip

Thus, it suffices to prove the following proposition.

\begin{proposition}\label{p:key}
The following estimate holds for each $\mathcal{D} = \mathcal{D}_U^{\Delta,L}$:
\begin{align}
    |\{ k \in  [J] : \sup_{I \in \mathcal{D}} |A_I^\delta(f,g)(k)| \gg \delta^{\tilde{\epsilon}} Q^{-\tilde{\epsilon}} \}| \ll \Delta^{-10} J.
\end{align} 
\end{proposition}

\section{Preparations: Tightening the Fourier Projections}\label{s:tight}
To prove Proposition \ref{p:key}, we will make heavy use of our tree/branch construction, see \S \ref{s:ENERGY}. But, to do so, we will need to further localize our multipliers in Fourier space. We develop the necessary machinery to do so below:

\medskip

With $\Lambda \subset \mathbb{Z}/M_0$ for 
\[ 2^{Q^{1/5}} \leq M_0 \in K_0 2^{\mathbb{N}}, \]
$|\Lambda| \leq V \delta^{-2}$, and $M \geq 2^{O(R)} M_0$, define
\begin{align}\label{e:LAMBDAM}
    \Lambda_M := \{ \theta \in \mathbb{Z}/M : 0 < \text{dist}(\theta,\Lambda) \leq 10 R/M \},
\end{align}
and consider the Fourier multiplier
\begin{align}
    \Phi_M := \mathbf{1}_{\Lambda_M}.
\end{align}

With $M = |I|$, recall
\begin{align}
    {g}_{\delta,I}(x) := \sum_{\xi \in \mathbb{Z}/|I|} \Psi_\delta\big( \frac{\mathcal{F}_Ig(\xi)}{|I|} \big) e(\xi x) \mathbf{1}_I(x)
\end{align}
so that
\begin{align}
    \mathcal{F}_{\mathbb{Z}} g_{\delta,I} \equiv \mathcal{F}_I g_{\delta,I},
\end{align}
and define
\begin{align}
    g_{I;M}(x) := \mathbb{E}_{\xi \in \mathbb{Z}/|I|} \Phi_M(\xi) \mathcal{F}_Ig(\xi) e(\xi x) \mathbf{1}_I(x)
\end{align}
and
\begin{align}
g_{\delta,I;M}(x) := \sum_{\xi \in \mathbb{Z}/|I|} \Phi_M(\xi) \Psi_\delta\big( \frac{\mathcal{F}_Ig(\xi)}{|I|} \big) e(\xi x) \mathbf{1}_I(x);
\end{align}
these latter functions isolate the portion of the local Fourier support of $g$ lying in a thin annulus around the distinguished set $\Lambda$, but not directly in $\Lambda$. 

Our central object of study will be the operators
\begin{align}\label{e:Bw}
    B_{w,I}^\delta(f,g)(x) := \sum_{n} \phi_I(n) f(2x-n) w_Q(n-x){g_{\delta,I;M}}(n) \mathbf{1}_I(x),
\end{align}
and Proposition \ref{p:orthog00} will thus show that this ``annular" contribution is negligible, with the quantitative savings in $Q$ provided by Lemma \ref{p:U3est} and Lemma \ref{l:gowersL2}.

\begin{proposition}\label{p:orthog00}
The following bound holds:
\begin{align}
    \sum_{M \leq 2^{-O(R)} |I'|} \sum_{|I| = M, \ I \subset I'} \| B_{w,I}^\delta(f,g) \|_{\ell^2}^2 \lesssim R \cdot Q^{- 1/20 } |I'|.
\end{align}
\end{proposition}

To prove this, we proceed by orthogonality, with the main tool being the following lemma.

\begin{lemma}\label{l:orth}
The following bound holds
\begin{align}
    \sum_{M \leq 2^{-O(R)} |I'|} \ \sum_{|I| =M, \ I \subset I'} \|g_{I;M}\|_{\ell^2}^2 \lesssim R \|g \|_{\ell^2}^2.
\end{align}
\end{lemma}
\begin{proof}
We sparsify our scales into $O(R)$ many subclasses and estimate the contribution from each class independently; henceforth all of our scales will satisfy
\[ |I_1| < |I_2| \Rightarrow |I_1| \leq 2^{-O(R)} |I_2|, \]
as well as the absolute constraint $|I| \leq 2^{-O(R)} |I'|$.

We proceed by orthogonality methods, viewing our desired estimate through the lens of Bessel-type inequalities for families of wave-packets, with the arithmetic structure of $\Lambda_M$ used precisly to guarantee almost orthogonality between wave-packets at widely separated scales.

In particular, we define the $\ell^2$-normalized family of functions,
\begin{align}
    \eta_{I,\xi}(n) := |I|^{-1/2} \mathbf{1}_I(n) e(\xi n),
\end{align}
which form an orthonormal basis for functions supported on $I$, and allows us to express
\begin{align}
    g_{I;M} = \sum_{\xi \in \Lambda_M} \langle g, \eta_{I,\xi} \rangle \eta_{I,\xi}.
\end{align}
So, our job is to prove 
\begin{align}
    \sum_M \sum_{|I| = M} \sum_{\xi \in \Lambda_M} |\langle g, \eta_{I,\xi} \rangle|^2 \lesssim \|g \|_{\ell^2}^2,
\end{align}
namely a Bessel inequality for the collection of wave packets
\begin{align}
    \big\{ \eta_{I,\xi} : I \subset I', \ \xi \in \Lambda_{|I|} \big\};
\end{align}
the key point in the below analysis will be the representation:
\begin{align}
\xi \in \Lambda_M \Rightarrow \xi =: \theta(\xi) + \frac{k(\xi)}{M}, \; \; \; \theta(\xi) \in \mathbb{Z}/M_0, \ 1 \leq |k(\xi)| \leq 10R.
\end{align}

We use a Cotlar-Stein orthogonality argument, setting
\begin{align}
    P_N g := \sum_{|I| = N} \sum_{\xi \in \Lambda_N} \langle g,\eta_{I,\xi} \rangle \eta_{I,\xi};
\end{align}
since our intervals are disjoint, we may trivially bound
\begin{align}
    \| P_N g \|_{\ell^2}^2 \leq \|g \|_{\ell^2}^2,
\end{align}
and since $\{ P_N \}$ are self-adjoint, it suffices to prove that, whenever $M \geq 2^{O(R)} N$
\begin{align}
    \| P_N P_M \|_{2 \to 2} \lesssim (N/M)^{1/2}.
\end{align}

To see this, let $I_0$ be an arbitrary interval of length $M$; since we may bound
\begin{align}
    \| P_N P_M \|_{2 \to 2}^2 &\leq \sup_{I_0} \sum_{\xi \in \Lambda_M} \| P_N \eta_{I_0, \xi} \|_{\ell^2}^2 \\
    & \lesssim R|\Lambda| \cdot \sup_{I_0} \max_{\xi \in \Lambda_M} \, \| P_N \eta_{I_0, \xi} \|_{\ell^2}^2,
\end{align}
it suffices to estimate
\[  \sup_{I_0}  \max_{\xi \in \Lambda_M} \, \| P_N \eta_{I_0, \xi} \|_{\ell^2}^2 \lesssim 2^{-O(R)} (N/M). \]
But, we compute
\begin{align}\label{e:wavepacketcoeff}
    \| P_N \eta_{I_0,\xi} \|_{\ell^2}^2  &= \sum_{I \subset I_0, \ \zeta \in \Lambda_N} |\langle \eta_{I_0,\xi}, \eta_{I,\zeta} \rangle|^2 \\
    & = \frac{M}{N} \sum_{1 \leq |k| \leq 10R} \ \sum_{\zeta \in \Lambda_N, \ k(\zeta) = k} |\langle \eta_{I_0,\xi}, \eta_{I,\zeta} \rangle|^2 \\
    & = \frac{1}{N^2} \cdot \sum_{1 \leq |k| \leq 10R} \ \sum_{\zeta \in \Lambda_N, \ k(\zeta) = k} \big| \sum_{n \leq N} e(n (\xi - \zeta) )\big|^2 \\
    & = \frac{1}{N^2} \cdot \sum_{1 \leq |k| \leq 10R} \ \sum_{\zeta \in \Lambda_N, \ k(\zeta) = k} \big| \frac{e(N(\xi - \zeta)) - 1}{e(\xi - \zeta) - 1} \big|^2,
\end{align}
and since
\begin{align}
    N(\xi - \zeta) = N( \theta(\xi) - \zeta) + O( RN/M) \equiv O(RN/M) \mod 1,
\end{align}
we can bound
\begin{align}
    \eqref{e:wavepacketcoeff} &\lesssim \frac{R^2}{M^2} \sum_{1 \leq |k| \leq 10 R} \ \sum_{\zeta \in \Lambda_N, \ k(\zeta) = k}  \frac{1}{\| \xi - \zeta\|^2} \\
    & \lesssim \frac{R^2}{M^2} \sum_{1 \leq |k| \leq 10 R} \ \sum_{\zeta \in \Lambda_N, \ k(\zeta) = k} \Big( \frac{1}{\| \theta(\xi) - \theta(\zeta)\|^2} \cdot \mathbf{1}_{\theta(\xi) \neq \theta(\zeta)}  + (\frac{N}{k})^2 \cdot \mathbf{1}_{\theta(\xi) = \theta(\zeta)} \Big) \\
    & \lesssim \frac{R^3 M_0^2}{M^2} + \frac{R^2 N^2}{M^2} \\
    & \lesssim \frac{R^3 N^2}{M^2} \\
    & \lesssim 2^{-O(R)} \frac{N}{M},
\end{align}
as desired.
\end{proof}

With Lemma \ref{l:orth} in hand, we can quickly establish Proposition \ref{p:orthog00}.

\begin{proof}
By admissibility, namely Lemma \ref{p:U3est}, and Lemma \ref{l:gowersL2}, we may bound
\begin{align}
\| B_{w,I}^\delta(f,g) \|_{\ell^2}^2  \lesssim Q^{-1/20} \| g_{\delta,I;M} \|_{\ell^2}^2  \leq Q^{-1/20} \| g_{I;M} \|_{\ell^2}^2 
\end{align}
where $M = |I|$. Then we sparsify our scales $M$ into $O(R)$ many sub-families, and apply Lemma \ref{l:orth} above.
\end{proof}

Below, for a fixed $\Lambda$ coming from the tree/branch construction, we will split 
\begin{align}\label{e:trichotomy}
{g}_{\delta,I}(x) &= \mathcal{E}_I + g_{\delta,I;|I|}(x) + \sum_{\theta \in \Lambda} \Psi_\delta\big( \mathbb{E}_{n \in I} e(-\theta n) g(n) \big) e(\theta x) \mathbf{1}_I(x) \\
& =: \mathcal{E}_I + g_{\delta,I;|I|}(x) + \Pi_I[\Lambda] g,
\end{align}
into
\begin{itemize}
    \item a negligible error, $\mathcal{E}_I$, coming from frequencies away from $\Lambda$, which satisfies
\[ \|\mathcal{E}_I \|_{\ell^2}^2 \ll t^2 |I|, \]
and so can be ignored by Lemma \ref{l:l2topointwise};
\item An annular component, $g_{\delta,I;|I|}$, controlled by orthogonality, namely a square function argument and Proposition \ref{p:orthog00}; and
\item The main structured component, $\Pi_I[\Lambda] g$, which behaves like a finite linear combination of characters, which we address using entropy methods.
\end{itemize}

More precisely, on any tree we may bound
\begin{align}\label{e:treeestimate}
    |\{ x \in I_v : \sup_{I \in \mathcal{D}_v(I_v)} |B_{w,I}^\delta(f,g)(x)| \gg t\}| \lesssim t^{-2} R Q^{-1/20} |I_v| \lesssim Q^{-1/25} |I_v|,
\end{align}
so summing over $v \leq V$, 
\begin{align}\label{e:treeestimate1}
\sum_{v \leq V} \sum_{I_v \in T_v^{\text{max}}} |\{ x \in I_v : \sup_{I \in \mathcal{D}_v(I_v)} |B_{w,I}^\delta(f,g)(x)| \gg t\}| \lesssim Q^{-1/30} J.
\end{align}
The upshot is that, on the branch, $\mathcal{B}_s(I_v)$, we may replace 
\begin{align}
    g \longrightarrow \Pi_I[\Lambda] g, \; \; \; \Lambda = \Lambda_{v,s}(I_v),
\end{align}
thereby reducing to the case where $g$ ``looks like" a linear combination of characters at all relevant scales and locations, completing the program described in our proof overview. Informally -- and ignoring arithmetic issues -- we have reduced our analysis to a ``variable-coefficient" variant of \cite{B2}.

\bigskip

With this reduction in mind, our next section will begin setting up the machinery needed for our entropy argument, the crucial mechanism behind our proof.
\section{Entropic Preliminaries}\label{s:entprelim}
We begin this section by introducing a few classes of multi-frequency ``projections;" throughout this section and the remainder of the main argument, we regard $Q,\ \delta$ as fixed.

For finite subsets $\Lambda \subset \mathbb{Z}/M_0$, and $|I| \geq 2^{O(R)} M_0$ -- we think of $\Lambda$ as deriving from our tree/branch construction -- we define the projection to the third term of \eqref{e:trichotomy} above,
\begin{align}
    \Pi_I g(x) := \Pi_I[\Lambda] g(x) := \sum_{\theta \in \Lambda} \Psi_\delta(\mathbb{E}_I \text{Mod}_{-\theta} g) e(\theta x) \mathbf{1}_{I}(x),
\end{align}
see \eqref{e:Psidelta}: this replaces $g$ on $I$ by a linear combination of characters, with amplitudes given by local Fourier averages, all approximately $\delta$ in magnitude. Indeed, when $\Lambda$ derives from our tree/branch construction, we will replace

\begin{align}
    A_I^\delta(f,g)(x) := A_I(f,g_{\delta,I})(x) \longrightarrow A_I(f,\Pi_I g),
\end{align}
as per \eqref{e:trichotomy} above. We note the trivial bound
\begin{align}
\sup_I |\Pi_I g| \lesssim |\Lambda|^{1/2}    
\end{align}
since $|g| \leq 1$, see Lemma \ref{l:sampling}.

With $\chi^N$ as in \eqref{e:chibump}, define
\begin{align}
    \Pi_N = \Pi_{N;(\Lambda,Q)}
\end{align}
by
\begin{align}
    \mathcal{F}_{\mathbb{Z}}(\Pi_N f) := \chi^N \cdot \mathcal{F}_{\mathbb{Z}}{f}.
\end{align}

A key point is that we may freely apply $\Pi_N$ to $f$: the essential Fourier support of $f$ can be taken to be suitably arithmetically constrained. Indeed,
\begin{align}\label{e:replacement}
    A_I(f,\Pi_I g) &= A_I(\Pi_N f,\Pi_I g) +O_A(\Delta^{-A}) \\
    &= A_I(\Pi_N (f \cdot \mathbf{1}_{\Delta I}),\Pi_I g) +O(Q^{-10}), \; \; \; |I| = N,
\end{align}
using the regularity of $\chi^N$. To see this, it suffices to show that by Fourier inversion
 \begin{align}
      &  \sum_{\substack{a/q \in \Gamma_Q \\ \theta \in \Lambda}} S(\frac{a}{q}) \Psi_{\delta}(\EE_I \text{Mod}_{-\theta} g) \\
      & \qquad \times \int_\TT \mathcal{F}_{\mathbb{Z}}(f \cdot \mathbf{1}_{\Delta I})(\beta) e(2x \beta) \cdot (1 - \chi^N)(\beta) \Bigl(\sum_{n \in \Z} \phi_I(n) e(n(\theta + a/q - \beta)) \Bigr) d \beta  \\
      & = O(Q^{-10}).
    \end{align}
But, applying Poisson summation yields the bound 
\begin{align}
   &|(1 - \chi^N )(\beta)| \cdot |\sum_{n \in \Z} \phi_I(n) e(n(\theta + a/q - \beta))|\\
   & \qquad \lesssim_A    |(1 - \chi^N )(\beta)| \cdot \frac{1}{\Delta^A N \|\theta + a/q - \beta \|},
\end{align} 
so a dyadic decomposition in Fourier space, along with the normalization
\[ \| f \cdot \mathbf{1}_{\Delta I} \|_{\ell^2}^2 \lesssim \Delta N \]
yields the result.

Consolidating Proposition \ref{p:mfvarprop} and Corollary \ref{c:mfproj}, we arrive at the following estimate:
\begin{proposition}\label{p:mfvar}
    The following bound holds for each $r > 2$:
    \begin{align}
      &  \| \mathcal{V}^r( \Pi_I g : x \in I, \ |I| \geq 2^{O(R)} M_0 )(x) \|_{\ell^2} \lesssim (\frac{r}{r-2})^2 \log^2 Q \| g \|_{\ell^2} \; \; \; \text{ and} \\
      &  \| \mathcal{V}^r( \Pi_N f : N \geq 2^{O(R)} M_0) \|_{\ell^2} \lesssim (\frac{r}{r-2})^2 \log^2 Q \|f \|_{\ell^2}.
    \end{align}
\end{proposition}
This proposition allows us to control the oscillation across scales for both families of projections. In particular, it implies that, generically, neither family exhibits significant variation, so that -- away from small exceptional sets -- the collections
\begin{align}\label{e:mfmults00}
    \{ \Pi_N f(x) : N \}, \; \; \; \{ \Pi_I g(x) : x \in I \}
\end{align}
may be efficiently approximated by finite subfamilies: for most $x$, the functions \eqref{e:mfmults00} change appreciably at a highly controlled number of scales. This reduction is the key input in passing from multi-scale analysis to single-scale estimates, which we develop below.

\subsection{Single Scale Estimates}
At the core of the entropic method are sufficiently strong single scale estimates. Taking into account the arithmetic context of our work, we will require two versions, over full intervals and over arithmetic progressions.

To consolidate notation, regarding $\delta$ as fixed, we define
\begin{align}
    A_I^{(i)} := A_I^{\delta,(i)}
\end{align}
to be like $A_I^\delta$, but with $w_Q$ replaced by $w_Q^{(i)}$, thus
\begin{align}\label{e:A_Idelta}
    A_I^{(i)}(f,g)(x) := \sum_n \phi_I(n) f(2x-n) w_Q^{(i)}(n-x) g_{\delta,I}(n) \cdot \mathbf{1}_{I \cap \overline{\mathcal{D}}}(x);
\end{align}
note the analogue of \eqref{e:replacement} holds as well for 
\[ A_I^{(i)}(f, \Pi_I[\Lambda] g), \]
i.e.\
\begin{align}\label{e:replacementi}
    A_I^{(i)}(f,\Pi_I g) &= A_I^{(i)}(\Pi_N f,\Pi_I g) +O_A(\Delta^{-A}) \\
    &= A_I^{(i)}(\Pi_N (f \cdot \mathbf{1}_{\Delta I}),\Pi_I g) +O(Q^{-10}).
\end{align}

To facilitate our argument, we begin with the following elementary consequence of Taylor expansion, which we record as a lemma; we use this to separate spatial dependence (in $x$) from frequency dependence, to reduce the above to finite sums of simpler expressions involving linear combinations of characters.

\begin{lemma}[Taylor Expansion Lemma]\label{l:easytaylorexpansion}
Suppose that $\phi \in \mathcal{C}_c^{\infty}([0,1])$ and that $\mathcal{D}$ is a shifted dyadic grid as in \S \ref{ss:dyadicgrids}, with parameter $\Delta$ as in \eqref{e:DELTA}.  Consider an operator
\begin{align}
        D_I F(x) := \sum_n \phi_I(n) F(2x-n) \mathbf{1}_{I \cap \overline{\mathcal{D}}}(x) 
    \end{align}
where 
\begin{align}
    \phi_I(n) := \frac{1}{|I|} \phi(\frac{n-c_I}{|I|}).
\end{align}
If $\kappa \lesssim \epsilon_0 \ll \kappa$ is sufficiently small, and
\begin{align}
    \phi^{(j)}_I(n) := \frac{1}{j!} \frac{1}{|I|} (\partial^j \phi)(\frac{n - c_I}{|I|}) \; \; \; \text{ and } \; \; \; 
    \chi_I(x) := \frac{x-c_I}{|I|} \cdot \mathbf{1}_{I \cap \overline{\mathcal{D}}}(x),
\end{align}
then we may Taylor expand
\begin{align}
    D_I F(x) = \sum_{j \leq \epsilon_0^{-1} A} \big( \sum_n \phi^{(j)}_I(n) F(n) \big) \cdot \chi_I(x)^j + O(\delta^A Q^{-A} \| F \|_{\ell^\infty}).
\end{align}
\end{lemma}

With this in mind, we present our first single scale estimate, where we sum over full intervals. 

\begin{lemma}[Single Scale Estimate]\label{l:1scale}
    Suppose that $|P| \geq 2^R M_0$, and $0 \leq i \leq \log_2 Q$. Then for any $I \supset P$
\begin{align}
    &\| \sum_n \phi_I(n) (\Pi_N f)(2x-n) w_Q^{(i)}(n-x) \Pi_I[\Lambda]g(n) \cdot \mathbf{1}_{I \cap \overline{\mathcal{D}}}(x) \|_{\ell^2(P)}^2 \\
    & \lesssim |P| \cdot Q^{\epsilon/2} \overline{\lambda}  \min\{ 2^i K_0, |\Lambda|\} \cdot (\delta/Q)^2,
\end{align}
and thus
\begin{align}\label{e:fullintest1}
   & \| \sum_n \phi_I(n) (\Pi_N f)(2x-n) w_Q^{(i)}(n-x) \Pi_I[\Lambda]g(n) \cdot \mathbf{1}_{I \cap \overline{\mathcal{D}}}(x)\|_{\ell^2}^2 \\
   & \qquad \lesssim |I| \cdot Q^{\epsilon/2} \overline{\lambda}  \min\{ 2^i K_0, |\Lambda|\} \cdot (\delta/Q)^2.
\end{align}
In particular, summing over all integers $i \leq \log_2 Q$, one bounds:
\begin{align}
    &\| \sum_n \phi_I(n) (\Pi_N f)(2x-n) w_Q(n-x) \Pi_I[\Lambda]g(n) \cdot  \mathbf{1}_{I \cap \overline{\mathcal{D}}}(x) \|_{\ell^2(P)}^2 \\
    & \qquad \lesssim |P| \cdot Q^{\epsilon} \overline{\lambda}^2 (\delta/Q)^2,
\end{align}
and analogously
\begin{align}\label{e:fullintest}
   & \| \sum_n \phi_I(n) (\Pi_N f)(2x-n) w_Q(n-x) \Pi_I[\Lambda]g(n) \cdot \mathbf{1}_{I \cap \overline{\mathcal{D}}}(x)\|_{\ell^2}^2 \\
   & \qquad \lesssim |I| \cdot Q^{\epsilon} \overline{\lambda}^2  (\delta/Q)^2.
\end{align}
\end{lemma}
\begin{proof}
We expand
\begin{align}
    w_Q(n-x), \ \Pi_I[\Lambda]g(n)
\end{align}
into linear combinations of characters (so frequencies lie in $\Gamma_Q^{(i)} - \Lambda$), apply the previous Taylor expansion Lemma, Lemma \ref{l:easytaylorexpansion}, and estimate the contribution of each $j$ independently (the error term is certainly of an acceptable size). Thus, with $j \leq \epsilon_0^{-1} A$, observe that 
\[ F(n) := (\Pi_N f)(n) \]
is essentially supported on an interval of length $Q^{\epsilon/4}N$, in that the tail decays pointwise like $O_A(Q^{-A})$ using the Schwartz decay of $\chi^N$. Then, if $\mathbf{F}_P$ a Schwartz function with compactly supported Fourier transform that is $\geq 1$ on $10 P$, we estimate
\begin{align}
    &\| \sum_{a/q \in \Gamma_Q^{(i)}, \ \theta \in \Lambda} \big( \sum_n \phi^{(j)}_I(n) (\text{Mod}_{a/q - \theta} F)(n) \big) \cdot \big( S(a/q) \Psi_\delta(\mathbb{E}_I \text{Mod}_{-\theta} g) \big) \\
        & \qquad \times e(2 \theta x - a/q x) \cdot \chi_I(x)^j  \mathbf{F}_P(x) \|_{\ell^2}^2 \\
        & \lesssim (\Delta)^{-j} Q^{o(1)-2} \delta^2 |P| \sum_{a/q \in \Gamma_Q^{(i)}, \ \theta \in \Lambda} |\sum_n \phi^{(j)}_I(n) (\text{Mod}_{a/q - \theta} F)(n) |^2 \\\
        & \qquad \times \sum_{a'/q' \in \Gamma_Q^{(i)}, \ \theta' \in \Lambda} \mathbf{1}_{\| (2\theta - a/q) - (2\theta' - a'/q') \| \leq 1/|P|} \\
        &\lesssim (\Delta)^{-j} \cdot |P| Q^{o(1)-2} \delta^2 \min\{ |\Lambda|, 2^i K_0 \} \sum_{\xi \in \Gamma_Q^{(i)} - \Lambda} |(\mathcal{F}_{\mathbb{Z}}{\phi^{(j)}_I})*(\mathcal{F}_{\mathbb{Z}} {F})(\xi)|^2 \\
    & \lesssim (\Delta)^{-j/2} |P| \overline{\lambda} Q^{\epsilon/2-2} \delta^2 \min\{ |\Lambda|, 2^i K_0 \};
\end{align}
the steps are as follows: we appealed to the AM-GM inequality, then Lemma \ref{l:maximal-multiplicity}, and finally sampling, see Lemma \ref{l:sampling}, using the above localization of $F$ to an interval of length $Q^{\epsilon/4} N$.

Summing over $0 \leq j \leq \epsilon_0^{-1} A$ completes the point.
\end{proof}

We next address the single-scale behavior along arithmetic progressions; this is a refinement of the single-scale estimate, which annihilates the oscillation coming from 
\[ \{ a/q \in \Gamma_Q^{(i)} \};\]
it will be crucial for entropy arguments later.

\begin{lemma}[Single Scale Estimate Along Arithmetic Progressions]\label{l:1scale0}
 Take any $i \leq \log_2 Q$ and suppose that 
\begin{align}
    |P| \geq \begin{cases} 2^{Q^{1/2+\epsilon}} M_0 & \text{ if } Q^{1/2} \leq 2^i \leq Q \\
    2^{Q^{1+\epsilon}} M_0 & \text{ if } 2^i < Q^{1/2}.
    \end{cases}
\end{align}
Then for any $l \in [\mathcal{Q}_i]$ and interval $|I| \geq 2^{O(R)} |P|$:
\begin{align}
    &\| \sum_n \phi_I(n) \Pi_N f(2x+2l-n) w_Q^{(i)}(n-x-l) \Pi_I[\Lambda]g(n) \|_{\ell^2(P \cap \mathcal{Q}_i\mathbb{Z})}^2 \\
    & \qquad \lesssim K_0 \overline{\lambda}Q^{\epsilon} \delta^{2} |P|/\mathcal{Q}_i
\end{align}
\end{lemma}
\begin{proof}
Arguing as above, and maintaining the notation introduced there, we estimate
\begin{align}\label{e:settinguptheest}
    &\| \sum_{a/q \in \Gamma_Q^{(i)}, \ \theta \in \Lambda} \big( \sum_n \phi^{(j)}_I(n) (\text{Mod}_{a/q - \theta} F)(n) \big) \cdot \big(e(2 \theta l - a/q l) S(a/q) \Psi_\delta(\mathbb{E}_I \text{Mod}_{-\theta} g) \big) \\
    & \qquad \times e(2 \theta x - a/q x) \cdot \chi_I(\mathcal{Q}_i x)^j \mathbf{F}_P(\mathcal{Q}_i x) \|_{\ell^2}^2 \\
    & \lesssim \Delta^{-j} \cdot |P|/\mathcal{Q}_i \cdot Q^{o(1)-2} \delta^2 |\Gamma_Q^{(i)}| \cdot \sum_{a/q \in \Gamma_Q^{(i)}, \ \theta \in \Lambda} |(\mathcal{F}_{\mathbb{Z}}{\phi^{(j)}_I})*(\mathcal{F}_{\mathbb{Z}} {F})(a/q-\theta )|^2  \\
    & \qquad \times  \sum_{\theta' \in \Lambda} \mathbf{1}_{\| 2\mathcal{Q}_i(\theta - \theta') \| \leq \frac{\mathcal{Q}_i}{10|P|}}.
\end{align}    
   Since
   \begin{align}
       \| 2 \mathcal{Q}_i (\theta - \theta') \| > 0 \Rightarrow \| 2 \mathcal{Q}_i (\theta - \theta') \| \geq  \frac{1}{M_0},
   \end{align}
 by our (sufficiently large) choice of $|P|$ and \eqref{eqn:size-of-Qi},
   we see that whenever 
\[ \| 2 \mathcal{Q}_i(\theta - \theta') \| \leq \frac{\mathcal{Q}_i}{10 |P|} \Rightarrow 2 \mathcal{Q}_i(\theta - \theta') \equiv 0 \mod 1;\] 
since the kernel of multiplication by $2 \mathcal{Q}_i$ acting on $\mathbb{Z}/M_0$ has cardinality 
\begin{align}
    \leq \text{gcd}(M_0,\mathcal{Q}_i) \lesssim K_0 2^i,
\end{align}
we may bound
\begin{align}
    \sum_{\theta' \in \Lambda} \mathbf{1}_{\| 2\mathcal{Q}_i(\theta - \theta') \| \leq \frac{\mathcal{Q}_i}{10|P|}} \lesssim K_0 2^i,
    \end{align}
and thus majorize the above by
\begin{align}
    \eqref{e:settinguptheest} \leq &\Delta^{-j/2} \cdot |P|/\mathcal{Q}_i \cdot  Q^{o(1)-2} \delta^2 |\Gamma_Q^{(i)}| \cdot (K_0 2^i \overline{\lambda} Q^{\epsilon/2})  \\
    & \qquad \lesssim \Delta^{-j/2} \cdot |P|/\mathcal{Q}_i \cdot Q^{\epsilon} K_0 \delta^2 \overline{\lambda}, 
\end{align}
and a final sum over $0 \leq j \leq \epsilon_0^{-1} A$ concludes the proof.
\end{proof}

\bigskip

In the following section, we will apply entropy-based arguments to appropriately discretize our infinite suprema, reducing matters to single-scale analysis; the above estimates will be used to close the argument.

\section{Branch Analysis}\label{s:nonboundary}
We begin this section by addressing the case of boundary branches. This argument follows from a straightforward application of Lemma \ref{l:gowersL2} and Lemma \ref{p:U3est}; we specialize $U = O(t^{-2} R)$ in the setting of Lemma \ref{l:branches}.

\begin{lemma}[Contribution from Boundary Branches]
The following estimate holds:
\begin{align}
    &\sum_{v \leq V} \sum_{I_v \in T_v^{\text{max}}} |\{ x \in I_v \smallsetminus \text{sh}(X_{v+1}) : \sup_{I \in \mathcal{B}_\infty(I_v)} |A^\delta_I(f,g)(x)| \gg t \}| \\
    & \lesssim Q^{-1/25} J.
\end{align}
\end{lemma}
\begin{proof}
We bound
\begin{align}
&    |\{ x \in I_v \smallsetminus \text{sh}(X_{v+1}) : \sup_{I \in \mathcal{B}_\infty(I_v)} |A^\delta_I(f,g)(x)| \gg t \}| \\
& \leq t^{-4} R^2 V \max_{N} \, \sum_{|I| = N, \ I \subset I_v} \| A^\delta_I(f,g) \|_{\ell^2}^2 \\
& \lesssim Q^{-1/20} |I_v|,
\end{align}
by Lemmas \ref{l:gowersL2} and \ref{p:U3est}, 
so 
\begin{align}
&\sum_{v \leq V} \sum_{I_v \in T_v^{\text{max}}} |\{ x \in I_v \smallsetminus \text{sh}(X_{v+1}) : \sup_{I \in \mathcal{B}_\infty(I_v)} |A^\delta_I(f,g)(x)| \gg t \}| \\
& \lesssim V Q^{-1/20} J \\
& \ll Q^{-1/25} J,
\end{align}
as desired.
\end{proof}

With this lemma in mind, we are finally in a position to implement the approach described in our proof overview: we have reduced our argument to studying single branches, provided we are suitably efficient.

Below, let 
\begin{align}
    \mathbf{C}_{\text{Branch}} :=     \mathbf{C}_{\text{Branch}}(\delta,Q)
\end{align}
denote the best constant so that the inequality is satisfied uniformly for each integer $0 \leq i \leq \log_2 Q$, each $s \leq U = O(t^{-2} R)$ and each $v \leq V$:
\begin{align}
    |\{ x \in I_v \smallsetminus \text{sh}(X_{v+1}) : \sup_{I \in \mathcal{B}_s(I_v)} |A_I^{(i)}( \Pi_N f,\Pi_I g)(x)| \gg t/\log Q \}| \leq \mathbf{C}_{\text{Branch}} |I_v|,
\end{align}
where
\[ \Lambda = \Lambda_{v,s}(I_v)\]
is fixed along the branch $\mathcal{B}_s(I_v)$,
and 
\[ \Pi_I = \Pi_I[\Lambda], \; \; \;  \Pi_N = \Pi_{N;(\Lambda,Q)}\]
depends on this choice of $\Lambda$.

Then we may bound the contribution from the non-boundary branches in a single dyadic grid by
\begin{align}
    &(t^{-2} V^{-1}  + Vt^{-2} R \cdot \log Q \cdot \mathbf{C}_{\text{Branch}})J \\
    & \lesssim Q^{-\rho/20} + Q^{\rho/5} \mathbf{C}_{\text{Branch}},
\end{align}
where the first term derives from the exceptional set estimate,
\begin{align}
    |\text{sh}(X_V)| \lesssim t^{-2} V^{-1}.
\end{align}
Thus, to establish Proposition \ref{p:key}, and with it, Theorem \ref{t:main} in the case where $Q \geq \delta^{-1/1000}$, it suffices to bound e.g.\ 
\[ \mathbf{C}_{\text{Branch}} \lesssim Q^{-\rho/2}.\]

\medskip

Below, we regard $\mathcal{B}_s(I_v), \ s \leq U = O(t^{-2} R), v \leq V$ as arbitrary but fixed. And, we are only interested in the case where $I \in \mathcal{B}_s(I_v)$ satisfy
\[ |I| \ll 2^{-O(R)} |I_v|\]
and
\[ |I|^{-1} \ll 2^{-O(R)} \min_{\theta \neq \theta' \in \Lambda} |\theta - \theta'|\]
if $\Lambda = \Lambda_{v,s}(I_v)$.

\subsection{Discarding Small Intervals}
We define $\mathcal{B}_s^{\geq}(I_v)$ as the set of intervals $\{ I \} \subset \mathcal{B}_s(I_v)$ whose lengths satisfy
\begin{align}\label{e:largeI}
    |I| \geq \begin{cases} 2^{Q^{1/2+3\epsilon}} M_0 & \text{ if } Q^{1/2} \leq 2^i \leq Q \\
    2^{Q^{1+3\epsilon}} M_0 & \text{ if } 2^i < Q^{1/2},
    \end{cases}
\end{align}
and let $\mathcal{B}_s^{\leq}(I_v)$ denote the complementary set; note that the number of scales appearing in $\mathcal{B}_s^{\leq}(I_v)$,
\[ \mathbf{I}(Q;i) := \mathbf{I}_{s,v}(Q;i) \]
is bounded, uniformly in $s,v$, by
\begin{align}
\mathbf{I}(Q;i) \lesssim \begin{cases} {Q^{1/2+3\epsilon}} & \text{ if } Q^{1/2} \leq 2^i \leq Q \\
    {Q^{1+3\epsilon}} & \text{ if } 2^i < Q^{1/2}.
    \end{cases}
\end{align}
In particular, there are few such scales, so crude $L^2$ bounds suffice to address their contribution: applying \eqref{e:fullintest1}, we bound 
\begin{align}
    &|\{ x \in I_v \smallsetminus \text{sh}(X_{v+1}) : \sup_{I \in \mathcal{B}^{\leq}_s(I_v)} |A_I^{(i)}( \Pi_N f,\Pi_I  g)(x)| \gg t/\log Q \}| \\
    & \lesssim t^{-3} \sum_{I \in \mathcal{B}^{\leq}_s(I_v)} \| A_I^{(i)}( \Pi_N f,\Pi_I g) \|_{\ell^2}^2 \\
    & \lesssim t^{-3} \cdot \mathbf{I}(Q;i) \cdot  Q^{\epsilon} \overline{\lambda}   \cdot \min\{ 2^i K_0, |\Lambda|\} \cdot (\delta/Q)^2 \cdot |I_v| \\
    & \lesssim t^{-3} \cdot Q^{\epsilon}  \overline{\lambda}   \cdot \min\{ 2^i K_0, |\Lambda|\} \cdot (\delta/Q)^2 \cdot \big( Q^{1/2+3\epsilon} \cdot \mathbf{1}_{Q^{1/2} \leq 2^i \leq Q} + Q^{1+3\epsilon} \cdot \mathbf{1}_{2^i < Q^{1/2}} \big) \cdot |I_v| \\
& \lesssim Q^{-1/5} \cdot |I_v|.
\end{align}

\medskip

It remains to establish the following Proposition.
\begin{proposition}\label{p:nearly}
For each $0 \leq i \leq \log_2 Q, \ s \lesssim t^{-2} R, \ v \leq V$
\begin{align}
|\{ x \in I_v \smallsetminus \text{sh}(X_{v+1}) : \sup_{I \in \mathcal{B}_s^{\geq}(I_v)} |A_I^{(i)}(\Pi_N f,\Pi_I g)(x)| \gg t/\log Q \}| \lesssim Q^{-\rho/2} |I_v|
\end{align}
where $\Lambda = \Lambda_{v,s}(I_v)$, and $\Pi_N, \Pi_I$ are defined in terms of $\Lambda = \Lambda_{v,s}(I_v)$ as above.
\end{proposition}
We emphasize that for the rest of this section, all time parameters will satisfy \eqref{e:largeI}.

\subsection{Conclusion: the Entropic Method}\label{ss:entropymethod}
We now apply the entropic method to prove Proposition \ref{p:nearly}. We do so via the following proposition. 

\begin{proposition}\label{p:overlap}
For each $I_v$ and each $i \leq \log_2 Q$, there exists a set
\begin{align}
\mathcal{O}_{I_v} := \mathcal{O}_{I_v}^{(i)} \subset I_v
\end{align}
with $|\mathcal{O}_{I_v}| \lesssim Q^{-\rho} |I_v|$ so that
    \begin{align}\label{e:overlap}
        &\| \sup_{I \in \mathcal{B}_s^{\geq}(I_v) } |\sum_n \phi_I(n) \Pi_N f(2x-n) w_Q^{(i)}(n-x) \Pi_I[\Lambda]g(n)| \mathbf{1}_I(x) \|_{\ell^2(\mathcal{O}_{I_v}^c)}^2 \\
        &\lesssim Q^{-2 \rho} |I_v|.
    \end{align}
    \end{proposition}
We will use the arithmetic single-scale estimate, Lemma \ref{l:1scale0}, crucially; it will be complemented by an entropy estimate. In particular, we will show that  outside a small exceptional set, each branchwise family can be discretized on each short progression into a net of controlled cardinality, and that each representative enjoys a favorable arithmetic single-scale estimate.

Below, we emphasize that all $I$ satisfy \eqref{e:largeI}, i.e.\ all times will be extremely large relative to \eqref{e:Psize} below. We now fix an $0 \leq i \leq \log_2 Q$ and describe $\mathcal{O}_{I_v} = \mathcal{O}_{I_v}^{(i)}$ now; it is made up of the union of two sets.

\medskip

We will say that interval $P \subset I_v$ with 
\begin{align}\label{e:Psize}
    |P| = \begin{cases} 2^{Q^{1/2+\epsilon}} M_0 & \text{ if } Q^{1/2} \leq 2^i \leq Q \\
    2^{Q^{1+\epsilon}} M_0 & \text{ if } 2^i < Q^{1/2}
    \end{cases}
\end{align}
is \emph{bad} if the modulated averages of $g$ fluctuate too much across it:
\begin{align}
\vec{N}_{c_0Q^{-2 \rho}} ( \big( \mathbb{E}_{I} \text{Mod}_{-\theta} g \big)_{\theta \in \Lambda} : I \supset P) \geq Q^{6 \rho},
\end{align}
where $0 < c_0 \leq 1$ is an absolute, but inessential, constant, depending only on the implicit constant in \eqref{e:Lip}. 

Collect
\begin{align}
    \mathcal{O}_{I_v;1} := \bigcup_{P \subset I_v \text{ bad}} P
\end{align}
so that 
\[ |\mathcal{O}_{I_v;1}| \lesssim Q^{-2 \rho} |I_v|\]
by our consequence of L\'{e}pingle's inequality, Corollary \ref{c:sep}.

\medskip

Describing the second component of our exceptional set takes a little more work:

First, restricting $f$ as we may to $\Delta I_v$, see \eqref{e:replacementi}, let
\begin{align}
    F_{\theta,N}^{(i)}(x) &:= \mathbb{E}_{[\mathcal{Q}_i]} w_Q^{(i)}(r) \text{Mod}_{-\theta} \Pi_{N} f(x-r)  \\
    & = \mathbb{E}_{[\mathcal{Q}_i]} w_Q^{(i)}(r) \Pi_{N}^{\theta} (\text{Mod}_{-\theta} f)(x-r)
\end{align}
where
\begin{align}
    \mathcal{F}_{\mathbb{Z}}(\Pi_N^{\theta} f)(\beta) := {\chi^N}(\beta+\theta) (\mathcal{F}_{\mathbb{Z}}{f})(\beta).
\end{align}

We define
    \begin{align}
        \mathcal{O}_{I_v;2} := \{ x \in I_v : \mathcal{V}^r\big( \big( \mathcal{Q}_i \sum_p \varphi_N(p \mathcal{Q}_i) F_{\theta,N}^{(i)}(x-p \mathcal{Q}_i) \big)_{\theta \in \Lambda} : N \big) \geq Q^{2\rho} \overline{\lambda}^{1/2} Q^{-1} \};
    \end{align}
    this captures the points where the pertaining $f$-component exhibits large variation across scales.

Note that since our times $N$ are so large relative to $|P|$, by the regularity of 
\[ \{\varphi_N : N\},\]
whenever $x \in \mathcal{O}_{I_v;2}$, necessarily 
\[ x + \mathcal{Q}_i \mathbb{Z} \cap P\]
lives inside
\begin{align}
    \{ x \in I_v : \mathcal{V}^r\big( \big( \mathcal{Q}_i \sum_p \varphi_N(p \mathcal{Q}_i) F_{\theta,N}^{(i)}(x-p \mathcal{Q}_i) \big)_{\theta \in \Lambda} : N\big) \geq \frac{1}{10} \cdot Q^{2\rho} \overline{\lambda}^{1/2} Q^{-1} \}:
\end{align}
the above variation operator is essentially constant along short arithmetic progressions with gap size $\mathcal{Q}_i$. We quantify this as follows: if we define the fluctuation
\begin{align}
    &\Omega_{f;i;N;\theta}(x,y) \\
    & := \big| \mathcal{Q}_i \sum_p \varphi_N(p\mathcal{Q}_i) \mathbb{E}_{r\in [\mathcal{Q}_i]} w_Q^{(i)}(r) (\text{Mod}_{-\theta } \Pi_N f)(x-r-p\mathcal{Q}_i) \\
    & \qquad \qquad  - 
\mathcal{Q}_i \sum_p \varphi_N(p\mathcal{Q}_i) \mathbb{E}_{r\in [\mathcal{Q}_i]} w_Q^{(i)}(r) (\text{Mod}_{-\theta } \Pi_N f)(y-r-p\mathcal{Q}_i) \big| \\
& = \big| \mathcal{Q}_i \sum_p \varphi_N(p\mathcal{Q}_i) F_{\theta,N}^{(i)}(x-p \mathcal{Q}_i) - \mathcal{Q}_i \sum_p \varphi_N(p\mathcal{Q}_i) F_{\theta,N}^{(i)}(y-p \mathcal{Q}_i) \big|,
\end{align}
then whenever $P$ satisfies \eqref{e:Psize}, and the sum runs over times $N$ satisfying \eqref{e:largeI}
\begin{align}\label{e:uncertaintyAP}
    &\sup_{\substack{ x_P \equiv y_P \mod \mathcal{Q}_i \\ x_P, y_P \in P}} \ \sum_{N} \sum_{\theta \in \Lambda} \Omega_{f;i;N;\theta}(x_P,y_P) \lesssim_A Q^{-A}.
\end{align}

In particular, after adjusting implicit constants, we will assume that whenever $\mathcal{P}$ is an arithmetic progression with gap size $\mathcal{Q}_i$
\begin{align}
(\mathcal{P} \cap P) \cap \mathcal{O}_{I_v;2} \neq \emptyset \Rightarrow (\mathcal{P} \cap P) \subset \mathcal{O}_{I_v;2}.
\end{align}
We now define
\[ \mathcal{O}_{I_v} := \mathcal{O}_{I_v;1} \cup  \mathcal{O}_{I_v;2}.\]

We will prove the following:
\begin{lemma}
The following estimate holds for each $2 < r \leq 2 + 2^{-1000}$:
\begin{align}
    |\mathcal{O}_{I_v;2}| \lesssim (\frac{r}{r-2})^4 Q^{o(1)-2 \rho} |I_v|.
\end{align}
\end{lemma}
\begin{proof}
    By Minkowski's inequality, it suffices to prove 
    \begin{align}
        &\| \| \mathcal{V}^r\big( \mathcal{Q}_i \sum_p \varphi_N(p \mathcal{Q}_i) F_{\theta,N}^{(i)}(x-p \mathcal{Q}_i) :N \big) \|_{\ell^2(\theta \in \Lambda)} \|_{\ell^2}^2 \\
        & \lesssim (\frac{r}{r-2})^4 \overline{\lambda} Q^{o(1)-2} |I_v|.
    \end{align}
The argument goes by way of Fourier analysis, with the principal tool being Proposition \ref{p:mfvar}: 

Noting that for each $\theta$ the multipliers
\begin{align}
    \{ \Xi_{N,\theta}(\beta) := (\mathcal{F}_{\mathbb{Z}}{\varphi_{N/\mathcal{Q}_i}})( \mathcal{Q}_i\beta) \chi^{N}(\beta+\theta) : N \}
\end{align}
are like $\{ \chi^N \}$ with at most $O(\Delta^{1/2})$ times the number of distinguished frequency points, by convexity we may bound
\begin{align}
    &\mathcal{V}^r\big( \mathcal{Q}_i \sum_p \varphi_N(p \mathcal{Q}_i) F_{\theta,N}^{(i)}(x-p \mathcal{Q}_i) :N \big) \\
    & \leq \mathcal{V}^r\big( \mathcal{Q}_i \sum_p \tilde{\varphi}_{M_0'}*_{\mathcal{Q}_i}\varphi_N(p \mathcal{Q}_i) F_{\theta,N}^{(i)}(x-p \mathcal{Q}_i) :N \big) + 2^{-O(R)} M_{\text{HL}}f(x),
\end{align}
see \eqref{e:M_0'}, where $\tilde{\varphi}$ is Schwartz with compact support in Fourier space, and the convolution occurs $\mod \mathcal{Q}_i$, 
we use the Fourier transform to express
\begin{align}
     &\mathcal{V}^r\big( \mathcal{Q}_i \sum_p \varphi_N(p \mathcal{Q}_i) F_{\theta,N}^{(i)}(x-p \mathcal{Q}_i) :N \big) \\
& \leq \mathcal{V}^r\big( \mathcal{Q}_i \sum_p \tilde{\varphi}_{M_0'}*_{\mathcal{Q}_i}\varphi_N(p \mathcal{Q}_i) F_{\theta,N}^{(i)}(x-p \mathcal{Q}_i) :N \big) + 2^{-O(R)} M_{\text{HL}}f(x) \\
& = \mathcal{V}^r\Big( \int \Xi_{N,\theta}(\beta) \cdot \big( (\mathcal{F}_{\mathbb{Z}}{\varphi_{M_0'}})( \mathcal{Q}_i \beta) \widehat{w_Q^{(i)}}(\beta) \chi^{M_0}(\beta+\theta) \cdot (\mathcal{F}_{\mathbb{Z}}{f})(\beta+\theta) \big) e(\beta x) : N \Big) \\
& \qquad + 2^{-O(R)} M_{\text{HL}}f(x),
\end{align}
where we here abbreviate
\begin{align}
    \widehat{w_Q^{(i)}}(\beta) := \mathbb{E}_{r\in [\mathcal{Q}_i]} w_Q^{(i)}(r) e(-\beta r)
\end{align}
as in \eqref{e:FTw_Q}.

By Corollary \ref{c:mfproj} and a square function argument to replace the smooth cut-offs
\begin{align}
    \mathcal{F}_{\mathbb{Z}}{\varphi_{N/\mathcal{Q}_i}} \longrightarrow \mathbf{1}_{[-\mathcal{Q}_i/N,\mathcal{Q}_i/N]},
\end{align}
this function has $\ell^2$ norm bounded by
\begin{align}
    (\frac{r}{r-2})^2 \log^2 Q \cdot \| (\mathcal{F}_{\mathbb{Z}}{\varphi_{ M_0'}})( \mathcal{Q}_i(\beta - \theta)) \widehat{w_Q^{(i)}}(\beta-\theta) \chi^{M_0}(\beta) \cdot  \mathcal{F}_{\mathbb{Z}}{f}(\beta) \|_{\ell^2}
\end{align}
so square summing in $\theta \in \Lambda$ produces an $\ell^2$ bound that's the square root of
\begin{align}
    &\lesssim (\frac{r}{r-2})^4 Q^{o(1)} |I_v| \cdot \sup_\beta \chi^{M_0}(\beta) \sum_{\theta \in \Lambda} |(\mathcal{F}_{\mathbb{Z}}{\varphi_{M_0'}})( \mathcal{Q}_i(\beta - \theta)) \widehat{w_Q^{(i)}}(\beta-\theta) |^2 \\
    & \lesssim (\frac{r}{r-2})^4 \overline{\lambda} Q^{o(1)-2} |I_v| 
    \end{align}
by Lemma \ref{variation-bound-finalizing}.
\end{proof}

We continue with the Proof of Proposition \ref{p:overlap}: our job is to establish \eqref{e:overlap}.

With this in mind, we begin as follows; below, we let $I_N(x)$ denote the unique interval of length $|I_N(x)| = N$ so that $I_N(x) \ni x$.

\begin{lemma} \label{l:size-net-jump-counting-functions}
Let $\eta = Q^{-\rho}$, $l \in [\mathcal{Q}_i]$, and suppose that $B_{\eta}(P,l)$ is an $\ell^1$-net in
\begin{align}
&    \Big\{ \big( \mathcal{Q}_i \sum_p \varphi_N(p\mathcal{Q}_i) \mathbb{E}_{r\in [\mathcal{Q}_i]} w_Q^{(i)}(r) (\text{Mod}_{-\theta } \Pi_N f)(x_P-l-r-p\mathcal{Q}_i) \cdot \Psi_\delta(\mathbb{E}_{I_N(x_P)} \text{Mod}_{-\theta} g) \big)_{\theta \in \Lambda} \\
& \qquad \qquad \qquad : |I| = N  \Big\} \\
& =  \Big\{ \big( \mathcal{Q}_i \sum_p \varphi_N(p\mathcal{Q}_i ) ( F_{\theta,N}^{(i)} \cdot \mathbf{1}_{-l \mod \mathcal{Q}_i}) (x_P-l-p\mathcal{Q}_i)  \cdot \Psi_\delta(\mathbb{E}_{I_N(x_P)} \text{Mod}_{-\theta} g) \big)_{\theta \in \Lambda} \\
& \qquad \qquad \qquad : |I| = N  \Big\}
\end{align}
where $x_P \in P \cap \mathcal{Q}_i \mathbb{Z}$ is arbitrary. 
Then, for each $l \in [\mathcal{Q}_i]$
\begin{align}\label{e:l2tol1net}
    &|B_\eta(P,l)| \\
    &\leq \vec{N}_{Q^{-2 \rho}}( \big( \mathcal{Q}_i \sum_p \varphi_N(p\mathcal{Q}_i) F_{\theta,N}^{(i)}(x_P-l-p\mathcal{Q}_i) \big)_{\theta \in \Lambda} : N) \\
    & \qquad \times \vec{N}_{c_0 Q^{-2 \rho}} ( \big(  \mathbb{E}_{I_N(x_P)} \text{Mod}_{-\theta} g  \bigr)_{\theta \in \Lambda} : |I| = N).
\end{align}
\end{lemma}
\begin{proof}
The basic inequality we use is that whenever 
\[ \{ a_N(\theta) : N \}, \ \{ c_N(\theta) : N \} \subset \ell^2(\Lambda),\]
with
\begin{align}
    \sup_N \| \vec{a}_N \|_{\ell^2} \leq A_0, \; \; \;     \sup_N \| \vec{c}_N \|_{\ell^2} \leq C_0,
\end{align}
then the fewest number of $\ell^1(\Lambda)$-balls of radius $\eta$ needed to cover
\begin{align}
    \{ a_N(\theta) c_M(\theta) : N,M \} 
\end{align}
-- and thus the fewest number of $\ell^1(\Lambda)$-balls of radius $\eta$ needed to cover
\begin{align}
    \{ a_N(\theta) c_N(\theta) : N \} 
\end{align}
-- is bounded by the product
\begin{align}
    \vec{N}_{\frac{\eta}{10 C_0}}( \vec{a}_N : N ) \cdot \vec{N}_{\frac{\eta}{10 A_0}}( \vec{c}_M : M),
\end{align}
where the jump functions are taken with respect to the $\ell^2(\Lambda)$ norm; this just follows from the triangle inequality, Cauchy-Schwarz, and the fact that the minimal number of $\ell^2$-balls of radius $\lambda$ that it costs to cover $\{ \vec{a}_N \}$ is bounded by
\begin{align}
    \vec{N}_{\lambda}(\vec{a}_N :N),
\end{align}
and similarly for $\{ \vec{c}_M \}$.
With this in mind, \eqref{e:l2tol1net} follows directly from the Lipschitz nature of $\Psi_\delta$, and the estimate, valid for all $N \geq 2^{Q^{1/5}}$:
\begin{align}\label{e:entest00}
& \| \mathcal{Q}_i \sum_p \varphi_{N}(p \mathcal{Q}_i) \mathbb{E}_{r\in [\mathcal{Q}_i]} w_Q^{(i)}(r) \text{Mod}_{-\theta} \Pi_{N}f(x_P -p\mathcal{Q}_i-l-r) \|_{\ell^2(\theta \in \Lambda)}^2 \\
&\lesssim Q^{\rho/10}.
\end{align}
Indeed, we compute,
    \begin{align}
\eqref{e:entest00} 
& = \| \sum_n e(\theta n) \varphi_N(n) w_Q^{(i)}(n) \Pi_N f(x_P-l-n) \|_{\ell^2(\theta \in \Lambda)}^2 + O(Q^{-100}) \\
& \lesssim \sum_n \frac{1}{N} (1 + |n/N|)^{-A}  |w_Q^{(i)}(n)|^2 |\Pi_N f(x_P-l-n)|^2 
\end{align}
by sampling, namely Lemma \ref{l:sampling} above, so by the Schwartz nature of $\chi^N$, we may bound the foregoing by
\begin{align} 
&\lesssim \frac{1}{N} \sum_{|n-x_P| \lesssim \Delta N} |w_Q^{(i)}(n)|^2 |\Pi_N f(x_P-n)|^2 + O(\delta^A Q^{-A})\\
        & \lesssim \Delta \cdot (\frac{1}{\Delta N} \sum_{|n-x_P| \lesssim \Delta N} |w_Q^{(i)}(n)|^{2k} )^{1/k} (\frac{1}{\Delta N} \sum_{|n-x_P| \lesssim \Delta N} |\Pi_N f(x_P -n)|^{2k'})^{1/k'}+ O(\delta^A Q^{-A})\\
        &\lesssim \Delta^2 \ll Q^{\rho/100},
    \end{align}
    provided we choose $k = k(\kappa)$ sufficiently large, which we are free to do since $N \geq 2^{Q^{1/5}}$; the key point is the moment estimate from \eqref{eqn:heath-brown-moments-w_Q^i} and the bound
    \begin{align}
        \|\Pi_N f\|_{\ell^\infty} \lesssim Q^2 |\Lambda|.
    \end{align}
\end{proof}

In particular, whenever
\begin{align}\label{e:emptyint}
    \big (P \cap (l + \mathcal{Q}_i \mathbb{Z}) \big) \smallsetminus \mathcal{O}_{I_v} \neq \emptyset,
\end{align}
by applying \eqref{e:uncertaintyAP}, we see that for any 
\[ x_P \in P \cap (l + \mathcal{Q}_i \mathbb{Z}),\]
we may bound
\begin{align}\label{e:jumpbound0}
&\vec{N}_{Q^{-2 \rho}} ( \big( \Psi_\delta ( \mathbb{E}_{I_N(x_P)} \text{Mod}_{-\theta} g ) \bigr)_{\theta \in \Lambda} : |I| = N) \\
& \leq\vec{N}_{c_0 Q^{-2 \rho}} ( \big(  \mathbb{E}_{I_N(x_P)} \text{Mod}_{-\theta} g  \bigr)_{\theta \in \Lambda} : |I| = N) \lesssim Q^{6 \rho},
\end{align}
see \eqref{e:Lip},
and
\begin{align}\label{e:vbound0}
    &\vec{N}_{Q^{-2 \rho}}( \big( \mathcal{Q}_i \sum_p \varphi_N(p\mathcal{Q}_i) F_{\theta,N}^{(i)}(x_P-l-p\mathcal{Q}_i) \big)_{\theta \in \Lambda} : N) \\
    & \lesssim Q^{2 r \rho} \mathcal{V}^r\big( \mathcal{Q}_i \sum_p \varphi_N(p \mathcal{Q}_i) F_{\theta,N}^{(i)}(x_P-l-p \mathcal{Q}_i) \big)_{\theta \in \Lambda} : N)^r \\
    & \lesssim Q^{4r \rho} \overline{\lambda}^{r/2} Q^{-r},
\end{align}
so for such $P,l$, provided $r > 2$ is sufficiently close to $2$, depending on $\rho$,
\begin{align}
    |B_{\eta}(P,l)| \lesssim Q^{15 \rho} \overline{\lambda}^{r/2} Q^{-r} \lesssim Q^{20 \rho} \overline{\lambda} Q^{-1}.
\end{align}
Combining \eqref{e:jumpbound0} and \eqref{e:vbound0}, we see that for $P,l$ so that \eqref{e:emptyint} is satisfied, we may bound
\begin{align}\label{e:boundforent}
 |B_{\eta}(P,l)| |P \cap (l + \mathcal{Q}_i \mathbb{Z})| \lesssim Q^{20 \rho} \overline{\lambda} Q^{-1} \frac{|P|}{\mathcal{Q}_i}.
\end{align}

With \eqref{e:boundforent} in hand, we are at last ready to prove Proposition \ref{p:overlap}. We do so by showing that outside our small exceptional set, $\mathcal{O}_{I_v}$, the supremum can be discretized into finitely many representatives, which can be each individually addressed using our single-scale estimates.

\begin{proof}[Proof of Proposition \ref{p:overlap}]
Let $\eta = Q^{-\rho}$, and for $x, x_P \in P \smallsetminus \mathcal{O}_{I_v}$ with $x-x_P \in \mathcal{Q}_i \mathbb{Z}$, by \eqref{e:uncertaintyAP}, we may generously estimate
\begin{align}
    &\sup_{\substack{I \in \mathcal{B}_s^{\geq}(I_v) \\ x \in I}} |\sum_n \phi_I(n) \Pi_N f(2x-n) w_Q^{(i)}(n-x) \Pi_I[\Lambda]g(n)|^2 \\
    & \qquad = \sup_{\substack{I \in \mathcal{B}_s^{\geq}(I_v) \\ x \in I}} |\sum_n \phi_I(n) \Pi_N f(2x_P-n) w_Q^{(i)}(n-x) \Pi_I[\Lambda]g(n)|^2 + O(\eta^{10})
\end{align}
since $I \in \mathcal{B}_s^{\geq}(I_v)$ are so large relative to $|P|$. Consequently, we may bound
\begin{align}
    &\sup_{\substack{I \in \mathcal{B}_s^{\geq}(I_v) \\ x \in I}} |\sum_n \phi_I(n) \Pi_N f(2x-n) w_Q^{(i)}(n-x) \Pi_I[\Lambda]g(n)|^2 \\
    &\lesssim \eta^2 + \sum_{l \in [\mathcal{Q}_i]} \sum_{I \in B_\eta(P,l)} |\sum_n \phi_I(n) \Pi_N f(2x_P-n) w_Q^{(i)}(n-x) \Pi_I[\Lambda]g(n)|^2 \cdot \mathbf{1}_{-l \mod \mathcal{Q}_i} \\
    & = \eta^2 + \sum_{l \in [\mathcal{Q}_i]} \sum_{I \in B_\eta(P,l)}  |\sum_{\theta \in \Lambda} e(2 \theta x) \cdot \Psi_\delta(\mathbb{E}_I \text{Mod}_{-\theta} g) \\
    & \qquad \qquad \qquad \qquad \qquad \times \sum_n \phi_I(n) \text{Mod}_{-\theta} \Pi_N f(2x_P-n) w_Q^{(i)}(n-x) |^2 \cdot \mathbf{1}_{-l \mod \mathcal{Q}_i};
\end{align}
above, we have indexed elements of $B_\eta(P,l)$ be representative intervals $I$, and then summed over these representatives. By Lemma \ref{l:1scale0}, if we sum over $x \notin \mathcal{O}_{I_v}$, we may bound the above by
\begin{align}
&\eta^2 |P| \\
&+ K_0 \overline{\lambda}Q^{\epsilon} \delta^{2} \sum_{l \mod \mathcal{Q}_i : (P \cap (l \mod \mathcal{Q}_i) ) \cap \mathcal{O}_{I_v} = \emptyset} |B_\eta(P,l)| \cdot |P|/\mathcal{Q}_i \\
& \lesssim \eta^2 |P| + Q^{25 \rho} \delta^2 \overline{\lambda}^2 Q^{-2} |P| \\
& \leq  (Q^{-2 \rho} + Q^{30 \rho-1} ) |P|;
\end{align}
applying \eqref{e:boundforent} now completes the proof of Proposition \ref{p:overlap}.
\end{proof}

This concludes Theorem \ref{t:main} in the case where $Q \geq \delta^{-1/1000}$. We now address the converse case; the argument below is of a simpler nature, more closely resembling \cite{B3}.

\section{The Remaining Case: $Q \leq \delta^{-1/1000}$}\label{s:remainder}
It remains to reduce to the case where
\[ Q \leq \delta^{-1/1000};\]
in this case, arithmetic structure is weak and a simpler convexity argument suffices.
Indeed, using that
\begin{align}
    \sum_{a/q \in \Gamma_Q} |S(a/q)| \lesssim Q^{1+o(1)} \ll \delta^{-1/990},
\end{align}
it suffices to prove that
\begin{align}
    |\{ x \in [J] \cap \overline{\mathcal{D}}: \sup_I |A_I(f,g_{\delta,I})(x)| \gg \delta^{1/800} \}| \lesssim \delta^{1/15} J,
\end{align}
in the special case where $w \equiv \mathbf{1}$,
where we \emph{reset} parameters 
\[ t := \delta^{1/700}, \ \Delta := \delta^{-1/300}, \ R = \delta^{-1/25}, \ V = \delta^{-1/12},\] 
maintaining the notation above in our tree/branch constructions. The arguments now all essentially derive from \cite{B3}, though we more closely follow the approach of \cite{BKMod}.

Here and below we let
\begin{align}
    {g}_{\delta,I}(x) := \sum_{\xi \in \mathbb{Z}/|I|} \Psi_\delta\big( \frac{\mathcal{F}_I {g}(\xi)}{|I|} \big) e(\xi x);
\end{align}
note that while there is no $\mathbf{1}_I$ cutoff, this makes no difference to the definition of $A_I(f,g_{\delta,I})$.

Note that whenever $|I|^{1/2} \leq \delta^{1/500-1}$,
\begin{align}
    \| A_I^\delta(f,g_{\delta,I}) \|_\infty \lesssim |I|^{1/2} \delta \lesssim \delta^{1/500},
\end{align}
so we will always assume that our intervals are sufficiently large, namely
\[ |I| \geq \delta^{1/200-2}.\]


Below, we define
\begin{align}
    \Omega_{\delta,I}(\beta) := \sum_{\xi \in \mathbb{Z}/|I| } (\mathcal{F}_{\mathbb{Z}}{\phi_I})(\beta-\xi)\Psi_\delta\big( \frac{\mathcal{F}_Ig(\xi)}{|I|} \big)
\end{align}
and
\begin{align}
    \Omega_{\delta,I}^{(j)}(\beta) := 2^{-A_0j} \sum_{\xi \in \mathbb{Z}/|I| } (\mathcal{F}_{\mathbb{Z}}{\phi_I^{(j)}})(\beta-\xi)\Psi_\delta\big( \frac{\mathcal{F}_Ig(\xi)}{|I|} \big)
\end{align}
where 

\begin{align}
    \sum_{j} 2^{-A_0 j} \phi_I^{(j)} = \phi_I
\end{align}
with $\{ \phi_I^{(j)} \}$ here defined to be $L^1$-normalized bump functions localized to $I$, with
\begin{align}
    \mathcal{F}_{\mathbb{R}}\phi_I^{(j)}(\xi),
\end{align}
supported in $\{ |\xi| \leq 2^j |I|^{-1}\}$; the constant $A_0$ depends only on e.g.\ the $O(A_0)$th Schwartz semi-norm of $\phi$, and we will be free to adjust it upwards as we wish. We emphasize that $\mathcal{F}_{\mathbb{Z}}^{-1} \Omega_{\delta,I}$ is supported in $I$, but the remaining inverse Fourier transforms are \emph{not} compactly supported. 

Then, implicitly restricting to 
\[ x \in I \cap \overline{\mathcal{D}},\]
we may express
\begin{align}
    A_I(f,g_{\delta,I})(x)   &= 
    \int (\mathcal{F}_{\mathbb{Z}}( f\cdot \mathbf{1}_{3I}))(\beta) \Omega_{\delta,I}(\beta) e(2 \beta x) \\
    & = \sum_{j} \int (\mathcal{F}_{\mathbb{Z}}(f \cdot \mathbf{1}_{3I}))(\beta) \Omega_{\delta,I}^{(j)}(\beta) e(2 \beta x) \ d\beta,
\end{align}
and we focus on the contribution of each $\Omega_{\delta,I}^{(j)}$ individually.

Note that 
\begin{align}
    |\Omega_{\delta,I}^{(j)}| \lesssim 2^{-A_0j} \delta
\end{align}
and
\begin{align}
    \| \Omega_{\delta,I}^{(j)} \|_{\ell^2} \lesssim 2^{-A_0 j} |I|^{-1/2},
\end{align}
so we may restrict our sum to $j = o(\log(1/\delta))$. In particular,
\begin{align}
   \{ \Omega_{\delta,I}^{(j)} \}
\end{align}
will all be supported in the $\delta^{-\epsilon}|I|^{-1}$ neighborhood of $\text{Spec}_\delta(I)$ inside $\mathbb{Z}/|I|$; call this set
\begin{align}
    \text{Spec}_\delta'(I),
\end{align}

We will implement the tree/branch selection argument, with 
\begin{align}
    \Sigma(I) := \text{Spec}_\delta'(I)
\end{align}
replacing the slightly smaller sets.

And, regarding $\delta > 0$ as fixed below, express
\begin{align}
    A_{I}^{(j)}(f,g)(x) := \int \mathcal{F}_{\mathbb{Z}} f (\beta) \Omega_{\delta,I}^{(j)}(\beta) e(2 \beta x) \ d\beta.
\end{align}

\subsection{Localizing to Branches}

We pass to branches as above; the contribution from the boundary branches is negligible, since the single scale contribution from the $j$th multiplier has 
\begin{align}
\| \int (\mathcal{F}_{\mathbb{Z}}(f \cdot \mathbf{1}_{L I}))(\beta) \Omega_{\delta,I}^{(j)}(\beta) e(2 \beta x) \ d\beta \|_2^2 \lesssim 2^{-2 A_0j} \delta^{2} L |I|;
\end{align}
we will of course only need the above when $L \leq \delta^{-1/10}$ (say).
\medskip

Once again, let $\Lambda \subset \mathbb{Z}/M_0$ and for $M \geq 2^{O(R)} M_0$, let $\varphi$ be a smooth function satisfying  
\[ \mathbf{1}_{[-1/4,1/4]} \leq \varphi \leq \mathbf{1}_{[-1/2,1/2]}, \]
and consider the (re-defined) Fourier multiplier
\begin{align}
    \Phi_{\Lambda,M}(\beta) := \sum_{\theta \in \Lambda} \big( \varphi(\frac{M}{2R}(\beta - \theta)) - \varphi(RM(\beta - \theta)) \big).
\end{align}
Notice that
\[ \Phi_{\Lambda,M} \cdot \mathbf{1}_{\mathbb{Z}/M} \]
is supported on
\[ \Lambda + B(2R/M) \smallsetminus \Lambda \subset \mathbb{Z}/M.\]

Then, if we here set
\begin{align}
    B_I^{(j)}(f,g)(x) := A_I^{(j)} ( \mathcal{F}_{\mathbb{Z}}^{-1}\Phi_{\Lambda,M}*f,g_{\delta,I})(x) \cdot \mathbf{1}_{I \cap \overline{\mathcal{D}}}(x),
\end{align}
we may express
\begin{align}
     B_I^{(j)}(f,g)(x) &= A_I^{(j)} ( (\mathcal{F}_{\mathbb{Z}}^{-1}\Phi_{\Lambda,M}*f) \cdot \mathbf{1}_{R^3I},g_{\delta,I})(x) \cdot \mathbf{1}_{I \cap \overline{\mathcal{D}}}(x) + O(2^{-A_0j} R^{-100}) \\
     & =: A_I^{(j)}(f_I,g_{\delta,I})(x) \cdot \mathbf{1}_{I \cap \overline{\mathcal{D}}}(x) + O(2^{-A_0j} R^{-100}).
\end{align}
By orthogonality in Fourier space, and the bounded overlap of $\{ C I : |I| = K \}$, we may estimate
\begin{align}
    \| (\sum_{I \in \mathcal{B}_s(I_v)} |f_I|^2)^{1/2} \|_{\ell^2(I_v)}^2 &= \sum_{N \leq 2^{-O(R)} |I_v|} \ \sum_{|I| = N, \ I \subset I_v } \| f_I \|_{\ell^2}^2 \\
    & \lesssim R^3 \sum_N \| \Phi_{\Lambda,N} \mathcal{F}_{\mathbb{Z}} f \|_{L^2(\mathbb{T})}^2 \\
    & \lesssim R^4 |I_v|,
\end{align}
since in evaluating the above, we may assume that $f$ is supported on $3I_v$ since our intervals in $\mathcal{B}_s(I_v)$ are so small. We may therefore estimate:
\begin{align}
    |\{ x \in I_v : \sup_{I \in \mathcal{B}_s(I_v)} |B_{I}^{(j)}(f,g)(x)| \gg t j^{-2}\}| \lesssim j^4 R^5 2^{-2A_0 j} \delta^2 |I_v| \ll 2^{-2A_0 j} \delta^{8/5} |I_v|.
\end{align}
Consequently, if we here define
\begin{align}
    m_{\Lambda,I}(\beta) := \sum_{\theta \in \Lambda} \varphi(R|I|(\beta - \theta))
\end{align}
with $\Lambda := \Lambda_{v,s}(I_v)$, and
\begin{align}
C_I^{(j)}(f,g)(x) := A_I^{(j)} ((\mathcal{F}_{\mathbb{Z}}^{-1} m_{\Lambda,I})*f,g_{\delta,I})(x) \mathbf{1}_{I \cap \overline{\mathcal{D}}}(x),
\end{align}
it suffices to show that the best constant
\begin{align}
    \mathbf{C}_{\text{Branch},j}
\end{align}
in the inequality
\begin{align}
    |\{ x \in I_v : \sup_{I \in \mathcal{B}_s(I_v)} |C_I^{(j)}(f,g)(x)| \gg t j^{-2} \}| \leq \mathbf{C}_{\text{Branch},j} |I_v|
\end{align}
satisfies
\begin{align}
    \mathbf{C}_{\text{Branch},j} \leq \delta^{1/6},
\end{align}
say.

We will again appeal to entropy methods; first, we tighten our Fourier projections.

\subsection{Tightening Fourier Projections}
For $x \in I \cap \overline{\mathcal{D}}$, express
\begin{align}
   &C_I^{(j)}(f,g)(x) \\
   & = \int (\mathcal{F}_{\mathbb{Z}}f)(\beta) e(\beta x) \sum_{\theta \in \Lambda} \varphi(R |I| (\beta - \theta)) \big( 2^{-A_0j} \sum_n \phi_I^{(j)}(n) g_{\delta,I}(n) e(\beta (x-n)) \big) \ d\beta \cdot \mathbf{1}_{I \cap \overline{\mathcal{D}}}(x)
\end{align}
and define
\begin{align}
    &C_I^{(j),\text{Loc}}(f,g)(x) \\
    &:= \int (\mathcal{F}_{\mathbb{Z}}f)(\beta) e(\beta x) \sum_{\theta \in \Lambda} \varphi(R |I| (\beta - \theta)) \big( 2^{-A_0j} \sum_n \phi_I^{(j)}(n) g_{\delta,I}(n) e(\theta (x-n)) \big) \ d\beta \cdot \mathbf{1}_{I \cap \overline{\mathcal{D}}}(x);
\end{align}
we claim that 
\begin{align}
    |C_I^{(j)} - C_I^{(j),\text{Loc}}| \ll 2^{-j} t^2 
\end{align}
pointwise.

To see this, it suffices to bound the $A(\mathbb{T})$ norm
of
\begin{align}
    \mathcal{E}_{\Lambda,I}(\beta) := \sum_{\theta \in \Lambda} \varphi(R |I| (\beta - \theta)) \Big( 2^{-A_0j} \sum_n \phi_I^{(j)}(n) g_{\delta,I}(n) \big( e(\theta (x-n)) - e(\beta (x-n)) \big) \Big)
\end{align}
by
\begin{align}
    \| \mathcal{E}_{\Lambda,I}(\beta) \|_{A(\mathbb{T})} \ll 2^{-j} t^2,
\end{align}
say.

By the bound of Lemma \ref{l:sob}, 
\begin{align}
    \| m \|_{A(\mathbb{T})} \lesssim \min_k \big( |m^{\vee}(k)| + \| m \|_{L^2(\mathbb{T})}^{1/2} \| \partial ( e(\cdot k) m ) \|_{L^2(\mathbb{T})}^{1/2} \big),
\end{align}
and the generous bound
\begin{align}
    \| \partial_\beta \mathcal{E}_{\Lambda,I}(\beta) \|_{L^2(\mathbb{T})} \ll 2^{-A_0/2 j} |I|^{1/2},
\end{align}
it suffices to prove that 
\begin{align}
    \| \mathcal{E}_{\Lambda,I}(\beta) \|_{L^2(\mathbb{T})} \ll t^{10} |I|^{-1/2}.
\end{align}

To do so, we set
\[ G_I(t) := G_I^{(j)}(t) := 2^{-A_0 j} \sum_{n} \phi_I^{(j)}(n) g_{\delta,I}(n) e(-(n-x)t), \]
so
\begin{align}
    |G_I(t)| \lesssim 2^{-A_0j};
\end{align}
we are interested in showing that 
\begin{align*}
    \sum_{\theta \in \Lambda_{v,s}(I_v)} \int_{|\beta - \theta| \leq R^{-1} |I|^{-1}} |G_I(\beta) - G_I(\theta)|^2 \ d\beta \ll t^{20} |I|^{-1}.
\end{align*}
By pointwise considerations, we may assume that $\phi_I^{(j)}$ is supported on $\delta^{-\rho} I$, and $\ell^1$-normalized, so that if $V_I$ is a degree $O(\delta^{-\rho} |I|)$ trigonometric polynomial, with 
\begin{align}\label{e:VJ} \mathbf{1}_{[-5\delta^{-\rho} |I|,5\delta^{-\rho} |I|]} \leq \mathcal{F}_{\mathbb{Z}}^{-1} V_I \leq \mathbf{1}_{[-10\delta^{-\rho} |I|,10\delta^{-\rho} |I|]} \end{align}
satisfying the natural derivative estimates up to a suitably high order,
then we can split
\begin{align*}
    |G_I(\beta) - G_I(\theta)| &\lesssim \| (\partial V_I) * G_I(t) \|_{L^{\infty}(\theta + B(1/R|I|))} \cdot \frac{1}{R|I|} \\
    & \qquad \leq \| (\partial V_I) * (G_I \cdot \mathbf{1}_{\theta + B(R^{1/2}/|I|)}) \|_{L^{\infty}(\theta + B(1/R|I|))} \cdot \frac{1}{R|I|} \\
    & \qquad \qquad + \| (\partial V_I) * (G_I \cdot \mathbf{1}_{(\theta + B(R^{1/2}/|I|) )^c}) \|_{L^{\infty}(\theta + B(1/R|I|))}  \cdot \frac{1}{R|I|};
\end{align*}
we will prove
\begin{align*}
    |G_I(\beta) - G_I(\theta)| \lesssim_A \delta^{-2\rho} \frac{|I|^{1/2}}{R}  \cdot \| G_I \cdot \mathbf{1}_{\theta + B(R^{1/2}/|I|)} \|_{L^2(\mathbb{T})} + R^{-A}.
\end{align*}
The local contribution is the main term, which we estimate
\begin{align*}
   & \| (\partial V_I) * (G_I \cdot \mathbf{1}_{\theta + B(R^{1/2}/|I|)}) \|_{L^{\infty}(\theta + B(1/R|I|))}  \cdot \frac{1}{R|I|} \\
& \qquad \lesssim  \| \partial V_I \|_{L^2(\mathbb{T})} \cdot \| G_I \cdot \mathbf{1}_{\theta + B(R^{1/2}/|I|)} \|_{L^2(\mathbb{T})}  \cdot \frac{1}{R|I|} \\
    & \qquad \qquad \qquad \lesssim \delta^{-2\rho} \frac{|I|^{1/2}}{R}  \cdot \| G_I \cdot \mathbf{1}_{\theta + B(R^{1/2}/|I|)} \|_{L^2(\mathbb{T})}.
\end{align*}
For the global contribution, by \eqref{e:VJ} and reproducing, we may bound 
\[ |\partial^i V_I(\beta)| \lesssim_A \delta^{-2\rho} |I|^{1+i} \cdot ( 1 + \delta^{-\rho} |I| \cdot  \|\beta \| )^{-A}, \; \; \; i=0,1 \]
for sufficiently large $A$, and thus whenever $\beta \in \theta + B(1/R|I|)$, we may bound
\begin{align*}
    |\int \partial V_I(\beta - t) \cdot (G_I \cdot \mathbf{1}_{(\theta + B(R^{1/2}/|I|) )^c})(t) \ dt| &\lesssim \int_{(\theta + B(R^{1/2}/|I|) )^c} |\partial V_I(\beta - t)| \ dt  \\
    & \qquad \lesssim_A |I| \cdot R^{-A},
\end{align*}
so that
\[ \| (\partial V_I) * (G_I \cdot \mathbf{1}_{(\theta + B(R^{1/2}/|I|) )^c}) \|_{L^{\infty}(\theta + B(1/R|I|))} \cdot \frac{1}{R|I|} \lesssim_A R^{-A}.\]
The total contribution, is therefore
\begin{align*}
    &\sum_{\theta \in \Lambda} \int_{|\beta - \theta| \leq R^{-1} |I|^{-1}} |G_I(\beta) - G_I(\theta)|^2 \ d\beta \\
    & \lesssim \delta^{-4\rho} \frac{|I|}{R^2}  \cdot  \sum_{\theta \in \Lambda_{v,s}(I_v)} \int_{|\beta - \theta|\leq R^{-1} |I|^{-1}} \| G_I \|_{L^2(\theta + B(R^{1/2}/|I|))}^2 \ d\beta  \\
    & \qquad + \sum_{\theta \in\Lambda_{v,s}(I_v)} \int_{|\beta - \theta| \leq R^{-1} |I|^{-1}} O_A(R^{-A}) \ d\beta\\
    & \lesssim \delta^{-4 \rho} \frac{1}{R^3} \cdot \sum_{\theta \in \Lambda} \| G_I \|_{L^2(\theta + B(R^{1/2}/|I|))}^2 +O_A(R^{-A}  \cdot |I|^{-1})  \\
    & \lesssim R^{-5/2} \cdot \| G_I \|_{L^2(\mathbb{T})}^2 +O_A(R^{-A} \cdot |I|^{-1}) \\
    &  \lesssim R^{-5/2} \cdot |I|^{-1}
\end{align*}
since 
\[ \min_{\theta \neq \theta' \in \Lambda} |\theta - \theta'| \gg 2^R  \cdot |I|^{-1} \]
for $I \in \mathcal{B}_s(I_v)$.

Thus, it suffices to redefine $\mathbf{C}_{\text{Branch},j}$ to be the best constant in
\begin{align}
    |\{ x \in I_v : \sup_{I \in \mathcal{B}_s(I_v)} |C_I^{(j),\text{Loc}}(f,g)(x)| \gg t j^{-2} \}| \leq \mathbf{C}_{\text{Branch},j} |I_v|;
\end{align}
note the representation
\begin{align}
    C_I^{(j),\text{Loc}}(f,g)(x) := \sum_{\theta \in \Lambda} e(2 \theta x) \varphi_{R|I|} * f_\theta(x) \big( 2^{-A_0 j} \sum_n \phi_I^{(j)}(n) g_{\delta,I}(n) e(-n \theta) \big)
\end{align}
where
\begin{align}
    \mathcal{F}_{\mathbb{Z}}f_\theta(\beta) := \varphi(2^R M_0 \beta) \mathcal{F}_{\mathbb{Z}} f(\beta + \theta),
\end{align}
where $\Lambda \subset \mathbb{Z}/M_0$, and $\varphi$ is Schwartz with Fourier transform smoothly approximating $\mathbf{1}_{|\xi| \leq 1/2}$.

\medskip

We can now quickly conclude our argument; we proceed similarly to \S \ref{ss:entropymethod}, though our current task is much less involved!


\subsection{Entropy, Again}
We repeat the entropy strategy in a simplified setting without arithmetic weights:

\medskip

For $I \in \mathcal{B}_s(I_v)$, let
\[ A_{I,j}(x) := 2^{-A_0 j} \sum_{\theta \in \Lambda} e(2 x \theta) \cdot \sum_n \phi_I^{(j)}(n) (\text{Mod}_{-\theta} f)(2x-n) \cdot     \Psi_\delta( \mathbb{E}_{I} \text{Mod}_{-\theta} g) \mathbf{1}_I(x). \]
We prove that
\begin{align}
    |\{ x \in I_v : \sup_{I \in \mathcal{B}_s(I_v)} |A_{I,j}(x)| \gtrsim j^{-2} t \}| \lesssim \delta^{1/6} |I_v|.
\end{align}

To do so, we work locally on intervals $P \subset I_v$, where
\[ |P|^{-1} = 2^{-R} \cdot \min_{\theta \neq \theta' \in \Lambda} |\theta - \theta'|, \]
so that all $I \in \mathcal{B}_s(I_v)$ satisfy $|I| \geq 2^{O(R)} |P|$.

For each $P$, and $x_P \in P$, define
\begin{align}
    A_{I,j,x_P}(x) := 2^{-A_0j} \sum_{\theta \in \Lambda} e(2 x \theta) \cdot \sum_n \phi_I^{(j)}(n) (\text{Mod}_{-\theta} f)(2x_P-n) \cdot     \Psi_\delta( \mathbb{E}_{I} \text{Mod}_{-\theta} g)  \mathbf{1}_I(x).
\end{align}
Note that
\begin{align}
|A_{I,j,x_P}(x) - A_{I,j}(x)| &\leq 2^{-A_0 j} \delta \sum_{\theta \in \Lambda} \frac{1}{|I|} \sum_{(x_P - 2^j I) \triangle (x - 2^j I)} |\text{Mod}_{-\theta}f(m)| \\
&\lesssim 2^{-O(R)} 2^{-A_0/2 j} (\frac{|P|}{|I|})^{1/2},
\end{align}
so that
\begin{align}
    \sup_{I \in \mathcal{B}_s(I_v)} |A_{I,j,x_P}(x) - A_{I,j}(x)| \lesssim 2^{-O(R)} 2^{-j}
\end{align}
and we can thus work with $A_{I,j,x_P}$. We will select a particular $x_P$ later in the argument.

So, fix some $P$ and estimate
\begin{align}
    \| \sup_{I \supset P} |A_{I,j,x_P}| \|_{\ell^2(P)}^2.
\end{align}

With 
\[ \mathcal{B}_{\epsilon_0}(P) := \mathcal{B}_{\epsilon_0,j}(P)\] an $\epsilon_0$-net for
\[ \{ \big( \sum_n \phi_I^{(j)}(n) \text{Mod}_{-\theta} f(2x_P - n) \cdot     \Psi_\delta( \mathbb{E}_{I} \text{Mod}_{-\theta} g) \big)_{\theta \in \Lambda} : P \subset I \in \mathcal{B}_s(I_v) \} \]
with respect to $\ell^1(\Lambda)$, bound 
\begin{align}
    \sup_{I \supset P} |A_{I,j,x_P}| \lesssim \epsilon_0 + ( \sum_{I \in \mathcal{B}_{\epsilon_0}(P)} |A_{I,j,x_P}|^2 )^{1/2},
\end{align}
where we identify elements of $\mathcal{B}_{\epsilon_0}(P)$ according to the interval that indexes them, and note that
\begin{itemize}
    \item $\| A_{I,j,x_P} \|_{\ell^2(P)}^2 \lesssim 2^{-2 A_0j} \delta^{2-1/500} |P|$; and
    \item $|\mathcal{B}_{\epsilon_0}(P)| \leq |\mathcal{F}_{{\epsilon_0}/10}(P)| \cdot |\mathcal{G}_{{\epsilon_0}/10}(P)|$ where $|\mathcal{F}_{{\epsilon_0}/10}(P)|$ is an ${\epsilon_0}/10$ net inside
    \[ \{ \big( \sum_n \phi_I^{(j)}(n) \text{Mod}_{-\theta} f(2x_P - n) \big)_{\theta \in \Lambda} : P \subset I \in \mathcal{B}_s(I_v) \} \]
    with respect to $\ell^2(\Lambda)$, and 
    $|\mathcal{G}_{{\epsilon_0}/10}(P)|$ is an ${\epsilon_0}/10$ net inside
    \[ \{ \big(      \mathbb{E}_{I} \text{Mod}_{-\theta} g )_{\theta \in \Lambda} : P \subset I \in \mathcal{B}_s(I_v) \}, \]
    also with respect to $\ell^2(\Lambda)$.
\end{itemize}
So, we can bound
\begin{align}
    \| \sup_{I \supset P} |A_{I,j,x_P}| \|_{\ell^2(P)}^2 \lesssim \epsilon_0^2 |P| + \delta^{2-1/500} \cdot \epsilon_0^{-4} \cdot ( \epsilon_0^2 |\mathcal{F}_{{\epsilon_0}/10}(P)|) \cdot (\epsilon_0^2 \cdot |\mathcal{G}_{{\epsilon_0}/10}(P)|) \cdot |P|.
\end{align}
By choosing $x_P$ to minimize the quantity $|\mathcal{F}_{{\epsilon_0}/10}(P)|$, and using vector-valued jump-counting inequalities, in particular Corollary \ref{c:sep}, we can bound
\begin{align}
    \sum_{P \subset I_v} \epsilon_0^2 |\mathcal{F}_{{\epsilon_0}/10}(P)| \lesssim |I_v|,
\end{align}
since we can assume that $f$ is supported on $3I_v$,
and similarly for $\mathcal{G}_{{\epsilon_0}/10}(P)$, by Lemma \ref{l:LEP}.


So, if we let
\begin{align}
    X := \bigcup_{P : \epsilon_0^2 |\mathcal{F}_{{\epsilon_0}/10}(P)| \geq \epsilon_0^{-2}} P \cup \bigcup_{P : \epsilon_0^2 |\mathcal{G}_{{\epsilon_0}/10}(P)| \geq \epsilon_0^{-2}} P,
\end{align}
with the choice of $x_P$ determined as above, then
\[ |X| \lesssim \epsilon_0^2 |I_v|.\]
So, we bound
\begin{align}
    &|\{ x \in I_v : \sup_{I \in \mathcal{B}_s(I_v)} |A_{I,j}(x)| \gtrsim t j^{-2} \}| \\
    & \lesssim |\{x \in  I_v : \sup_{I \in \mathcal{B}_s(I_v)} |A_{I,j}(x)| \gtrsim \delta^{1/600} \}| \\
    &\lesssim |\{ x \in I_v : \sup_{I \in \mathcal{B}_s(I_v)} |A_{I,j,x_P}(x)| \gtrsim \delta^{1/600} \}| \\
    &  \lesssim |X| + \delta^{-1/300} \sum_{P \not \subset X} \| \sup_{I \supset P} |A_{I,j,x_P}| \|_{\ell^2(P)}^2 \\
    &  \lesssim \epsilon_0^2 |I_v| + \delta^{-1/300} \sum_{P \not \subset X} \epsilon_0^2 |P| + \delta^{2-1/500} \epsilon_0^{-8}|P| \\
    &  \lesssim \delta^{-1/300} \epsilon_0^2 |I_v| + \delta^{2-1/500} \epsilon_0^{-8} |I_v| \lesssim \delta^{1/6}|I_v|
    \end{align}
after setting ${\epsilon_0}:= \delta^{1/10}$.

\bigskip

This (at last!) concludes the proof of Theorem \ref{t:main}.

\end{document}